\documentclass[12pt,oneside,reqno]{amsart}
\usepackage{amsmath,amssymb, amsthm,mathtools}
\usepackage[margin=1.5in]{geometry}
\usepackage[parfill]{parskip}
\usepackage{enumitem}
\usepackage{tikz}
\usepackage{tikz-cd} 
\usepackage{dsfont}
\usepackage{mathrsfs}
\usepackage{bbm}
\usepackage{dirtytalk}
\usepackage{bm}
\usepackage{microtype}
\usepackage{caption}
\usepackage{subcaption}
\usepackage{array}
\usepackage{longtable}

\DeclarePairedDelimiter{\floor}{\lfloor}{\rfloor}
\providecommand{\N}{\mathbb{N}}
\providecommand{\F}{\mathbb{F}}
\providecommand{\R}{\mathbb{R}}
\providecommand{\C}{\mathbb{C}}

\providecommand{\Z}{\mathbb{Z}}
\providecommand{\Q}{\mathbb{Q}}

\renewcommand{\vec}[1]{\boldsymbol{#1}}

\newcommand{\paren}[1]{\left( #1 \right)}
\newcommand{\brac}[1]{\left[ #1 \right]}

\newcommand{\abs}[1]{\left\vert#1\right\vert}

\newcommand{\set}[1]{\left\{#1\right\}}

\DeclareMathOperator{\rank}{rank}
\DeclareMathOperator{\im}{im}

\newtheorem{Theorem}{Theorem}
\newtheorem{Lemma}[Theorem]{Lemma}
\newtheorem{Definition}[Theorem]{Definition}
\newtheorem{Proposition}[Theorem]{Proposition}

\newtheorem{Corollary}[Theorem]{Corollary}

\newcounter{cnstcnt}
\newcommand{\newconstant}{%
\refstepcounter{cnstcnt}%
\ensuremath{c_{\thecnstcnt}}}
\newcommand{\oldconstant}[1]{\ensuremath{c_{\ref{#1}}}}

\begin{document}
\title[Potts Lattice Gauge Theory in Codimension Two]{A Sharp Deconfinement Transition for Potts Lattice Gauge Theory in Codimension Two}

\author{Paul Duncan}
\email{paul.duncan@mail.huji.ac.il}
\address{Einstein Institute of Mathematics, Hebrew University of Jerusalem, Jerusalem 91904, Israel}
\author{Benjamin Schweinhart}
\email{bschwei@gmu.edu}
\address{Department of Mathematical Sciences, George Mason University, Fairfax, VA 22030, USA}

\begin{abstract}
In 1983, Aizenman, Chayes, Chayes, Fr\"ohlich, and Russo~\cite{aizenman1983sharp} proved that $2$-dimensional Bernoulli plaquette percolation in $\Z^3$ exhibits a sharp phase transition for the event that a large rectangular loop is ``bounded by a surface of plaquettes.'' We extend this result both to $(d-1)$-dimensional plaquette percolation in $\Z^d,$ and to a dependent model of plaquette percolation called the plaquette random-cluster model. As a consequence, we obtain a sharp phase transition for Wilson loop expectations in $(d-2)$-dimensional $q$-state Potts hyperlattice gauge theory on $\Z^d$ dual to that of the Potts model.  Our proof is unconditional for Ising lattice gauge theory, but relies on a regularity conjecture for the random-cluster model in slabs when $q>2.$ We also further develop the general theory of the $i$-plaquette random cluster model and its relationship with $(i-1)$-dimensional Potts lattice gauge theory.
\end{abstract}

\maketitle

\section{Introduction}\label{sec:intro}

We study models of plaquette percolation on the cubical complex $\Z^d.$ The $2$-dimensional Bernoulli plaquette percolation on $\Z^3$ was defined by Aizenman, Chayes, Chayes, Fr\"ohlich, and Russo~\cite{aizenman1983sharp} to be the random cubical complex that includes all vertices and edges of $\Z^d$ and adds each square (two-dimensional) plaquette independently with probability $p.$ Their main theorem concerns the event that a rectangular loop $\gamma$ is ``bounded by a surface of plaquettes,'' denoted $V_{\gamma}.$  We will see later that there is some subtlety in defining this event. 

\begin{Theorem}[Aizenman, Chayes, Chayes, Fr\"ohlich, Russo~\cite{aizenman1983sharp}]
\label{thm:accfr}
For $2$-dimensional Bernoulli plaquette percolation on $\Z^3$ there are constants $0<\newconstant\label{const:1}(p), \newconstant\label{const:2}(p) <\infty$ so that 
\begin{align*}
  -\frac{\log\paren{\mathbb{P}_p(V_\gamma)}}{\mathrm{Area}(\gamma) } \rightarrow &   \oldconstant{const:1}(p) \qquad && p < 1-p_c(\mathbb{Z}^3)\\
  -\frac{\log\paren{\mathbb{P}_p(V_\gamma)}}{ \mathrm{Per}(\gamma)} \rightarrow &   \oldconstant{const:2}(p) \qquad && p > 1-p_c(\mathbb{Z}^3)
   \,,
\end{align*}

for rectangular loops $\gamma,$ as both dimensions of $\gamma$ are taken to $\infty.$  
   \end{Theorem}
As a historical note, this theorem originally relied on a conjecture about the continuity of the critical probability of percolation in slabs that was later proven by Grimmett and Marstrand~\cite{grimmett1990supercritical}. \cite{aizenman1983sharp} also proved partial results in higher dimensions. The $i$-dimensional Bernoulli plaquette percolation on $\Z^d$ is defined analogously; it is the random $i$-complex that adds each $i$-dimensional plaquette independently with probability $p$ to the $(i-1)$-skeleton consisting of all lower-dimensional plaquettes. In~\cite{aizenman1983sharp}, it is demonstrated that there are ``area law'' and ``perimeter law'' regimes at sufficiently extreme values of $p.$ We show that the transition between these two regimes is sharp when $i=d-1.$   

Theorem~\ref{thm:accfr} was motivated by an analogy with Wilson loops in lattice gauge theory: ``It turns out that, at least for the abelian $\Z(2)$ gauge model such a transition [for Wilson loop variables] can be traced exactly to a geometric effect of the type discussed here, albeit in a system of interacting plaquettes~\cite{aizenman1983sharp}\,.'' While efforts towards this end were stymied by the discovery of so-called topological anomalies~\cite{aizenman1984topological} in $q$-state Potts lattice gauge theory (which we discuss below), we showed that the plaquette random-cluster model exhibits the postulated relationship between Wilson loop variables and an event of the form $V_\gamma$ in~\cite{duncan2022topological} for prime $q.$ Here, we extend this to general $q$ and leverage it to prove a sharp phase transition for Wilson loop variables directly analogous to Theorem~\ref{thm:accfr}. As in~\cite{aizenman1983sharp}, our proof uses a dual characterization of the event $V_\gamma$ to reduce the theorem to statements concerning the one-dimensional random cluster model.

Lattice gauge theories are a family of models studied in physics as discretizations of Yang--Mills Theory. They were introduced by Wilson~\cite{wilson1974confinement}, with the special case of Ising lattice gauge theory being defined earlier by Wegner~\cite{wegner1971duality}. Lattice gauge theories on $\Z^d$ assign random spins from a complex matrix group $G$ to the edges of that lattice. When $d=4$ and $G$ is taken to be one of the compact Lie group $U(1), SU(2),$ or $SU(3),$ these systems model the fundamental forces of the standard model of particle physics, and a detailed understanding of them would resolve some of the most important open questions in mathematical physics~\cite{chatterjee2016yang}. However, even the behavior of simpler lattice gauge theories remains poorly understood from the perspective of rigorous mathematics. 

We specialize to the cases $G=\Z(2)$ or $\Z(3)$ is the group of second or third complex roots of unity, and to a separate family of models that they fit into. This family, called $(d-2)$-dimensional $q$-state Potts (hyper)lattice gauge theory, was defined by Kogut et al.~\cite{kogut1980z}. While these models may not be directly physically relevant, they are well-studied in the physics literature (see e.g. ~\cite{mack1979comparison, creutz1979experiments, koide2022non}). In addition, the mathematical properties of Ising lattice gauge theory and other lattice gauge theories with finite abelian gauge groups have been of recently renewed interest~\cite{cao2020wilson, chatterjee2020wilson, forsstrom2020wilson, forsstrom2023free}. 

The most important random variables in lattice gauge theory are the Wilson loop variables. Roughly speaking, they measure the product of spins on the edges of a loop $\gamma.$  When $G=SU(3),$ the asymptotics of Wilson loop expectations for rectangular loops $\gamma$ are thought to be related to the phenomenon of quark confinement, hence the terminology ``deconfinement transition''. In particular, they are conjectured to follow an ``area law'' and decay asymptotically in the area of $\gamma$ as its dimensions are taken to $\infty$ for any value of $\beta$. 

Different asymptotics are conjectured for $q$-state Potts lattice gauge theory on $\Z^d$ and --- more generally --- for $k$-dimensional Potts lattice gauge theories which assign spins to $k$-dimensional cells of a cell complex. In particular, it is thought that there is a critical threshold $\beta_c=\beta_c\paren{q,k,d}$ so that Wilson loop expectations for the boundary of a $(k+1)$-dimensional box follow an area law when $\beta<\beta_c$ and decay exponentially in the volume of the box but exhibit a ``perimeter law'' when $\beta>\beta_c$ and decay exponentially in its surface area. The special case of $k=0$ is the sharpness of the phase transition for the classical $q$-state Potts model; the ``area'' of a $1$-dimensional box is the distance between its endpoints and its ``perimeter'' is a constant. This was proven by Aizenman, Barsky, and Fern\'{a}ndez~\cite{aizenman1987phase} when $q=2$ and by Duminil-Copin, Raoufi, and Tassion~\cite{duminil2019sharp} in general. 

Two previous results are known for $k>1.$ Laanait, Messager, and Ruiz~\cite{laanait1989discontinuity} demonstrated the conjecture for sufficiently large $q$ when $k=1$ and $d=4$. In addition, Bricmont, Lebowitz, and Pfister~\cite{bricmont1980surface} proved that the Wilson loop tension (that is, the coefficient of area law decay for Wilson loop variables) for $1$-dimensional Ising lattice gauge theory on $\Z^3$ equals the surface tension of the dual Ising model. As we describe below, this quantity is defined in terms of the exponential decay of a different probability. This, combined with the later theorem of Lebowitz and Pfister~\cite{lebowitz1981surface} on the non-vanishing of the surface tension at criticality demonstrate that Ising lattice gauge theory exhibits area law behavior precisely when the dual Ising model is not subcritical.  Using the plaquette random-cluster model construction, we extend the result of Bricmont, Lebowitz, and Pfister by showing that the Wilson loop tension of $(d-2)$-dimensional $q$-state Potts lattice gauge theory equals the surface tension of the dual random-cluster model. A theorem of Bodineau~\cite{bodineau2005slab} demonstrates that this surface tension is non-zero when the random-cluster model is supercritical in a slab. We also prove that a perimeter law holds when the dual random-cluster model is subcritical.

Our proof begins by reducing the conjecture to a question concerning the stochastic topology of a random cell complex. Stochastic topology is a relatively new field, which studies the topological invariants of random structures. Previous work in that area concentrated on generalizing classical results from random graph theory to higher dimensional cell complexes and on gaining a statistical understanding of noise and signal in the context of topological data analysis (see~\cite{kahle2014topology,bobrowski2018topology, bobrowski2022random} for an overview). Only a few recent papers have addressed connections with statistical physics and percolation theory~\cite{bobrowski2020homological,bobrowski2022homological,duncan2020homological,duncan2022topological,roa2018topological,sarnak2019topologies}. We hope that the current work spurs further interest in the intersection of these fields.

The idea of representing $1$-dimensional Potts lattice gauge theory with a $2$-dimensional cell complex dates back to soon after its introduction ~\cite{ginsparg1980large,maritan1982gauge}. These earlier attempts by physicists were limited by imprecise notions which counted degrees of freedom in terms of ``independent surfaces of plaquettes'' rather than homology, failing to account for the dependence of one-dimensional homology on the coefficient group. Aizenman and Fr\"ohlich discovered that there were ``topological anomalies'' in Potts lattice gauge theory and its Wilson loop variables~\cite{aizenman1984topological}. Specifically, they found that the weight assigned to $1$-cochains consistent with a given plaquette configuration $P$ is not always proportional to $q^{\rank H^1\paren{P;\;\Z}}.$ This is in contrast to the situation for the classical Potts model, which weights $0$-cochains consistent with a graph $P'$ proportionally to $q^{\rank H^0\paren{P';\;\Z}}$ for any value of $q$. In addition, Aizenman and Fr\"ohlich constructed examples of plaquette systems for which a discrete analogue of Stokes' Theorem  for Wilson loop variables fails. After these observations, this project seems to have become dormant. 

The plaquette random-cluster model was introduced as a higher dimensional generalization of the classical random-cluster model by Hiraoka and Shirai~\cite{hiraoka2016tutte}. They demonstrated that it can be coupled with $q$-state Potts lattice gauge theory when $q$ is a prime integer. This extends the well-known coupling of the random-cluster model with the Potts model, which, together with other graphical representations of spin models~\cite{duminil2016graphical}, have been powerful tools in statistical mechanics. In earlier work~\cite{duncan2022topological}, we proved that --- under this coupling --- the Wilson loop expectation for a cycle $\gamma$ equals the probability that $\gamma$ is null-homologous in the plaquette random-cluster model when coefficients are taken in $\Z_q$ (roughly speaking, that $\gamma$ is ``bounded by a surface of plaquettes''). That is, we showed that ``topological anomalies'' noticed in~\cite{aizenman1984topological} can be accounted for by weighting the plaquette random-cluster model by $q^{\rank H^i\paren{P;\;\Z_q}}$ rather than $q^{\rank H^i\paren{P;\;\Z}}.$ We suggested that the case of general $q$ could be handled by replacing $q^{\rank H^i\paren{P;\;\Z_q}}$ with $\abs{H_i\paren{P;\;\Z_q}}$ but deferred the proof to a later paper. Here, we show that this definition results in a coupling that fully accounts for the ``topological anomalies.''\footnote{Note that $\abs{H_i\paren{P;\;\Z_q}}=\abs{H^i\paren{P;\;\Z_q}}$; see Corollary~\ref{cor:UCTC} below.}

With the plaquette random-cluster model in hand, the study of the deconfinement transition becomes amenable to the arguments of~\cite{aizenman1983sharp}. In some places, the adaptation of the proofs is complicated both by the dependence between disjoint plaquette events and also because of the extension to higher dimensions. Our results demonstrate that higher-dimensional cellular representations of spin models possess some of the same power of graphical representations of spin models.

\section{Background and Main Results}\label{sec:background}

The $k$-dimensional Potts (hyper)lattice gauge theory assigns random spins in the abelian group $\Z_q$ to the $k$-cells of a cell complex $X$ in a way so that reversing the orientation of a cell multiplies the spin by $-1.$ Following the language of algebraic topology, we call such a spin assignment an $k$-cochain and denote the collection of them by $C^k\paren{X;\;\Z_q}.$ This collection has a natural structure as a $\Z_q$ module. For more detailed topological definitions, see Section~\ref{subsec:homology} in the appendix.

\begin{Definition}[\cite{wegner1971duality, kogut1980z}]
The $k$-dimensional $q$-state Potts lattice gauge theory (or PLGT) on a finite cubical complex $X$ is the measure on $C^k\paren{X;\;\Z_q}$ given by
\[\nu_{X,\beta,q,k}\paren{f} \coloneqq \frac{1}{\mathcal{Z}} e^{-\beta H\paren{f}}\]
where $\beta$ is a parameter called the inverse temperature, $\mathcal{Z} = \mathcal{Z}\paren{X,\beta,q,k}$ is a normalizing constant, and $H$ is the Hamiltonian
\begin{equation}
    \label{eq:hamiltonian_potts}
    H\paren{f}=-\sum_{\sigma}K\paren{\delta f\paren{\sigma},0}\,.
\end{equation}
\end{Definition}

Here, $\delta$ is the coboundary operator satisfying $\delta f\paren{\sigma}=f\paren{\partial \sigma}$ and $K$ is the Kronecker delta function. When $q=2$ or $q=3,$ these measures coincide with the $\Z(2)$ and $\Z(3)$ Euclidean lattice gauge theories mentioned above, up to a rescaling of the parameter $\beta.$ The special case $k=0$ is the classical $q$-state Potts model. When $k>1,$ these models are sometimes called ``hyperlattice'' gauge theories, but we refer to them as lattice gauge theories to keep our language simple. The case $k=1,d=3$ may be of greatest interest, but the same methods we use to prove our main theorem are equally applicable to more general case of $k=d-2,d\geq 3.$  In the same vein, we will use terminology best suited for the special case $k=1.$  In particular, we will refer to the number of $(k-1)$-plaquettes in the boundary of an $k$-dimensional box as its ``perimeter'' and the number of $k$-dimensional plaquettes in its interior as its ``area,'' even though ``surface area'' and ``volume'' might be more appropriate when $k>1.$ 

We also consider Potts lattice gauge theory on subset of the integer lattice $\Z^d$ with boundary conditions. For convenience, we restrict ourselves to boxes $r\subset \Z^d.$ We call the measure defined above Potts lattice gauge theory with free boundary conditions and denote it by  $\nu^{\mathbf{f}}_{r,\beta,q,k}.$ The other boundary conditions of interest in this paper specify that the cochain agrees with specified cocycle $\eta\in  Z^{k}\paren{\partial r;\;\Z_q}$ on $\partial r$ (a cocycle is a cochain $\eta$ so that $\delta \eta=0$). For example, we could choose $\eta$ to assign a constant element of $\Z_q$ to each $k$-face of $\partial r.$ It will turn out that expectations of gauge invariant quantities do not depend on the specific choice of $\eta,$ and agree with a different type of boundary conditions for Potts lattice gauge theory that we call wired boundary conditions. 

Fix $\eta\in  Z^{k}\paren{\partial r;\;\Z_q},$ let $\psi:C^{k}\paren{r;\;\Z_q} \to C^{k}\paren{\partial r;\;\Z_q}$ be the map which restricts a cochain on $r$ to one on $\partial r,$ and set $D_{\eta}\paren{r;\;\Z_q}=\psi^{-1}\paren{\eta}.$ 

\begin{Definition}
The $k$-dimensional $q$-state Potts lattice gauge theory (or PLGT) on a box $r\subset \Z^d$ with boundary conditions $\eta$ is the restriction of  $\nu_{r,\beta,q,k}$ to $D_{\eta}\paren{r;\;\Z_q}.$ That is, it is the Gibbs measure $\nu^{\eta}_{r,\beta,q,k}$ on  $D_{\eta}\paren{r;\;\Z_q}$ induced by the Hamiltonian~(\ref{eq:hamiltonian_potts}). Similarly, the PLGT on $r$ with wired boundary conditions is the restriction of $\nu_{r,\beta,q,k}$ to $\ker\delta\circ \psi.$  It is denoted by  $\nu^{\mathbf{w}}_{r,\beta,q,k}.$
\end{Definition}

\begin{Definition}
    Let $\gamma$ be an $k$-cycle in $Z_{k}\paren{X;\;\Z_q}$. In the PLGT, the Wilson loop variable $W_{\gamma}$ is \[W_{\gamma}\paren{f}=\paren{f\paren{\gamma}}^{\mathbb{C}}\,,\]
where the $\C$ superscript denotes that we are viewing the variable as a complex number by identifying the $\Z_q$ with the multiplicative group of complex $q$-th roots of unity. 
\end{Definition}

The asymptotics of Wilson loop variables for $1$-dimensional lattice gauge theories have been of great interest in both the mathematical and physical literatures. We mention only a few results which are relevant to the cases $G=\Z\paren{2}$ and $G=\Z\paren{3}$; see~\cite{chatterjee2016yang} for a more thorough account. Classically, series expansions arguments were employed to demonstrate the existence of area law and perimeter law regimes at sufficiently extreme temperatures~\cite{osterwalder1978gauge, seiler1982gauge}. Recently, these methods were used to produce a refined understanding of the low temperature asymptotics of lattice gauge theories with with finite gauge groups~\cite{chatterjee2020wilson,cao2020wilson,forsstrom2020wilson}. There has been less specific interest in the PLGT, except for the special case of $q=2$~\cite{chatterjee2020wilson,forsstrom2023free}.  The series expansion techniques should be easily adaptable to show the existence of area law and perimeter law regimes for the $k$-dimensional PLGT at sufficiently extreme temperatures. This result can also be proven using a coupling with the plaquette random-cluster model and a comparison with Bernoulli plaquette percolation~\cite{duncan2022topological} (that reference considers the case of prime $q,$ but the argument extends to the generalized plaquette random-cluster model introduced here). Finally, as mentioned above, \cite{laanait1989discontinuity} proved that the $1$-dimensional PLGT on $\Z_4$ exhibits a sharp deconfinement transition for sufficiently large $q$ and Bricmont, Lebowitz, and Pfister~\cite{bricmont1980surface} demonstrated that the Wilson loop tension in Ising lattice gauge theory on $\Z^3$ coincides with the surface tension of the dual Ising model.

Our main theorem characterizes the asymptotics of Wilson loop variables for the boundaries of $(d-1)$-dimensional boxes in $(d-2)$-dimensional Potts lattice gauge on $\Z^d$ at all but the critical value of $\beta.$ Before stating it, we introduce some notation. For now, we let  $\nu_{\Z^d,\beta,q,d-1}$ denote an infinite volume PLGT on $\Z^d.$ In Section~\ref{subsec:comparisoninfinite} we describe how to construct such measures as a weak limit of finite volume measures using both free, wired, and $\eta$ boundary conditions. Our result holds for any such limiting measure. 

Let $r$ be a an $i$-dimensional box in $\Z^d.$ That is, $r$ is a set of the form $\brac{0,N_1}\times\ldots \brac{0,N_i}\times \set{0}^{d-i}$ or one obtained from it by symmetries of the lattice. When $G$ is the additive (abelian) group of a ring with unity, we can identify $\partial r$ with the chain $\sum_{\sigma \in \partial r} \sigma \in C_{i-1}\paren{\Z^d;\;G},$ where the sum is taken over the (positively oriented) $(i-1)$-plaquettes of $\gamma.$ This is an abuse of notation as, strictly speaking, there is a difference between the set $\partial r$ and the chain $\partial r.$ 

To obtain our sharpest result for the perimeter law regime, we require a minor regularity hypothesis on the boxes considered.  For a box $r,$ let $m\paren{r}$ be its minimum dimension and let $M\paren{r}$ be its maximum dimension. We say that a family of $(d-1)$-dimensional boxes $r_{l}$ is suitable if its  $(d-1)$ dimensions diverge to $\infty$ and if $m\paren{r_{l}}=\omega\paren{log\paren{M\paren{n}}}.$ When $r_{l}$ is suitable, we say that $\gamma_l=\partial r_{l}$ is a \emph{suitable} family of rectangular boundaries. 

\begin{Theorem}
  \label{thm:sharpness}
Fix integers $q,d \geq 2$ and set $\nu=\nu_{\Z^d,\beta,q,d-1}.$ There exist constants $0<\newconstant\label{const:plgt1}(\beta,q), \newconstant\label{const:plgt2}(\beta,q)<\infty$ so that, if $\set{\gamma_l}$ is a suitable family of rectangular $(d-1)$-boundaries then
\begin{align*}
    -\frac{\log\paren{\mathbb{E}_{\nu}(W_{\gamma_l})}}{\mathrm{Area}(\gamma) } \rightarrow &   \oldconstant{const:plgt1}(\beta,q) \qquad && \beta < \beta^*\paren{\beta_{\mathrm{surf}}(q)}\\
   -\frac{\log\paren{\mathbb{E}_{\nu}(W_{\gamma_l})}}{ \mathrm{Per}(\gamma)} \rightarrow &   \oldconstant{const:plgt2}(\beta,q) \qquad && \beta > \beta^*\paren{\beta_{c}(q)}
      \,,
\end{align*}
where $\beta_{c}(q)$ is the critical inverse temperature for the Potts model on $\Z^d,$ $\beta_{\mathrm{surf}}(q)=-\log\paren{1-p_{\mathrm{surf}}\paren{q}}$ is the inverse temperature corresponding to the vanishing of the surface tension in the random-cluster model, and 
\[\beta^*\paren{\beta}=\log\paren{\frac{e^{\beta}+q-1}{e^{\beta}-1}}\,.\]
\end{Theorem}

We define $p_{\mathrm{surf}}$ below. As a consequence of a theorem of Bodineau~\cite{bodineau2005slab}, $p_{\mathrm{surf}}\paren{q}\leq p_{\mathrm{slab}}\paren{q},$ where $p_{\mathrm{slab}}\paren{q}$ is the slab threshold for the random-cluster model. The constant $\oldconstant{const:plgt1}(\beta,q)$ may a priori depend on the infinite volume measure $\nu,$ but $\nu$ is known to be unique in the high temperature regime, so $\oldconstant{const:plgt2}(\beta,q)$ does not. The only place we use the assumption that $\gamma_l$ is suitable is the proof of the existence of the sharp constant $ \oldconstant{const:plgt2}(\beta,q) $ in the perimeter law regime.  Note that the special case of $d=2$ is sharpness of the phase transition for the planar Potts model, a result due to Beffara and Duminil-Copin~\cite{beffara2012self}.

It is a conjecture of Pisztora that $\beta_{c}(q)=\beta_{\mathrm{slab}}(q)$ for all $q$~\cite{pisztora1996surface}. Also, Lebowitz and Pfister~\cite{lebowitz1981surface} proved that $\beta_{\mathrm{surf}}(2)=\beta_c(2)$ for all $d.$


We prove Theorem~\ref{thm:sharpness} using a cellular representation of the Potts lattice gauge theory. The plaquette random-cluster model (or PRCM) with coefficients in a field $\F$ was defined in~\cite{hiraoka2016tutte} to be the random $i$-dimensional subcomplex of a cell complex $X$ so that
\begin{equation}\label{eqn:fieldrcm}
    \mathbb{P}\paren{P}\propto p^{\abs{P}} \paren{1-p}^{\abs{X^{\paren{i}}}-\abs{P}} q^{\mathbf{b}_{i-1}\paren{X;\;\F}}
\end{equation}
where $\abs{X^{\paren{i}}}$ denotes the number of $i$-cells of $X,$ $\abs{P}$ denotes the number of $i$-cells of $P,$ and ${\mathbf{b}_{i-1}\paren{X;\;\F}}$ is the rank of the $(i-1)$-homology group $H_{i-1}\paren{P;\;\F}.$ 
\cite{hiraoka2016tutte} and \cite{duncan2022topological} show a number of results about these models, and in particular that they are coupled with the $(i-1)$-dimensional $q$-state PLGT when $q$ is a prime integer and $\F=\Z_q.$  Here, we extend the definition of the PRCM to produce a model coupled with the $q$-state PLGT even when $q$ is non-prime. This definition is equivalent to the one we suggested in~\cite{duncan2022topological} and deferred to later study.

\begin{Definition}
Let $X$ be a finite $d$-dimensional cell complex, $i<d$ and $p\in\brac{0,1}.$ The $i$-dimensional plaquette random-cluster model on $X$ with coefficients in a finite abelian group group $G$ is the random $i$-complex $P$ that includes the $(i-1)$-skeleton of $X$ and is distributed as follows. 
\begin{align}
\label{eqn:rcmdefinition}
    \tilde{\mu}_{X,p,G,i}\paren{P} \coloneqq \frac{1}{Z}p^{\abs{P}}\paren{1-p}^{\abs{X^{\paren{i}}} - \abs{P}}\abs{H^{i-1}\paren{P;\;G}}\,,
\end{align}
where $Z=Z\paren{X,p,G,i}$ is a normalizing constant and $H^{i-1}\paren{P;\;G}$ is the reduced cohomology of $P$ with coefficients in $G.$ 
\end{Definition}

\begin{figure}[t]
    \centering
    \includegraphics[width=\textwidth]{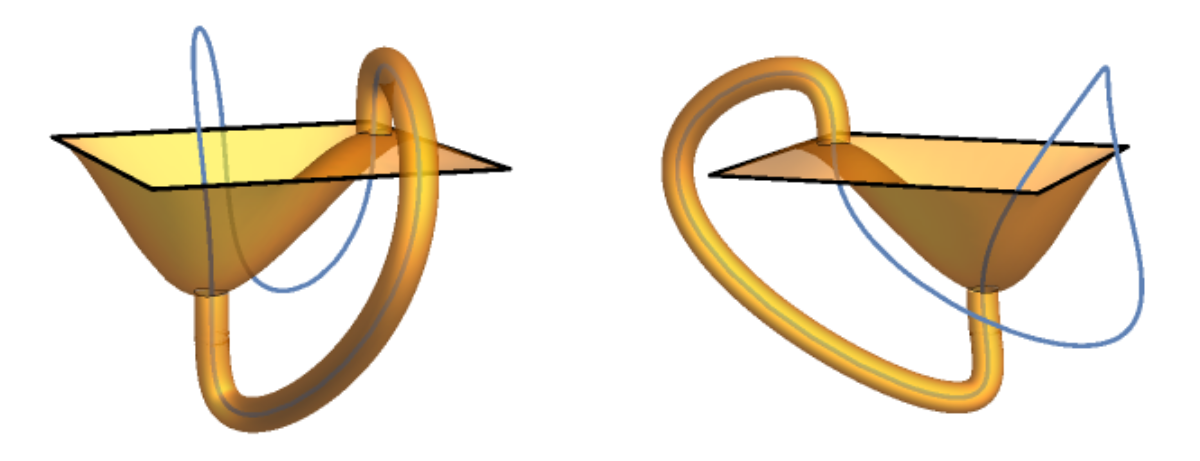}
    \caption{A non-orientable plaquette surface (shown in orange) realizing the event $V_{\gamma}\paren{2}$ but not the events $V_{\gamma}\paren{q}$ for odd $q.$ $\gamma$ is depicted by the thick black line. Note that the dual blue loop has linking number $2$ with $\gamma$; we will see later that $V_{\gamma}\paren{q}$ occurs if and only if there is a dual loop whose link number with $\gamma$ is not $0$ modulo $q.$ This figure was inspired by Figure 1 of~\cite{aizenman1983sharp} and was created in Mathematica using the CurveTube custom function included with the textbook~\cite{gray2006modern}.}
    \label{fig:vgamma2}
\end{figure}

When $X=r$ is a box in $\Z^d,$ we call this random cell complex the plaquette random-cluster model with free boundary conditions and denote it by $\tilde{\mu}^{\mathbf{f}}_{r,p,G,i}.$ By convention, we do not include the $i$-plaquettes in the boundary of $r,$ and instead write $\overline{r}$ for the full induced subcomplex. We also define the PRCM with wired boundary conditions by replacing the term $\abs{H^{i-1}\paren{P;\;G}}$ in (\ref{eqn:rcmdefinition}) with $\abs{H^{i-1}\paren{P\cup \partial r;\;G}}.$ This has same effect as adding all the $i$-plaquettes in the boundary of $r.$ This measure is denoted by  $\tilde{\mu}^{\mathbf{w}}_{r,p,G,i}\paren{P}.$

We prove that $\tilde{\mu}_{X,p,\Z_q,i}$ is coupled with the $(i-1)$-dimensional $q$-state PLGT (Proposition~\ref{prop:couple}) in a way so that a Wilson loop expectation equals the probability that the loop is null-homologous.  This generalizes one of the main theorems of~\cite{duncan2022topological}, which covers the case when $q$ is a prime integer. Until we specialize to the $(d-1)$-dimensional random-cluster model, we will follow the convention of~\cite{duncan2022topological} and reserve $i$ for the dimension of the random-cluster model.

We would like to relate the Wilson loop expectation $W_{\gamma}$ to the probability that $\gamma$ is ``bounded by a surface of plaquettes.'' To do so, we must account for the dependence on both $q$ and also on the boundary conditions. Set $V^{\mathrm{fin}}_{\gamma}\paren{G}$ be the event that $\gamma \in B_{i-1}\paren{P;\;G}$ and $V^{\mathrm{inf}}_{\gamma}\paren{G}$ to be the event that 
union of $V^{\mathrm{fin}}_{\gamma}\paren{G}$ and the event that $\gamma$ is ``homologous to infinity'' in the sense that it is homologous to a cycle in the boundary of the cube $\brac{-n,n}^d$ for arbitrarily large $n.$  When $G=\Z_q$ we simply write $V^{\mathrm{fin}}_{\gamma}\paren{q}$ and $V^{\mathrm{inf}}_{\gamma}\paren{q}$ for these events. To further simplify notation, we use $V^{\mathrm{fin}}_{\gamma}\paren{1}$ and $V^{\mathrm{inf}}_{\gamma}\paren{1}$ when $G=\Z.$ When $i=d-1,$ the event $\neg V_{\gamma}\paren{q}$ can be characterized in terms of the existence of dual loop whose linking number with $\gamma$ is non-zero modulo $q;$ this is Proposition~\ref{prop:linkq} below.

See Figures~\ref{fig:vgamma2} and~\ref{fig:vgamma}. 

\begin{Theorem}
\label{thm:comparison}
Let $0<i<d-1,$ let $\gamma$ be an $(i-1)$-cycle in $\Z^d,$ $q\in\N+1,$ and $\paren{\#_1,\#_2}\in \set{\paren{\mathbf{f},\mathrm{fin}}, \paren{\mathbf{w},\mathrm{inf}}}.$ 
Then
\[\mathbb{E}_{\nu}\paren{W_{\gamma}}=\tilde{\mu}\paren{V_{\gamma}}\,,\]
where $\nu=\nu_{\Z^d,\beta,q,i-1}^{\#_1}$ is the PLGT with boundary conditions $\#_1,$ $\tilde{\mu}=\tilde{\mu}^{\#_1}_{\Z^d,1-e^{-\beta},\Z_q,i}$ is the corresponding random-cluster model, and $V_{\gamma}=V^{\#_2}\paren{q}.$
\end{Theorem}


We note that the analogue of this theorem for finite cell complexes (Proposition~\ref{prop:comparisonbody} below) would fail if $V_{\gamma}\paren{q}$ is replaced with $V_{\gamma}\paren{1}.$  This was one of the ``topological anomalies'' documented by Aizenman and Fr\"ohlich; see the examples in Section 4.1 of~\cite{aizenman1984topological}. Their construction can be described as follows: Let $r$ be a rectangle in $\Z^3$ and let $\gamma = \partial r.$ Choose a tube $T$ of width one that intersects $r$ $k$ times, each time in a single plaquette $\sigma_j$. Let $P$ include all plaquettes in the exterior of $T.$ Then the chains $\brac{\partial\sigma_j}$ are homologous in $H_{1}\paren{P;\Z}$ and $\brac{\gamma}=k\brac{\sigma_1}\neq 0$ in that group. In particular, $V_{\gamma}\paren{1}$ occurs but $V_{\gamma}\paren{q}$ does not. In addition, conditional on $P,$ \[W_{\gamma} = \paren{W_{\sigma_1}}^k\,.\]
 Thus, if $q\mid k,$ $W_{\gamma}$ is uniformly distributed on a subgroup of $\Z\paren{q}$  and 
 \[\mathbb{E}\paren{W_{\gamma}\middle | P}=1 \neq I_{V_{\gamma}\paren{1}}\]
 but 
  \[\mathbb{E}\paren{W_{\gamma}\middle | P}=I_{V_{\gamma}\paren{q}}\,.\]

\begin{figure}[t]
    \centering
    \includegraphics[height=1.75in]{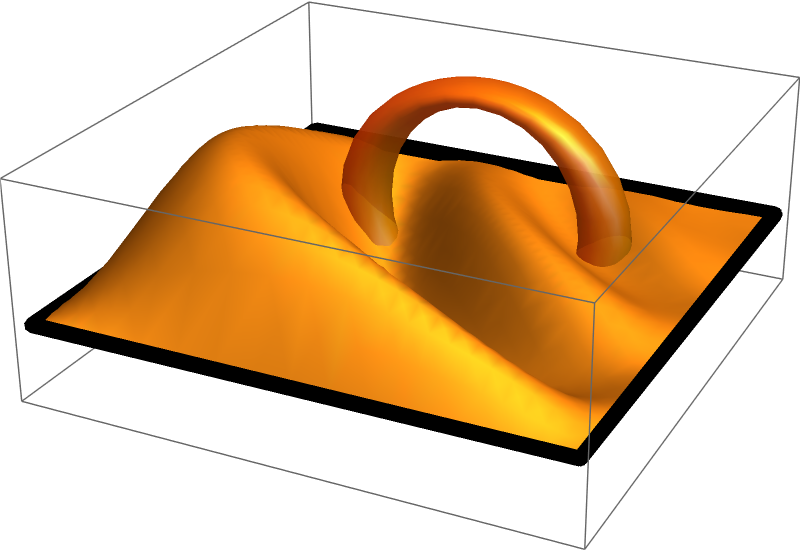}
    \qquad
    \includegraphics[height=1.75in]{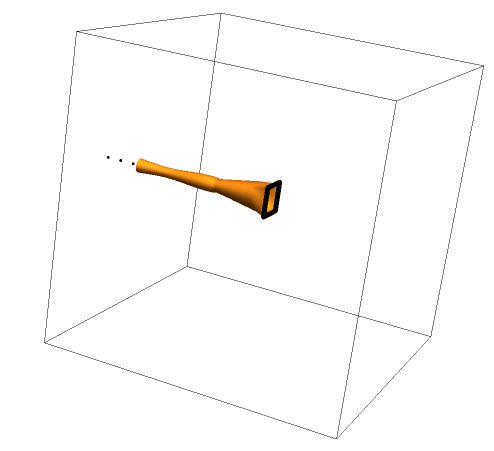}
    \caption{Illustrations of the events $V_{\gamma}^{\mathrm{fin}}$ and $V_{\gamma}^{\mathrm{inf}}.$ }
    \label{fig:vgamma}
\end{figure}

For the special case of the PRCM in codimension one ($i=d-1$), we have that 
\[H^{d-i-1}\paren{P;\;G}\cong G^{\mathrm{rank}\paren{H_{d-i-1}\paren{P;\;\Q}}}\]
so the law in (\ref{eqn:rcmdefinition}) takes the same form as in the classical random-cluster model (see Proposition~\ref{prop:codim1} below). Let $\mu_{X,p,q,i}$ be the PRCM with coefficients in $\Q$ as defined in (\ref{eqn:fieldrcm}). Then we can place the models $\tilde{\mu}_{X,p,G,d-1}$ inside the larger family of measures
$\mu_{X,p,q,i}$ where $q$ can take on any positive real value. In this case $\tilde{\mu}_{X,p,G,d-1}$ coincides with $\mu_{X,p,\abs{G},d-1}.$  

Let $r$ be a $d$-dimensional box in $\Z^d.$ By a small adaptation of Theorem 18 in~\cite{duncan2022topological},  $\mu_{r,p,q,d-1}$ is dual to the classical ($1$-dimensional) random-cluster model on a dual box with appropriate boundary conditions. A similar duality relation holds for general $i,$ but we do not require it here. We will include a proof in a separate manuscript.

Suppose $\gamma=\partial r$ and let $V_{\gamma}$ be the event $V_{\gamma}^{\mathrm{fin}}\paren{m}$ or $V_{\gamma}^{\mathrm{inf}}\paren{m}.$ Then the probability of the event $V_{\gamma}$ bounded below by the probability that all plaquettes contained in $r$ are included, which decays exponentially in the area of $\gamma.$ On the other hand, $V_{\gamma}$ is precluded if one of $\gamma$'s constituent $(d-2)$-plaquettes is not adjacent to a $(d-1)$-plaquette of $P$. It is not difficult to show that the probability that there are no such ``isolated'' $(d-2)$-cells decays exponentially in the perimeter of $\gamma.$ 



We now recall the definition of surface tension in the classical random-cluster model.
Let $\partial^+ \Lambda_N=\partial\Lambda_N\cap\set{\vec{e}_{d}>0}$ and $\partial^- \Lambda_N=\partial\Lambda_N\cap\set{\vec{e}_{d}<0}.$ 
\begin{Definition}
    The surface tension constant for the random-cluster model with parameter $p,q$ is defined by 
    $$\tau_{p,q}=\lim_{N\rightarrow\infty} \frac{-\log\paren{\mu^{\mathbf{w}}_{\Lambda_N,p,q,1}\paren{\partial^+\Lambda_{N}\xleftrightarrow[Q]{} \partial^-\Lambda_N}}}{\paren{2N}^{d-1}}\,.$$
\end{Definition}

We also write 
\[p_{\mathrm{surf}}\paren{q} \coloneqq \inf \sup\set{p : \tau_{p,q} > 0}\,\]
and 
\[\beta_{\mathrm{surf}}\paren{q} = -\log\paren{1-p_{\mathrm{surf}}\paren{q}}.\]

For the Ising model, this coincides with a notion of surface tension defined using mixed boundary conditions~\cite{bodineau1999wulff}. Bodineau~\cite{bodineau2005slab} studied a modified definition of surface tension which instead considers free boundary conditions and connections between the top and bottom faces instead of hemispheres. He proved that this modified quantity is non-zero precisely when the RCM is supercritical in a sufficiently thick slab. Since free boundary conditions and smaller target sets both make connections less likely, it follows that $p_{\mathrm{surf}} \leq p_{\mathrm{slab}}.$


Here, $\mu_{\Z^d,p,q,d-1}$ will denote any infinite volume random-cluster measure obtained as a weak limit of finite volume measures (see Section~\ref{sec:RCMboundary}).

\begin{Theorem}
\label{thm:sharpnessRCM}
  Let $\mu_{\Z^d,p} = \mu_{\Z^d,p,q,d-1}$
  where $p\in\brac{0,1}$ and $q\in[1,\infty).$ Also, and $V_{\gamma}=V_{\gamma}^{\mathrm{fin}}\paren{m}$ or $V_{\gamma}^{\mathrm{inf}}\paren{m}$ for some $m\in \N.$ Then 
  there exist constants and $0<\newconstant\label{const:prcm1}(p,q), \newconstant\label{const:prcm2}(p,q)<\infty$ so that 
  for any suitable family of rectangular $(d-1)$-boundaries $\set{\gamma_l},$ 
\begin{align*}
    -\frac{\log\paren{\mu_{\Z^d,p,q}\paren{V_{\gamma_l}}}}{\mathrm{Area}(\gamma_l)} \rightarrow &   \oldconstant{const:prcm1} \qquad &&  p < p^*\paren{p_{\mathrm{surf}}(q)}\\
    -\frac{\log\paren{\mu_{\Z^d,p,q}\paren{V_{\gamma_l}}}}{ \mathrm{Per}(\gamma_l)} \rightarrow &   \oldconstant{const:prcm2} \qquad && p > p^*\paren{p_{c}(q)}
      \,,
\end{align*}
where $p_c\paren{q}$ is the critical probability for the classical (one-dimensional) random-cluster model on $\Z^d$ and 
\[p^*=p^*(p)= \frac{\paren{1-p}q}{\paren{1-p}q + p}\,.\]
\end{Theorem}

The constant $\oldconstant{const:prcm1}$ may depend on the infinite volume measure $\mu_{\Z^d,p,q,d-1}$ or the choice of the event $V_{\gamma}$ but $\oldconstant{const:prcm2}$ does not depend on the infinite volume measure or on the choice of $V_{\gamma}^{\mathrm{inf}}$ or $V_{\gamma}^{\mathrm{fin}}$  (though it may depend on the coefficient group). By the preceding discussion, Theorem~\ref{thm:sharpness} follows by taking $q=m\in \N+1.$ Moreover, by the Grimmett--Marstrand Theorem~\cite{grimmett1990supercritical,barsky1991dynamic}, $p_{\mathrm{slab}}(1)=p_c\paren{\Z^d},$ so we have the the following generalization of Theorem~\ref{thm:accfr}.

\begin{Theorem}
\label{thm:accfrHigher}
For $(d-1)$-dimensional Bernoulli plaquette percolation on $\Z^d$ there are constants $0<\newconstant\label{const:3}(p), \newconstant\label{const:4}(p) <\infty$ so that 
\begin{align*}
  -\frac{\log\paren{\mathbb{P}_p(V_\gamma)}}{\mathrm{Area}(\gamma) } \rightarrow &   \oldconstant{const:3}(p) \qquad && p < 1-p_c(\mathbb{Z}^d)\\
  -\frac{\log\paren{\mathbb{P}_p(V_\gamma)}}{ \mathrm{Per}(\gamma)} \rightarrow &   \oldconstant{const:4}(p) \qquad && p > 1-p_c(\mathbb{Z}^d)
   \,,
\end{align*}

for rectangular $(d-1)$-boundaries $\gamma$ as all dimensions of $\gamma$ are taken to $\infty,$ where $V_{\gamma}=V_{\gamma}^{\mathrm{fin}}\paren{m}$ or $V_{\gamma}=V_{\gamma}^{\mathrm{inf}}\paren{m}$ for some $m\in\N.$ 
   \end{Theorem}
The hypothesis that the limit be taken over a suitable family of rectangular boundaries is not necessary here, as discussed at the beginning of Section~\ref{sec:sharpnessperimeter}. 

\section{Outline}

The paper is split into two parts. First, we explore properties of the generalized plaquette random-cluster model and its relationship with Potts lattice gauge theory. The tools of algebraic topology allow us to reduce the proof of Theorems~\ref{thm:sharpness} to questions concerning the classical one-dimensional random-cluster model (RCM). The second half of the paper resolves these questions, and involves arguments more typical of the literature on percolation theory. In addition, we include a review of the definitions of homology and cohomology in Appendix~\ref{sec:topology}.

We give a brief overview of our proof of the sharpness of the deconfinement transition for the PLGT (Theorem~\ref{thm:sharpness}). We show that it is equivalent to a special case of Theorem~\ref{thm:sharpnessRCM} for the PRCM by demonstrating that Wilson loop expectations can be computed in terms of the probability of the topological event $V_{\gamma}$ (Theorem~\ref{thm:comparison}). This is done in two steps: we begin by establishing the corresponding result for finite volume measures (Proposition~\ref{prop:comparisonbody} in Section~\ref{subsec:couplingfree}). Then, we construct an infinite volume coupling between the PRCM and PLGT with free boundary conditions and extend the result on Wilson loop variables  (Proposition~\ref{prop:comparisoninfinite} in Section~\ref{subsec:comparisoninfinite}). The corresponding statement for wired boundary conditions is Proposition~\ref{prop:comparisoninfinitewired} in  Appendix~\ref{subsec:couplingwired}. These arguments work for the general $i$-dimensional PRCM, but to continue we specialize to the case $i=d-1.$ Then, we characterize the event $V_{\gamma}\paren{q}$ in terms of the dual one-dimensional random-cluster model (RCM). Specifically, we show that  $V_{\gamma}\paren{q}$ occurs if and only if there is no dual loop which has non-zero linking number modulo $q$ with $\gamma$ (Proposition~\ref{prop:linkq} in Section~\ref{sec:dualityvgamma}). 

The remainder of the paper focuses on either constructing or precluding the existence of such dual loops. We demonstrate the two parts of Theorem~\ref{thm:sharpnessRCM} separately. In Section~\ref{sec:arealaw}, we show an area law upper bound in the subcritical regime by showing that the coefficient of area law decay equals the surface tension of the dual random-cluster model.  Section~\ref{sec:perimeterlaw} contains a first proof of a perimeter law upper bound in the supercritical regime, following the approach of~\cite{aizenman1983sharp}. We apply the almost sure finiteness of components in the dual subcritical RCM to construct a hypersurface of plaquettes with boundary $\gamma$ (Proposition~\ref{prop:perimeter}). While this is relatively straightforward, showing the existence of a sharp constant for the perimeter law requires more involved geometric arguments. We carry these out in Section~\ref{sec:sharpnessperimeter}, working with both the plaquette system and the dual bond system to build a hypersurface of plaquettes that prevents dual loops from linking with $\gamma.$ These constructions give matching upper and lower bounds for the perimeter law constant (Theorem~\ref{theorem:sharp_perimeter}, proven at the end of Section~\ref{subsec:sharp_perimeter}), concluding our proof of Theorem~\ref{thm:sharpnessRCM}.

\part{The Plaquette Random-Cluster Model with Coefficients in an Abelian Group}

As defined in the introduction, the PRCM random-cluster model with coefficients in an finite abelian group $G$ is the random $i$-dimensional subcomplex of $X$ so that
 \[\tilde{\mu}_{X,p,G,i}\paren{P} \propto p^{\abs{P}}\paren{1-p}^{\abs{X^{\paren{i}}} - \abs{P}}\abs{H^{i-1}\paren{P;\,G}}\,.\]
In Section~\ref{sec:subcomplexes}, we introduce the complexes and dual complexes in $\Z^d$ that we primarily work with in this article. 

Section~\ref{sec:coupling} returns to the general setting and studies the relationship between the PRCM with  coefficients in $\Z_q$ and the $q$-state PLGT. These results hold for any $i.$ We cover the case of finite complexes in Section~\ref{subsec:couplingfree}.
The definitions of the infinite volume measures constructed with free boundary conditions are given in Section~\ref{subsec:comparisoninfinite}. The details for similar results on finite and infinite complexes with wired boundary conditions can be found in ~\ref{sec:wiredinfinite}. 

Next, in Section~\ref{sec:codim1PRCM} we show that --- when $i=d-1$ ---  the PRCM with coefficients in $G$ coincides with the PRCM with coefficients in $\Q$ with parameter $q=\abs{G}.$ By earlier results of~\cite{hiraoka2016tutte} and~\cite{duncan2022topological}, it satisfies many nice properties, which we summarize below. It turns out these hold for general $i$ when $G=\Z_q,$ but we defer their proof until a later paper. We then discuss we discuss more general boundary conditions and infinite volume limits for the codimension one PRCM in Section~\ref{sec:RCMboundary}. While these are not difficult to define in general, we use duality with the classical random-cluster model to reduce technical overhead.

Finally, in Section~\ref{sec:dualityvgamma} we relate the events $V_{\gamma}$ to corresponding one for the dual RCM. While this might fit more logically in the second part, we leave it here as the arguments are more topological in nature. 

\section{Subcomplexes of $\Z^d$}\label{sec:subcomplexes}

The PRCM is a random percolation subcomplex of a cell complex $X,$ where an $i$-dimensional \emph{percolation subcomplex} of $X$ satisfies
\[X^{\paren{i-1}} \subset P \subset X^{\paren{i}}\,.\]
Here, $X^{(j)}$ denotes the \emph{$j$-skeleton} of $X$: the union of all cells of dimension at most $k.$ In what follows, $X$ will usually either be $\Z^d$ or a subcomplex thereof. Recall that the $i$-dimensional cells of $\Z^d$ are the $i$-dimensional unit cubes with integer corner points. This complex has an associated dual complex $\paren{\Z^d}^{\bullet},$ which is obtained from it by shifting by $1/2$ in each coordinate direction. There is then a pairing that matches each $i$-plaquette of $\Z^d$ to the unique $(d-i)$-plaquette intersecting it at its center point. In particular, an $i$-dimensional percolation subcomplex $P$ has an associated dual complex $Q$ consisting of the union of the $(d-i-1)$-skeleton of $\paren{\Z^d}^{\bullet}$ and the dual of each omitted plaquette of $P.$ 

Let $r=\brac{a_1,b_1}\times\ldots\times \brac{a_d,b_d}$ be a box in $\Z^d.$ For convenience, we will abuse notation throughout and let $r$ refer both to the aforementioned closed box and (when $i$ is specified) the $i$-dimensional subcomplex obtained from $r^{(i)}$ by removing the $i$-plaquettes in $\partial r.$ In particular, an $i$-dimensional subcomplex of $r$ is not allowed to contain $i$-plaquettes in $\partial r.$ We will denote the full $i$-skeleton of $r$ by $\overline{r}.$ 

The dual of a percolation subcomplex $P \subset r$ (or $\overline{r}$) will be a percolation subcomplex of a slightly shrunk (expanded) box. For $\epsilon\geq -1,$ let $r^{\epsilon}$ be the box $\brac{a_1-\epsilon,b_1+\epsilon}\times\ldots\times \brac{a_d-\epsilon,b_d+\epsilon}.$ The dual complex $Q'$ of $P' \subset \overline{r}$ is defined exactly as above and is a percolation subcomplex of $r^{\bullet}\coloneqq r^{1/2},$ while the dual complex $Q$ of $P \subset r$ is a percolation subcomplex of $\overline{r^{\bullet}}\coloneqq \overline{r^{-1/2}}.$

 The topological properties of $Q$ are closely related to those of the complement $\R^d\setminus P.$
\begin{Proposition}
\label{cor:alexander}
Fix $0<i<d$ and a box $r$ in $\Z^d.$ If $P$ is a percolation subcomplex of $\overline{r}$ ($r$), $Q$ is the dual complex, and $r'$ is the box $r^{\bullet}$ (respectively $\overline{r^{\bullet}}$) then there is an isomorphism
 \[\mathcal{I} : H_{i}\paren{P_r;\;\Z} \to H^{d-i-1}\paren{Q\cup \partial r';\; \Z}\]
 where $H_j\paren{X;\;\Z}$ and $H^j\paren{X;\;\Z}$ denote the $j$-dimensional reduced homology and the $j$-dimensional reduced cohomology of $X$ with integral coefficients.
\end{Proposition}
A proof is included in Appendix~\ref{sec:dual}. 

\section{Relationship with the PLGT}
\label{sec:coupling}
The main goal of this section is to prove Theorem~\ref{thm:comparison} for infinite volume measures constructed with free boundary conditions. We begin by coupling the PRCM and the PLGT on a finite cell complex $X.$ Recall that if $X\subset \Z^d,$ these measures are said to have free boundary conditions. 

\subsection{Finite Volume Measures}
\label{subsec:couplingfree}

The following statement is a generalization of the classical Edwards--Sokal coupling when $i=1$~\cite{edwards1988generalization} and of previous results for the special case when $q$ is a prime integer~\cite{hiraoka2016tutte}.
\begin{Proposition}\label{prop:couple}
Let $X$ be a finite cubical complex, $q\in\N+1,$ $\beta \in [0,\infty),$ and $p = 1-e^{-\beta}.$ Define a coupling on $C^{i-1}\paren{X}\times \set{0,1}^{X^{\paren{i}}}$ by
\[\kappa\paren{f,P} \propto \prod_{\sigma \in X^{\paren{i}}}\brac{\paren{1-p}I_{\set{\sigma \notin P}} + p I_{\set{\sigma \in P,\delta f\paren{\sigma}=0}}}\,.\]

Then $\kappa$ has the following marginals.
\begin{itemize}
    \item The first marginal is
    $\nu_{X,\beta,q,i-1}.$ 
    \item The second marginal is
    $\tilde{\mu}_{X,p,\Z_q,i}.$ 
\end{itemize}
\end{Proposition}

\begin{proof}
The proof that the first marginal is the $q$-state PLGT is identical that for the case where $q$ is a prime integer~\cite{hiraoka2016tutte}.

The computation of the second marginal proceeds similarly, with minor differences towards the end of the computation. 
\begin{align*}
    \kappa_2\paren{P} & \coloneqq \sum_{f \in C^{i-1}\paren{X}} \kappa\paren{f,P}\\
    &\propto \sum_{f \in C^{i-1}\paren{X}} \prod_{\sigma \in X^{\paren{i}}} \brac{\paren{1-p}I_{\set{\sigma \notin P}} + p I_{\set{\sigma \in P,\delta f\paren{\sigma}=0}}}\\
    &= \paren{1-p}^{\abs{X^{\paren{i}}} -\abs{P}}p^{\abs{P}}\sum_{f \in C^{i-1}\paren{X}} \prod_{\substack{\sigma \in X^{\paren{i}}\\ \sigma \in P}} I_{\set{\delta f\paren{\sigma}=0}}\\
    &= \paren{1-p}^{\abs{X^{\paren{i}}} -\abs{P}}p^{\abs{P}}\abs{Z^{i-1}\paren{P;\;\Z_q}}\\
    &= \paren{1-p}^{\abs{X^{\paren{i}}} -\abs{P}}p^{\abs{P}}\abs{H^{i-1}\paren{P;\;\Z_q}}\abs{B^{i-1}\paren{P;\;\Z_q}}\\
    &\propto \paren{1-p}^{\abs{X^{\paren{i}}} -\abs{P}}p^{\abs{P}}\abs{H^{i-1}\paren{P;\;\Z_q}}\,,
\end{align*}
because $B^{i-1}\paren{P;\;\Z_q}$ does not depend on $P.$
\end{proof}

The following characterization of the conditional measures of the coupling follows from the same proof of Proposition 21 of~\cite{duncan2022topological}.

\begin{Corollary}\label{cor:conditionalmeasures}
    Let $p = 1-e^{-\beta}.$ Then $\kappa$ has the following conditional measures:
\begin{itemize}
    \item Given $f,$ the conditional measure $\kappa\paren{\cdot \mid f}$ is Bernoulli plaquette percolation with probability $p$ on the set of plaquettes $\sigma$ that satisfy $\delta f\paren{\sigma} = 0.$
    \item Given $P,$ the conditional measure $\kappa\paren{\cdot \mid P}$ is the uniform measure on $\paren{i-1}$-cocycles in $Z^{i-1}\paren{P;\;\Z_q}.$
\end{itemize}
\end{Corollary}

We now show the analogue of Theorem~\ref{thm:comparison} for finite cubical complexes. This is a generalization of Theorem 5 of~\cite{duncan2022topological}. 
  \begin{Proposition}
\label{prop:comparisonbody}
Let $X$ be a finite cubical complex, $0<i<d-1,$  $q\in\N+1,$ and $\gamma\in Z_{i-1}\paren{X;\;\Z_q}.$  Then, if $H_{i-2}\paren{X;\;\Z_q}=0,$ 
\[\mathbb{E}_{\nu}\paren{W_{\gamma}}=\tilde{\mu}\paren{V_{\gamma}}\,,\]
where $\nu=\nu_{X,\beta,q,i-1,d}$ is the PLGT, $\tilde{\mu}=\tilde{\mu}_{X,1-e^{-\beta},\Z_q,i}$ is the corresponding PRCM, and $V_{\gamma}$ is the event that $\brac{\gamma}=0$ in $H_{i-1}\paren{P;\;\Z_q}.$ 
\end{Proposition}
\begin{proof}
We compute $\mathbb{E}_{\nu}\paren{W_{\gamma}}$ using the law of total conditional expectation:
\[\mathbb{E}_{\nu}\paren{W_{\gamma}}=\mathbb{E}_{\kappa}\paren{W_{\gamma}}=\mathbb{E}_{\kappa}\paren{W_{\gamma}  \middle |   V_{\gamma}}\kappa\paren{V_{\gamma}}+\mathbb{E}_{\kappa}\paren{W_{\gamma}\mid \neg V_{\gamma}}\kappa\paren{\neg V_{\gamma}}\,.\]
The desired result will follow if we demonstrate that $\mathbb{E}_{\kappa}\paren{W_{\gamma} \middle |  V_{\gamma}}=1$ and  $\mathbb{E}_{\kappa}\paren{W_{\gamma} \middle |  \neg V_{\gamma}}=0.$ 

First, if $W_{\gamma}$ occurs then
\[\gamma=\partial \paren{\sum_{\sigma} a_{\sigma} \sigma}\]
where the sum is taken over $i$-plaquettes $\sigma$ so that $f\paren{\partial \sigma}=0.$ By linearity, $f\paren{\gamma}=0$ and $W_{\gamma}=f\paren{\gamma}^{\mathbb{C}}=1.$   Thus $\mathbb{E}_{\kappa}\paren{W_{\gamma} \middle |  V_{\gamma}}=1.$ 

Now, assume that $\brac{\gamma}\neq 0$ in $H_{i-1}\paren{P;\;\Z_q}.$ By Corollary~\ref{cor:UCTC}, 
\[H^{i-1}\paren{P;\;\Z_q}\cong \mathrm{Hom}\paren{H_{i-1}\paren{P;\;\Z_q},\Z_q} \cong H_{i-1}\paren{P;\;\Z_q}\,,\]
where we are using the assumption that $H_{i-2}\paren{X;\;\Z_q}=0.$ Thus, there exists an $f\in Z^{i-1}\paren{P;\;\Z_q}$ so that $f\paren{\gamma}\neq 0.$ It follows that the conditional random variable $\paren{W_{\gamma} \middle |  P}$ does not vanish and in fact --- by symmetry --- is distributed on a non-trivial subgroup of the $q$-th complex roots of unity $\Z\paren{q}.$ The only such subgroups are those of the form $\Z\paren{m}$ for $m\mid q.$ The $m$-th roots of unity sum to zero so  $\mathbb{E}_{\kappa}\paren{W_{\gamma}\mid P}=0$ for any $P$ so that $\neg V_{\gamma}.$ Therefore, by the law of total conditional expectation $\mathbb{E}_{\kappa}\paren{W_{\gamma} \middle |  \neg V_{\gamma}}=0.$ 
\end{proof}


\subsection{Infinite Volume Measures}
\label{subsec:comparisoninfinite}
We construct infinite volume measures for both the PRCM and the PLGT using free boundary conditions. First, we prove two technical lemmas.

Let $P$ be a percolation subcomplex of $\Z^d.$ For convenience, set $P_n=P\cap \Lambda_n,$ where $\Lambda_n \coloneqq [-n,n]^d.$ For $N > n,$ define two restriction maps 
\[\phi_{N,n}:Z^{i-1}\paren{P_N;\;\Z_q}\to Z^{i-1}\paren{P_n;\;\Z_q}\] and 
\[\phi_{\infty,n}:Z^{i-1}\paren{P;\;\Z_q}\to Z^{i-1}\paren{P_n;\;\Z_q}\,.\]
Set $Y_{N,n}=\im\phi_{N,n}$ and $Y_{\infty,n}=\im\phi_{\infty,n}.$ 

\begin{Lemma}
    \label{lemma:phiNM}
    Let $\#\in\set{N,\infty}$ and $f\in Z^{i-1}\paren{P_n;\;\Z_q}.$ Then  $f\notin \im\phi_{\#,n}$ if and only if there exists a cycle $\sigma\in Z_{i-1}\paren{P_n;\;\Z_q}\cap B_{i-1}\paren{P_{\#};\;\Z_q}$ so that $f\paren{\sigma}\neq 0.$ 
\end{Lemma}
\begin{proof}
We begin by showing that the statement is equivalent to an analogous one for cohomology classes. Suppose that $\brac{f_n}=\brac{f_n'}\in H^{i-1}\paren{P_n;\;\Z_q}.$ Then there exists a $g_n\in C^{i-2}\paren{P_n;\;\Z_q}$ so that $\delta g_n=f_n-f_n'.$ We may extend $g_n$ to obtain a cochain $g_{\#}\in C^{i-2}\paren{P_{\#};\;\Z_q}$ which vanishes on  $(i-2)$-plaquettes outside of $\Lambda_n.$ Then $f_n-f_n'=\phi_{\#,n}\paren{\delta g_\#}.$ It follows $f_n\in \im\phi_{\#,n}\iff  f_n' \in \im \phi_{\#,n}.$ 

Let $\phi_{\#,n}^*:H^{i-1}\paren{P_{\#};\;\Z_q}\to H^{i-1}\paren{P_n;\;\Z_q}$ be the induced map on cohomology. It suffices to demonstrate the corresponding characterization of  $\im \phi_{\#,n}^*.$ This follows directly from the definition of the long exact sequence of the pair $(P_{\#},P_n)$ and the associated boundary map (see page 199 of~\cite{hatcher2002algebraic}; more detail is given for the homological analogue on page 115). 

\end{proof}

\begin{Lemma}
Fix $n\in \N.$ 
    \label{lemma:cylinder}
    \begin{itemize} 
    \item $Y_{N,n}=Y_{\infty,n}$ for all sufficiently large $N.$
    \item For $N$ sufficiently large as in the previous statement, the pushforward by $\phi_{N,n}$ of the uniform measure on $Z^{i-1}\paren{P_N;\;\Z_q}$ is the uniform measure on $Y_{\infty,n}.$ 
    \item For all $n<N,$ $\phi_{N,n}\paren{Y_{\infty,N}}=Y_{\infty,n}$ and the pushforward of the uniform measure on $Y_{\infty,N}$ by $\phi_{N,n}$ is the uniform measure on $Y_{\infty,n}.$ 
    \end{itemize}
\end{Lemma}

\begin{proof}
The first statement follows quickly from the previous lemma. By construction, for each $f\in Z^{i-1}\paren{P_n;\;\Z_q}\setminus \im \phi_{\infty,n}$ we can find a chain $\sigma \in C_{i}\paren{P;\;\Z_q}$ so that $\partial \sigma\in Z_{i-1}\paren{P_n;\;\Z_q}$ and $f\paren{\partial\sigma}\neq 0.$ By definition, $\sigma$ is finite so it is supported on all sufficiently large cubes in $\Z^d.$ In fact, $Z^{i-1}\paren{P_n;\;\Z_q}$ is itself finite so we may choose a cube $\Lambda_N$ large enough to support obstructing boundaries for all such $f.$ Then $Y_{N,n}=Y_{\infty,n}.$   

To show the second statement, we use the decomposition 
\[Z^{i-1}\paren{P_n;\;\Z_q}\cong \ker \phi_{N,n}\oplus \im \phi_{N,n}\,.\]
In particular, a uniform element of a direct sum can be obtained by taking a uniform element of each summand, and so the pushforward is simply a uniform element of the second summand.

The third statement follows from a similar argument applied to the decomposition
\[Y_{\infty,N} = \ker\paren{\phi_{N,n}\mid_{Y_{\infty,N}}} \oplus \im\paren{\phi_{N,n}\mid_{Y_{\infty,N}}}\]
where we used that $\phi_{N,n}\paren{Y_{\infty,N}}=Y_{\infty,n}.$
\end{proof}

The same proof goes through in a slightly more general setting, where $r \subset \Lambda_N$ is a $d$-dimensional box. Let $\phi_{N,r} : Z^{i-1}\paren{P \cap \Lambda_N;\;\Z_q} \to Z^{i-1}\paren{P \cap r;\;\Z_q}$ and $\phi_{\infty,r} : Z^{i-1}\paren{P;\;\Z_q} \to Z^{i-1}\paren{P \cap r;\;\Z_q}$ be the restriction maps. Set $Y_{\infty, r} = \im \paren{\phi_{\infty,r}}$ and set $Y_{N,r} = \im \paren{\phi_{N,r}}.$ We only require the analogue of the third bullet point above. 

\begin{Corollary}\label{cor:boxcylinder}
$\phi_{N,r}\paren{Y_{\infty,N}}=Y_{\infty,r}$ and the pushforward of the uniform measure on $Y_{\infty,N}$ by $\phi_{N,r}$ is the uniform measure on $Y_{\infty,r}.$ 
\end{Corollary}

We now state the main result of this section.

\begin{Proposition}\label{prop:comparisoninfinite}
Let $0<i<d-1,$ $q\in\N+1,$ $\beta \in \paren{0,\infty}$ and $p = 1-e^{-\beta}.$  
The weak limits
\[\mu_{\Z^d,p}^{\mathbf{f}}=\lim_{N\rightarrow \infty} \mu^{\mathbf{f}}_{\Lambda_N,p,q,d-1}\]
and
\[\nu_{\Z^d}^{\mathbf{f}}=\lim_{n\rightarrow \infty} \nu_{\Lambda_N,\beta,q,d-1}^{\mathbf{f}}\]
exist and are translation invariant.
Moreover, if $\gamma$ is a $(i-1)$-cycle in $\Z^d$ then
\[\mathbb{E}_{\nu_{\Z^d}^{\mathbf{f}}}\paren{W_{\gamma}}=\mu_{\Z^d,p}^{\mathbf{f}}\paren{V_{\gamma}^{\mathrm{fin}}}\,.\]
\end{Proposition}
 Note that the last statement is Theorem~\ref{thm:comparison} for free boundary conditions.
\begin{proof}

The weak limit of $\mu^{\mathbf{f}}_{\Lambda_N,p,q,d-1}$ exists and is translation invariant by the same logic as for the case $i=1.$ The proof uses the FKG inequality and a standard monotonicity argument; see Theorem 4.19 of~\cite{grimmett1999percolation}. In fact, we may couple the PRCMs with $P\paren{1}\subset P\paren{2}\subset \ldots$ with $P\paren{N} \sim \mu^{\mathbf{f}}_{\Lambda_N,p,q,d-1}$ and $P=\cup_{N} P\paren{N} \sim \mu^{\mathbf{f}}_{\Z^d,p}.$ 

We will construct a coupling whose first marginal is $\mu^{\mathbf{f}}_{\Z^d,p}$ and whose second marginal is the weak limit of $\nu_{\Lambda_N,\beta,q,d-1}^{\mathbf{f}}$ as $N\to\infty.$ Let $\Omega=\set{0,1}^{\paren{\Z^d}^{(i)}}$ and $\Sigma=C^{i-1}\paren{\Z^d;\;\Z_q}.$ The $\sigma$-algebra on $\Omega$ is generated by cylinder events of the form
\[\mathcal{K}\paren{P_n}\coloneqq \set{P\subset \Z^d: P\cap \Lambda_n=P_n}\]
for a percolation subcomplex $P_n$ of $\Lambda_n.$ 
 Similarly, the cylinder events for $\Sigma$ are
\[\mathcal{L}\paren{f_n}\coloneqq \set{f\in C^{i-1}\paren{\Z^d;\;\Z_q}:f\mid_{\Lambda_n}=f_n}\]
 where $f_n\in C^{i-1}\paren{\Lambda_n;\;\Z_q}.$ 
Then the $\sigma$-algebra on $\Omega \times \Sigma$ is generated by the products $\mathcal{K}\paren{P_n} \times \mathcal{L}\paren{f_n}.$

We define a coupling $\kappa_{\beta,q}^{\mathbf{f}}$ on $\Omega\times \Sigma$ by first specifying it on cylinder events.  Set
\begin{align*}
    \kappa_{\beta,q}^{\mathbf{f}}\paren{\mathcal{K}\paren{P_n}\times \mathcal{L}\paren{f_n}} = \sum_{H\subseteq Z^{i-1}\paren{P_n;\;\Z_q}}\frac{I_{\set{f_n \in H}}}{\abs{H}}\mu_{\Z^d,p}\paren{Y_{\infty,n}=H \cap \mathcal{K}\paren{P_n}}\,,
\end{align*}
where we note that $\set{Y_{\infty,n}=H}$ is measurable because $Y_{\infty,n} = \bigcap_{N>n} Y_{N,n}.$
In words, given $P_n$ we can sample $f_n$ by revealing $Y_{\infty,n}$ and selecting a uniform random element therein.

We now need to check that our partial definition of $\kappa$ extends to a measure on $\Omega \times \Sigma.$ By Carath\'{e}odory's extension theorem for semi-rings, the only remaining requirement is countable additivity on cylinder sets. 

As $\Omega \times \Sigma$ is a product of countably many finite spaces, no cylinder set is an infinite countable disjoint union of cylinder sets (see Chapter 8.6 of~\cite{lindstrom2017spaces}). Thus, it suffices to check finite additivity. We claim that it is enough to show that 
\begin{equation}
    \label{eq:consistent}
\kappa_{\beta,q}^{\mathbf{f}}\paren{\mathcal{K}\paren{P_n}\times \mathcal{L}\paren{f_n}}=\sum_{\paren{P_N,f_N}\in \mathcal{R}\paren{P_n,f_n,N}}\kappa_{\beta,q}^{\mathbf{f}}\paren{\mathcal{K}\paren{P_N},\mathcal{L}\paren{f_N}}
\end{equation}
when $n>N,$ where
\[\mathcal{R}\paren{P_n,f_n,N}=\set{\paren{P_n,f_n}:P_N\cap \Lambda_n=P_n, \phi_{N,n}\paren{f_N}=f_n}\,.\]
Suppose that
\[\mathcal{K}\paren{P_n} \times \mathcal{L}\paren{f_n} = \bigsqcup_{k=1}^l C_k\]
where each $C_k$ is a cylinder event for a a cube $\Lambda_{N_k}.$ Then, if we choose $N=\max_{k}N_k,$ we can use (\ref{eq:consistent}) to reduce the statement 
\begin{equation}
\label{eq:cons1}
\kappa_{\beta,q}^{\mathbf{f}}\paren{\mathcal{K}\paren{P_n} \times \mathcal{L}\paren{f_n}} \stackrel{?}{=} \sum_{k=1}^l  \kappa_{\beta,q}^{\mathbf{f}}\paren{C_k}
\end{equation}
to one of the form
\begin{equation}
\label{eq:cons2}
\kappa_{\beta,q}^{\mathbf{f}}\paren{\mathcal{K}\paren{P_n} \times \mathcal{L}\paren{f_n}} \stackrel{?}{=} \sum_{l=1}^L  \kappa_{\beta,q}^{\mathbf{f}}\paren{\mathcal{K}\paren{P_N^l}\cap \mathcal{L}\paren{f_N^l}}
\end{equation}
where the $P_N^l$ are subcomplexes of $\Lambda_N$ and the $f_N^l$ are cochains in $C^{i-1}\paren{\Lambda_n;\;\Z_q}$ satisfying 
\[\mathcal{K}\paren{P_n} \times \mathcal{L}\paren{f_n} = \bigsqcup_{l=1}^L \mathcal{K}\paren{P_N^l}\cap \mathcal{L}\paren{f_N^l}\,.\]
Now, the only way to achieve this decomposition is if $P_N^l\cap \Lambda_n=P_n$ and $f_N^l\mid_{\Lambda_n}=f_n$ for each $l,$ and the pairs $\paren{P_N^l,f_N^l}$ cover all possible plaquette/spin combinations in the annulus $\Lambda_N\setminus \Lambda_n.$ That is, we must have 
\[\set{\paren{P_N^l,f_N^l}}_{l=1}^L=\mathcal{R}\paren{P_n,f_n,N}\]
and then (\ref{eq:cons1}) and (\ref{eq:cons2}) follow from (\ref{eq:consistent}).

Now, we demonstrate (\ref{eq:consistent}).
\begin{align*}
   &\sum_{\paren{P_N,f_N}\in \mathcal{R}\paren{P_n,f_n,N}}\kappa_{\beta,q}^{\mathbf{f}}\paren{\mathcal{K}\paren{P_N},\mathcal{L}\paren{f_N}}\\
     =&\sum_{\substack{\paren{P_N,f_N}\in \mathcal{R}\paren{P_n,f_n,N}\\H\subseteq Z^{i-1}\paren{P_N;\;\Z_q}}}\frac{I_{\set{f_N \in H}}}{\abs{H}}\mu_{\Z^d,p}\paren{\set{Y_{\infty,N}=H} \cap \mathcal{K}\paren{P_N}}\\
    =&\sum_{\substack{P_N: P_N\cap \Lambda_n = P_n\\H\subseteq Z^{i-1}\paren{P_N;\;\Z_q}}}\frac{\abs{\set{f_N\in H: f_N \mid_{\Lambda_n} = f_n}}}{\abs{H}}\mu_{\Z^d,p}\paren{\set{Y_{\infty,N}=H} \cap \mathcal{K}\paren{P_N}}\\
     =&\sum_{\substack{P_N: P_N\cap \Lambda_n = P_n\\H'\subseteq Z^{i-1}\paren{P_n;\;\Z_q}}}\frac{I_{\set{f_n\in H'}}}{\abs{H'}}\mu_{\Z^d,p}\paren{\phi_{N,n}\paren{\set{Y_{\infty,N}}=H'} \cap \mathcal{K}\paren{P_N}}\\
     =&\sum_{H'\subseteq C^{i-1}\paren{\Lambda_n;\;\Z_q}}\frac{I_{\set{f_n\in H'}}}{\abs{H'}}\sum_{P_N: P_N\cap \Lambda_n = P_n}\mu_{\Z^d,p}\paren{\set{Y_{\infty,n}=H'} \cap \mathcal{K}\paren{P_N}}\\
     =&\sum_{H'\subseteq C^{i-1}\paren{\Lambda_n;\;\Z_q}}\frac{I_{\set{f_n\in H'}}}{\abs{H'}}\mu_{\Z^d,p}\paren{\set{Y_{\infty,n}=H'}\cap \mathcal{K}\paren{P_n}}\\
     =& \kappa_{\beta,q}^{\mathbf{f}}\paren{\mathcal{K}\paren{P_n}\times \mathcal{L}\paren{f_n}}\,,
 \end{align*}
 where we used the third bullet point in Lemma~\ref{lemma:cylinder} to move from the third to the fourth line of the computation. Thus, we have defined a measure $\kappa_{\beta,q}^{\mathbf{f}}$ on $\Omega\times \Sigma.$ Notice that a similar computation can be used to prove translation invariance. That is, if $\Lambda_n'$ is a translate of $\Lambda_n,$ $P_n'$ and $f_n'$ are the corresponding shifts of $P_n$ and $f_n,$ and $N$ is sufficiently large then
\begin{align*}
  &\kappa_{\beta,q}^{\mathbf{f}}\paren{\mathcal{K}\paren{P_n'}\times \mathcal{L}\paren{f_n'}}\\
  =&\sum_{\substack{P_N: P_N\cap \Lambda'_{n} = P_{n}\\H'\subseteq Z^{i-1}\paren{P_{n'};\;\Z_q}}}\frac{I_{\set{f'_n\in H'}}}{\abs{H'}}\mu_{\Z^d,p}\paren{\phi_{N,\Lambda_{n}'}\paren{\set{Y_{\infty,N}}=H'} \cap \mathcal{K}\paren{P_N}}\\
  =&  \sum_{\substack{P_N: P_N\cap \Lambda_n = P_n\\H'\subseteq Z^{i-1}\paren{P_n;\;\Z_q}}}\frac{I_{\set{f_n\in H'}}}{\abs{H'}}\mu_{\Z^d,p}\paren{\phi_{N,n}\paren{\set{Y_{\infty,N}}=H'} \cap \mathcal{K}\paren{P_N}}\\
  =&\kappa_{\beta,q}^{\mathbf{f}}\paren{\mathcal{K}\paren{P_n}\times \mathcal{L}\paren{f_n}}
  \end{align*}
 where we used the translation invariance of $\mu_{\Z^d,p}$ and Corollary~\ref{cor:boxcylinder}.

Next, we verify that the marginals are as claimed. That the first one is $\mu_{\Z^d,p}^{\mathbf{f}}$ is immediate from the definition. For the second marginal, consider the conditional distribution obtained by restricting to $\Lambda_n.$ 

By definition, it assigns to $f_n\in C^{i-1}\paren{\Lambda_n;\;\Z_q}$ the probability 
\begin{equation}\label{eq:marginalinf}    \mathbb{E}_{\mu_{\Z^d,p}^{\mathbf{f}}}\brac{\sum_{H\subseteq Z^{i-1}\paren{P_n;\;\Z_q}}\frac{I_{\set{f_n \in H, Y_{\infty,n} = H}}}{\abs{H}}}\,.
\end{equation}
On the other hand, by Corollary~\ref{cor:conditionalmeasures}, the restriction of $\nu_{\Lambda_N,\beta,q,d-1}^{\mathbf{f}}$ to $\Lambda_n$ gives the probability as
\begin{equation}\label{eq:marginaln}
    \mathbb{E}_{\mu_{\Z^d,p}^{\mathbf{f}}} \brac{\sum_{H\subseteq Z^{i-1}\paren{P(N)\cap \Lambda_n;\;\Z_q}}\frac{I_{\set{f_n \in H, Y_{N,n} = H}}}{\abs{H}}}\,.
\end{equation}
Fix $P.$ We may choose $N$ large enough so that $P\paren{N}\cap \Lambda_n=P_n$ and $Y_{N,n}=Y_{\infty,n},$ where the first claim follows because $P\paren{N}\nearrow P$ and the second is the first item of Lemma~\ref{lemma:cylinder}. Thus, the inner term of (\ref{eq:marginaln}) converges to the inner term of (\ref{eq:marginalinf}) pointwise as as function of $P$ as $N \to \infty.$ Therefore, by bounded convergence theorem, the second marginal of $\kappa_{\beta,q}^{\mathbf{f}}$ is the weak limit of the measures $\nu_{\Lambda_N}^{\mathbf{f}}.$ 

We now demonstrate Theorem~\ref{thm:comparison} for free boundary conditions. Let $V\paren{N}$ be the event that $\gamma$ is a boundary in $P\paren{N}.$ Then $P\paren{N}\nearrow P$ so $V\paren{N}\nearrow V_{\gamma}^{\mathrm{fin}}$ and
\[\mu_{\Z^d,p}^{\mathbf{f}}\paren{V_{\gamma}^{\mathrm{fin}}} = \lim_{N \to \infty} \mu_{\Lambda_N,p}^{\mathbf{f}}\paren{V\paren{N}} = \lim_{N \to \infty} \mathbb{E}_{\nu_{\Lambda_N}^{\mathbf{f}}}\paren{W_{\gamma}} = \mathbb{E}_{\nu_{\Z^d}^{\mathbf{f}}}\paren{W_{\gamma}}\,,\]
where the second equality follows from Proposition~\ref{prop:comparisonbody} and the third is implied by weak convergence.

\end{proof}

\section{The PRCM in Codimension One}\label{sec:codim1PRCM}

\subsection{Basic Properties}
\label{subsec:codim1}
Recall that the PRCM with coefficients in the rational numbers $\mathbb{Q}$ is the random $i$-dimensional percolation subcomplex of $X$ satisfying 
 \[\mu_{X,p,q,i}\paren{P} \propto p^{\abs{P}}\paren{1-p}^{\abs{X^{\paren{i}}} - \abs{P}}q^{\mathbf{b}_{i-1}\paren{P;\Q}}\,.\]

\begin{Proposition}\label{prop:codim1}
Let $r$ be a box in $\Z^d.$ Then $(d-1)$-dimensional PRCM $\tilde{\mu}_{r,p,G,d-1}$ on $r$ with coefficients in a finite abelian group $G$ is equal in distribution to the PRCM $\mu_{r,p,\abs{G},d-1}$ with coefficients in $\Q.$ In particular, 
\[\tilde{\mu}_{r,p,\Z_q,d-1} \,{\buildrel d \over =}\, \mu_{r,p,q,d-1}\,.\]
\end{Proposition}
\begin{proof}
This follows from the fact that, if $P$ is a percolation subcomplex of $r$ then
\[H^{d-2}\paren{P;\;\Z_q}\cong \Z_q^{\mathbf{b}_{d-2}\paren{P;\;\Q}}\,.\]
This is Proposition~\ref{prop:codim1topology} in the appendix. The proof uses several standard tools from algebraic topology, including the universal coefficients theorems for homology and cohomology and Alexander duality. See Section~\ref{subsec:homologyprops}.
\end{proof}

We review some properties of the PRCM with coefficients in $\Q.$ First, it satisfies the FKG inequality.

\begin{Theorem}[\cite{hiraoka2016tutte}]\label{thm:FKG}
Let $p\in(0,1),$ and $q\geq 1,$ $i\in \N,$ and $X$ a finite cubical complex. Then $\mu_{X,p}=\mu_{X,p,q,i}$ satisfies the FKG inequality. That is, if $A$ and $B$ are increasing events then 
\begin{equation*}
    \mu_{X,p}\paren{A\cap B} \geq \mu_{X,p}\paren{A}\mu_{X,p}\paren{B}.
\end{equation*}
\end{Theorem}

Next, we have the following duality relation.

\begin{Theorem}[\cite{duncan2022topological}]\label{thm:duality}
\begin{equation*}
   \mu_{r,p,q,i}\paren{P^{\bullet}} = \mu_{\overline{r^{\bullet}},p^*,q,d-i}^{\mathbf{w}}\paren{P}\,.
\end{equation*}
\end{Theorem}
\begin{proof}
    The proof is nearly identical to that of Theorem 18 in~\cite{duncan2022topological} but uses Proposition~\ref{cor:alexander} instead of Theorem 14 of that paper. 
\end{proof}

\subsection{Boundary Conditions and Infinite Volume Measures}\label{sec:RCMboundary}

The duality between the $(d-1)$- and $1$-dimensional random-cluster models allows us to take a shortcut to defining more general boundary conditions for random-cluster measures on finite subsets of $\Z^d.$ We do not go into too much detail on this topic, as it has been proven that the resulting infinite volume measures are unique except at at most countably many values of $p$~\cite{grimmett1995stochastic}.

First, we recall boundary conditions for the classical random-cluster model on a graph. A boundary condition on a subgraph induced by a vertex set $S$ is a percolation subcomplex $\xi$ on $\paren{\Z^d \setminus S} \cup \partial S.$  Let $P^{\xi}$ be the union of $P$ and the edges of $\xi.$ The idea is to define a random-cluster measure on $S$ with the additional edges of $P^{\xi}$ added for the purpose of counting connected components. Of course, $P^{\xi}$ will have infinitely many connected components in general, but finitely many of them are connected to $S.$ 

More precisely, there is a corresponding random-cluster measure on $S$ with boundary condition $\xi$ written as $\mu_{S,p,q,1}^{\xi}\paren{P},$ where the term $\mathbf{b}_0\paren{P}$ counting the number of connected components of $P$ in $S$ is replaced by the number of connected components of $P^{\xi}$ that intersect $S.$ The free and wired boundary conditions discussed previously can be thought of as the extremal cases of $\xi$ containing no edges or all possible edges, respectively.

We define boundary conditions for the $(d-1)$-dimensional PRCM on a box $r\subset \Z^d$ by duality. 
Let $\xi$ be a boundary condition and denote by $\xi^{\bullet}$ the dual configuration of edges of the dual lattice corresponding to the plaquettes not included in $\xi.$

\begin{Definition}\label{defn:boundaryconditions}
The measure $\tilde{\mu}_{r,p,\Z_q,1}^{\xi}$ is defined by
\[\mu_{r,p,q,d-1}^{\xi}\paren{P}\coloneqq \mu_{\overline{r^{\bullet}},p^*,q,1}^{\xi^\bullet}\paren{Q}\,.\]
\end{Definition}

This is equivalent to setting 
\[\mu_{r,p,q,d-1}^{\xi}\paren{P}\propto p^{\abs{P}}\paren{1-p}^{\abs{X^{\paren{i}}} - \abs{P}}q^{\rank \phi^*}\]
where $\phi^*:H^{d-1}\paren{P^{\xi};\;\Q}\rightarrow H^{d-1}\paren{P;\; \Q}$ is the map on cohomology induced by the inclusion $P\hookrightarrow P^{\xi}.$ A proof of this fact and the definition of an analogous notion for the PRCM with coefficients in an abelian group will be contained in another paper. Alternatively, the PRCM with boundary conditions can be obtained as the restriction to $r$ of the PRCM on a sufficiently large cube $\Lambda,$ conditioned on the states of the plaquettes in $\Lambda \setminus r.$ 

The free boundary conditions (denoted by $\mathbf{f}$) contain no $(d-1)$-plaquettes of $\Z^d\setminus r$ and the wired boundary conditions (denoted by $\mathbf{w}$) contain all $(d-1)$-plaquettes in $\partial r$ (this has the same effect as taking all complementary plaquettes, but is more convenient). By construction, duality maps the PRCM with free boundary conditions to the classical random-cluster model with wired boundary conditions and vice versa. Also, as a consequence of Theorem~\ref{thm:duality}, the PRCM on a box with free boundary conditions coincides with the PRCM on the finite complex $r$ defined above. 

\begin{Proposition}
Let $\set{\xi_n}$ be a sequence of boundary conditions for the $(d-1)$-dimensional PRCM on the cube $\Lambda_n.$ The weak limit 
\[\lim_{n\rightarrow \infty} \mu_{\Lambda_n,p,q,d-1}^{\xi_n}\]
exists if and only if the dual weak limit
\[\lim_{n\rightarrow \infty} \mu_{\overline{\Lambda_n^{1/2}},p^*,q,1}^{\xi_n^{\bullet}}\]
does and the resulting infinite volume measures are dual.
\end{Proposition}

\begin{proof}
    This is immediate from Definition~\ref{defn:boundaryconditions}.
\end{proof}

The following two propositions are corollaries of the analogous, well-known results for the dual classical random-cluster model.

\begin{Proposition}\label{prop:FKGboundary}
Let $p\in\paren{0,1}, q\geq 1, d\in \N.$ Also, fix a box $r$ in $\Z^d$ with boundary conditions $\xi.$  Then $\mu_r^{\xi}=\mu^{\xi}_{r,p,q,d-1}$ satisfies the FKG inequality. That is, if $A$ and $B$ are increasing events then 
\begin{equation*}
    \mu_{r,p}^{\xi}\paren{A\cap B} \geq \mu_{r,p}^{\xi}\paren{A}\mu_{r,p}^{\xi}\paren{B}.
\end{equation*}
\end{Proposition}

\begin{Proposition}\label{prop:extremal}
 Let $r$ be a box in $\Z^d$ and let $\xi$ be any boundary conditions. Then    
 \[\mu_{r,p,q,d-1}^{\mathbf{f}} \leq_{\mathrm{st}}\mu_{r,p,q,d-1}^{\xi}  \leq_{\mathrm{st}}\mu_{r,p,q,d-1}^{\mathbf{w}}\,.\]

In addition, if $\mu_{\Z_d,p,q,d-1}$ is any infinite volume plaquette random-cluster measure obtained as a weak limit of finite volume measures

 \[\mu_{r,p,q,d-1}^{\mathbf{f}} \leq_{\mathrm{st}} \mu_{\Z_d,p,q,d-1}^{\mathbf{f}}   \leq_{\mathrm{st}}\mu_{\Z_d,p,q,d-1}  \leq_{\mathrm{st}} \mu_{\Z_d,p,q,d-1}^{\mathbf{w}}   \leq_{\mathrm{st}} \mu_{r,p,q,d-1}^{\mathbf{w}}\,.\]
\end{Proposition}

Compare this statement with Corollary~\ref{cor:vgammacontain}.

\section{Duality and $V_{\gamma}$}
\label{sec:dualityvgamma}

We can use Alexander duality to characterize the events $V_{\gamma}$ in terms of the dual RCM. We begin by describing this relationship for homology with integer coefficients. In this case, $\brac{\gamma}=0$ in $H_i\paren{P;\;\Z}$  if and only if $\mathcal{I}\paren{\gamma}=0$ in $H^{d-i}\paren{Q;\;\Z}.$  When $i=d-1$ we obtain a more precise statements using linking numbers. 

Fix $k_1,k_2$ so that $d=k_1+k_2+1.$ Let $\gamma_1 \in Z_{k_1}\paren{S^d;\;\Z}$ and let $\gamma_2$ be an oriented embedding of $S^{k_2}$ into $S^d\setminus \gamma_1.$ Define the \emph{linking number} $l\paren{\gamma_1,\gamma_2}$ to be $k$ if $\gamma_1$ is homologous to $k$ times the generator of $H_{k_1}\paren{S^d \setminus \gamma_2;\;\Z}\cong\Z.$ This is equivalent to setting \[l\paren{\gamma_1,\gamma_2}=\mathcal{I}\paren{\brac{\gamma_2}}\paren{\brac{\gamma_1}}\]
where 
\[\mathcal{I}:H_{k_2}\paren{\gamma_2}\rightarrow H^{k_1}\paren{S^d\setminus \gamma_2}\]
is the Alexander duality isomorphism. This is true because both notions define isomorphisms from $H_{k_1}\paren{S^d\setminus \gamma_2;\;\Z}$ to $\Z,$ so they send generators of the former group to $\pm 1$ (the ambiguity in sign is resolved by choosing the generator of $H_{k_1}\paren{S^d \setminus \gamma_2;\;\Z}$ appropriately). For more on this and other definitions of the linking number, see Chapter 5 of~\cite{rolfsen2003knots}. 

We require a standard property of linking numbers, namely that if $\gamma_1$ is an oriented embedding of $S^{k_1}$ into $S^d$ then the linking number is either symmetric or anti-symmetric:
\[l\paren{\gamma_1,\gamma_2}=\paren{-1}^{k_1 k_2 +1}l\paren{\gamma_2,\gamma_1}\,.\]
This has the following corollary.

\begin{Corollary}\label{cor:link3}
    Let $\gamma_1$ and $\gamma_2$ be disjoint, oriented embeddings of $S^{k_1}$ and $S^{k_2}$ into $S^d,$ respectively. If either $\gamma_1$ is contractible in the complement of $\gamma_2$ or $\gamma_2$ is contractible in the complement of $\gamma_1$ then $l\paren{\gamma_1,\gamma_2}=0.$ 
\end{Corollary}

Next, we prove that --- in codimension one --- the homology class of a $(d-2)$-cycle is determined by linking numbers. Recall that a basis for a free $\Z$-module is a linearly independent generating set.
\begin{Proposition}\label{prop:link1}
Let $P$ be a $(d-1)$-dimensional percolation subcomplex of a box $\overline{r} \subset \Z^d$. There are simple cycles $\alpha_1,\ldots,\alpha_n$ of  $Q'\coloneqq =Q\cup \partial r^{\bullet}$ so that the homomorphism $L:H_{d-1}\paren{P;\;\Z}\rightarrow \Z^n$ defined by
\[L\paren{\brac{\gamma}}=\paren{l\paren{\gamma,\alpha_1},\ldots,l\paren{\gamma,\alpha_n}}\]
is an isomorphism.
\end{Proposition}
\begin{proof}

First, we find a basis for $H^1\paren{Q';\;\Z}.$ Let $T'$ be a spanning tree for the one-skeleton of $\partial r^{\bullet}$ and let $Q''=Q\cup T'.$   The inclusion $i:Q''\hookrightarrow Q'$ induces an isomorphism $i_*:H_1\paren{Q'';\;\Z}\rightarrow H_1\paren{Q';\;\Z}$ (adding $T'$ to $Q$ has the same effect on homology as merging all vertices in $\partial r^{\bullet}$).  As $Q''$ is a simple graph, $Z_1\paren{Q'';\Z}=H_1\paren{Q'';\;\Z}$ has a basis of simple cycles. We may construct such a basis by finding a minimum spanning tree $T$ for $Q''$  and choosing a simple cycle $\alpha_j$ for each edge of $Q''\setminus T.$  $H_0\paren{Q';\;\Z}$ is a free $\Z$-module so the duals $\brac{\alpha_1^*},\ldots,\brac{\alpha_n^*}$ form a basis for $H^1\paren{Q';\;\Z}$ by Corollary~\ref{cor:UCTC}. 

For each $j\in \set{1,\ldots,n}$ we define three maps on (co)homology. Let $\phi_j:H_{d-2}\paren{P;\;\Z}\rightarrow H_{d-2}\paren{S^d\setminus \alpha_j}$ and  $\psi_j:H^1\paren{Q';\;\Z}\rightarrow H^1\paren{\alpha_j}$ be the maps induced by the inclusions $P\hookrightarrow S^d\setminus \alpha_j$ and $\alpha_j\hookrightarrow Q',$ respectively. Also, denote by $\mathcal{I}_j:H_{d-2}\paren{S^d\setminus \alpha_j}\rightarrow H^{1}\paren{\alpha_j;\;\Z}$ the Alexander duality isomorphism. Alexander duality is functorial, so the following diagram commutes in the sense that  $\mathcal{I}_j \circ \phi_j=\psi_j\circ \mathcal{I}.$ 

\[\begin{tikzcd}
H_{d-2}\paren{P;\;\Z} \arrow{r}{\phi_j} \arrow[swap]{d}{\mathcal{I}} & H_{d-2}\paren{S^d\setminus \alpha_j;\;\Z} \arrow{d}{\mathcal{I}_j} \\
H^1\paren{Q';\;\Z} \arrow{r}{\psi_j} & H^1\paren{\alpha_j;\;\Z}
\end{tikzcd}
\]

Our next step is to combine the horizontal maps in the previous diagram from different values of $i.$ Before doing so, note that if we choose a generator of $H_{d-2}\paren{S^d\setminus \alpha_j;\;\Z}$ (say $\mathcal{I}^{-1}\paren{\brac{\alpha_j^*}}$), we obtain an isomorphism $L_j: H_{d-2}\paren{S^d\setminus \alpha_j;\;\Z}\rightarrow \Z$ by sending $\brac{\gamma}$ to the linking number $l\paren{\gamma,\alpha_j}.$ Consider the commutative diagram.

\[\begin{tikzcd}
H_{d-2}\paren{P;\;\Z} \arrow{r}{\oplus_j \phi_j} \arrow[swap]{d}{\mathcal{I}} & \oplus_j H_{d-2}\paren{S^d\setminus \alpha_j;\;\Z} \arrow{r}{\oplus_j L_j} \arrow{d}{\oplus_j\mathcal{I}_j} & \Z^n \\
H^1\paren{Q';\;\Z} \arrow{r}{\oplus_j \psi_j} & \oplus_j H^1\paren{\alpha_j;\;\Z}
\end{tikzcd}
\]

The map $\oplus_j \psi_j$ is an isomorphism because the cohomology classes $\brac{\alpha_j^*}$ are a basis for $H^1\paren{Q';\;\Z}.$ The downward maps are also isomorphisms, so $\oplus_j\phi_j$ is as well, by commutativity of the diagram. Finally, we may conclude that that $L=\paren{\oplus_j L_j}\circ \paren{\phi_j}$ is an isomorphism, because each $L_j$ is an isomorphism.

\end{proof}

We have the following immediate corollary. 
\begin{Corollary}\label{cor:link2}
Assume the same hypotheses as in the previous proposition, and $\gamma_1$ be an oriented embedding of $S^{d-2}$ in $P.$ Then $\brac{\gamma_1}\neq 0$ in $H_{d-2}\paren{P;\;\Z}$ if and only if there exists a simple, oriented loop $\gamma_2$ in $Q'$ so that $l\paren{\gamma_1,\gamma_2}\neq 0.$ 
\end{Corollary}

Next, we find an analogous criterion for homology with coefficients in $\Z_q.$ 

\begin{Proposition}\label{prop:linkq}
Assume the same hypotheses as above. Then $\brac{\gamma_1}\neq 0$ in $H_{d-2}\paren{P;\;\Z_q}$ if and only if there exists a simple, oriented loop $\gamma_2$ in $Q'$ so that $l\paren{\gamma_1,\gamma_2}\not\equiv 0 \pmod{q}.$ 
\end{Proposition}

\begin{proof}
We can relate homology with $\Z$ and $\Z_q$ coefficients using the sequence
\begin{equation}
\label{eq:bockstein}
0\rightarrow  H_{d-2}\paren{P;\;\Z} \xrightarrow[]{ \times q} H_{d-2}\paren{P;\;\Z}\xrightarrow[]{\pmod{q}} H_{d-2}\paren{P;\;\Z_q} \rightarrow 0
\end{equation}
which is exact in the sense that the image of each map is the kernel of the next. Exactness follows from the Universal Coefficient Theorem for Homology (Theorem 3A.3 in~\cite{hatcher2002algebraic}) using the properties of $\otimes \Z_q,$ and the fact that  $\mathrm{Tor}\paren{H_{d-3}\paren{P;\;\Z}}=0$ (as $H_{d-3}\paren{P;\;\Z}=0$). More detail on this topics is included in Section~\ref{subsec:homologyprops} in the appendix. (Equivalently, the Bockstein homomorphism $ H_{d-2}\paren{P;\;\Z_q}\to H_{d-2}\paren{P;\;\Z}$ vanishes; see the beginning of Chapter 10 of~\cite{mccleary2001user}).

In words, exactness of the sequence is equivalent to the statement that $\brac{\gamma_1}=0$ in  $H_{d-2}\paren{P;\;\Z_q}$ if and only if there exists a $\gamma_3\in Z_{d-2}\paren{P;\;\Z}$ so that $\brac{\gamma_1}=q\brac{\gamma_3}$ in $H_{d-2}\paren{P;\;\Z}.$ Thus, if $\brac{\gamma_1}=0$ in  $H_{d-2}\paren{P;\;\Z_q}$ then 
\[l\paren{\gamma_1,\gamma_2}=l\paren{q\gamma_3,\gamma_2}=q l\paren{\gamma_3,\gamma_2} \equiv 0 \pmod{q}\]
for all $\gamma_2\in  Z_1\paren{Q';\;\Z}.$ On the other hand, if  $l\paren{\gamma_1,\gamma_2}\equiv 0 \pmod{q}$ for all simple closed loops $\gamma_2$ in $Q'$ then there are integers $b_1,\ldots,b_n$ so that $l\paren{\gamma_1,\alpha_i}=q b_i$ for $i=1,\ldots,n,$ for the simple cycles $\alpha_1,\ldots,\alpha_n$ constructed in Proposition~\ref{prop:link1}. In particular, we have that $\brac{\gamma_1}=q L^{-1}\paren{b_1,\ldots,b_n}$ in $H_{d-2}\paren{P;\;\Z}$ so so $\brac{\gamma_1}=0$ in  $H_{d-2}\paren{P;\;\Z_q}.$
 \end{proof}

 We apply these results using the following two statements. The first is true for any $i,$ but we state it for $i=d-1.$
 
\begin{Corollary}\label{cor:vgammacontain}
Let $m\in \N$ and let $\gamma$ be the boundary of a $(d-1)$-dimensional box $r'$ of $\Z^d.$ Then
    \[V_{\gamma}^{\mathrm{fin}}\paren{1}\subset V_{\gamma}^{\mathrm{fin}}\paren{m} \subset V_{\gamma}^{\mathrm{inf}}\paren{m}\,,\]
    where we recall the notation $V_{\gamma}^{\mathrm{fin}}\paren{1}=V_{\gamma}^{\mathrm{fin}}\paren{\Z}.$ 
\end{Corollary}
\begin{proof}
The first containment is an immediate consequence of (\ref{eq:bockstein}), and the second follows from the definition of $V_{\gamma}^{\mathrm{inf}}\paren{m}.$ 
\end{proof}

\begin{Corollary}\label{cor:linkG}

Let $P$ be a $(d-1)$-dimensional percolation subcomplex of $\Z^d$ and let $\gamma_1=\partial r'$ be the boundary of a $(d-1)$-dimensional box $r'$ of $\Z^d.$ 
\begin{itemize}
    \item If there exists a simple loop $\gamma_2$ of $Q$ so that $l\paren{\gamma_1,\gamma_2}=\pm 1$ then $V_{\gamma_2}^{\mathrm{inf}}\paren{m}$ does not occur for any $m\in \N.$ 
    \item If there is a box $r$ containing $\gamma_1$ so that  $l\paren{\gamma_1,\gamma_2} =0$ for all simple closed loops $\gamma_2$ of $Q'\coloneqq \paren{Q\cap r^{\bullet}} \cup \partial r^{\bullet}$ then  $V_{\gamma}^{\mathrm{fin}}\paren{m}$ occurs for every $m\in \N.$ 
\end{itemize}

\end{Corollary}

\begin{proof}
For the first statement, suppose that $V_{\gamma}^{\mathrm{inf}}\paren{m}$ occurs and let $\gamma_2$ be a simple loop of $Q.$ Let $\Lambda_n \coloneqq \brac{-n,n}^d$ and denote by $P_n$ the percolation subcomplex $\paren{P\cap \Lambda_n} \cup \partial \Lambda_n.$  We have that $\brac{\gamma}=0$ in $H_{d-2}\paren{P_n;\; G}$ for all sufficiently large $n.$ By choosing $n$ large enough so that $\gamma_2\subset \Lambda_n,$ we can conclude that $l\paren{\gamma_1,\gamma_2}\neq \pm 1$ by either Proposition~\ref{prop:linkq} or Corollary~\ref{cor:link3}. 

The second statement follows from Corollary~\ref{cor:link2} and Corollary~\ref{cor:vgammacontain}. 
\end{proof}
By standard results, Corollaries~\ref{cor:vgammacontain} and~\ref{cor:linkG} (and in fact Theorem~\ref{thm:sharpnessRCM}) hold when homology coefficients are taken in the additive group $G$ of a ring with unity. That is, we may state them for the events $V_{\gamma}^{\mathrm{inf}}\paren{G}$ and $V_{\gamma}^{\mathrm{fin}}\paren{G}.$

Finally, we consider the case where $\gamma$ is the ``equator'' of a box and $P$ is a percolation subcomplex of the same box, which will be used in the proof of the area law. Here, the dual criterion for $V_{\gamma}$ does not depend on the coefficients. For example, observe that a non-orientable surface in $\R^3$ whose boundary is contained in a three-dimensional box must leave the box. Let $\Lambda=\Lambda_0\times\brac{-N,N}$ be a $d$-dimensional box, let $\gamma=\partial \Lambda_0\times\set{0},$ and let $\partial^+ \Lambda=\partial\Lambda\cap\set{\vec{e}_{d}>0}$ and $\partial^- \Lambda=\partial\Lambda\cap\set{\vec{e}_{d}<0}.$

\begin{Corollary}\label{corollary:CrossingJ}
 Let $P$ be a percolation subcomplex of $\overline{\Lambda}$ and let  $Q$ be its dual. Then
 $$\neg V_{\gamma} \iff \partial\Lambda^+\xleftrightarrow[Q]{} \partial\Lambda^{-}$$
  where $V_\gamma$ is any of the events $V_{\gamma}\paren{m}.$
 \end{Corollary}
 \begin{proof}
 For the if direction, note that we can construct a loop whose linking number with $\gamma$ is one by taking any path connecting $\partial^+\Lambda$ with $\partial^-\Lambda$ and completing it to a loop in $\Z^d\setminus \Lambda.$ The only if direction holds because any loop linked with $V_{\gamma}$ must have a segment contained in $\Lambda$ that enters $\partial^+\Lambda$ and departs from $\partial^-\Lambda$ or vice versa.   
 \end{proof}

\part{Proof of the Deconfinement Transition}
We now proceed to the proof of Theorem~\ref{thm:sharpnessRCM} for the PRCM with coefficients in $\Q.$ As a consequence of Theorem~\ref{thm:comparison} and Proposition~\ref{prop:codim1}, this suffices to demonstrate Theorem~\ref{thm:sharpness} on the deconfinement transition in Potts lattice gauge theory. The proof is divided into three parts.

In Section~\ref{sec:arealaw}, we show that an area law upper bound holds for $V_{\gamma}$ when $ p< p^*\paren{p_{\mathrm{surf}}(q)},$ meeting the trivial lower bound found by including all plaquettes in a minimal null-homology. Our technique is similar to proofs in~\cite{aizenman1983sharp} and~\cite{bricmont1980surface}. We write a loop $\gamma$ as an approximate sum of many translated copies of a smaller loop $\gamma'$ that form a ``tiling'' of $\gamma.$ This both shows that the area law constant is well defined and allows us to the Wilson loop tension to the surface tension by comparing events with different boundary conditions.

We provide two proofs of the perimeter law for the supercritical PRCM, in Sections~\ref{sec:perimeterlaw} and~\ref{sec:sharpnessperimeter}. First, we construct a a null-homology for $\gamma$ as the boundary of a union of components in the dual RCM. This provides a perimeter law lower bound, complementing the obvious upper bound. Our other proof is substantially more complex, but has the advantage of demonstrating the existence of a sharp constant in the exponent of the perimeter law. Towards that end, we build a hypersurface of plaquettes in the PRCM which precludes the existence of a dual loop linking with $\gamma.$ Another application of Corollary~\ref{cor:linkG} then yields the desired result. 

\section{The Area Law Regime}\label{sec:arealaw}

We begin by comparing two notions of surface tension for the general $i$-dimensional plaquette random-cluster model on $\Z^d$ with the classical notion for the random-cluster model. For an infinite volume PRCM $\mu^{\xi}_{\Z^d,p,q}$ and a choice of  $V_{\gamma} = V_{\gamma}^{\mathrm{fin}}\paren{q'}$ or $V_{\gamma}^{\mathrm{inf}}\paren{q'}$ define the Wilson loop tension as 
\[\tau'=\tau_{p,q,\xi}'\coloneqq \lim_{N\to\infty} \frac{-\log\paren{\mu^{\xi}_{\Z^d,p,q}\paren{V_{\gamma_N}}}}{\mathrm{Area}\paren{\gamma_N}}\,,\]
where $r_N=\brac{-N,N}^{i}\times\set{0}^{d-i}$ and $\gamma_N=\partial r_N.$
Our first result shows that this limit exists, and agrees with the limiting area law constant for any sequence of $i$-dimensional hyperrectangular boundaries whose dimensions diverge to $\infty.$ The proof is the same as that of Proposition 2.4 of~\cite{aizenman1983sharp}.

\begin{Lemma}\label{lemma:sharpconstant}
 $\tau'$ is well-defined and satisfies
\[\lim_{l\rightarrow\infty} \frac{-\log\paren{\mu^{\xi}_{\Z^d,p,q}\paren{V_{\gamma_l}}}}{\mathrm{Area}(\gamma_l)} = \tau'\] 
for any sequence  $\set{\gamma_l}$ of hyperrectangular $(i-1)$-boundaries whose dimensions diverge with $l.$ 
\end{Lemma}

The proof is included below. Next, we demonstrate that $\tau'\paren{p,q,\mathbf{f}}$ coincides with a different notion of Wilson loop tension defined using the PRCM with free boundary conditions on a box whose ``equator'' is $\gamma.$  Set $\Lambda_{N}=\brac{-N,N}^d,$ $\gamma_N=\partial\brac{-N,N}^{i}\times\set{0}^{d-i},$ and 
\[\tau''=\tau_{p,q}''\coloneqq\lim_{N\to\infty} \frac{-\log\paren{\mu_{\Lambda_{N},p,q}^{\mathrm{f}}\paren{V_{\gamma_N}}}}{\paren{2N}^i}\,.\]

\begin{Lemma}\label{lemma:sharpconstant2}
Let $p$ be such that there is a unique infinite volume PRCM. Then
\[\lim_{l\rightarrow\infty} \frac{-\log\paren{\mu^{\mathbf{f}}_{\Lambda_l,p,q}\paren{V_{\gamma_l}}}}{\mathrm{Area}(\gamma_l)} =\tau_{p,q,\mathbf{f}}'\]
In particular, $\tau_{p,q}''$ is well-defined and equals $\tau_{p,q}'.$ 
\end{Lemma}

Finally, we specialize to the case of $i=d-1$ and show that $\tau_{p,q}'$ coincides with the surface tension $\tau_{p^*\paren{p},q}$ of the dual random-cluster model as defined in Section~\ref{sec:background}.  This generalizes the result of Bricmont, Lebowitz, and Pfister~\cite{bricmont1980surface} for the case $q=2,d=3,$  and implies the area law in our main theorem.

\begin{Proposition}\label{prop:dualtension}
Let $p$ be such that there is a unique infinite volume PRCM. Then
$$\tau_{p,q}''=\tau_{p^*\paren{p},q}\,.$$
\end{Proposition}
\begin{proof}
    This follows immediately from Lemma~\ref{lemma:sharpconstant2} and Corollary~\ref{corollary:CrossingJ}.
\end{proof}

\begin{proof}[Proof of Lemma~\ref{lemma:sharpconstant}]
We proceed as in Proposition 2.4 of~\cite{aizenman1983sharp}. For a loop $\gamma,$ let $r\paren{\gamma}$ be the box whose boundary is the support of $\gamma.$ Let 
\[\mathcal{E} \coloneqq \set{\set{\gamma_l} : m\paren{r\paren{\gamma_l}} \xrightarrow{l \to \infty} \infty}\]
where we recall that $m\paren{r}$ is the minimum of the $i$ dimensions of $r,$ 
and let $\set{\gamma_k'} \in \mathcal{E}.$ We may tile $r\paren{\gamma_l}$ with $m\coloneqq \floor{\frac{\mathrm{Area}(\gamma_l)}{\mathrm{Area}(\gamma_k')}}$ translates of $r\paren{\gamma_k'}$ (call them $r_1,\ldots,r_m$) with the exception of $o\paren{\mathrm{Area}(\gamma_l)}$ $i$-plaquettes (call the set of such plaquettes $T.$) Notice that
\[\partial r\paren{\gamma}=\sum_{j=1}^m \partial r_m+\sum_{\sigma \in T} \partial \sigma\]
when the chains are oriented appropriately. It follows that $V_{\gamma_l}$ is implied by at most $\floor{\frac{\mathrm{Area}(\gamma_l)}{\mathrm{Area}(\gamma_k')}}$ translates of $V_{\gamma_k'}$ together with $o\paren{\mathrm{Area}(\gamma_l)}$ additional plaquettes. Then by the FKG inequality we have
\begin{align*}
    &\limsup_{\set{\gamma_l} \in \mathcal{E}} \frac{-\log\paren{\mu_{\Z^d,p,q}\paren{V_{\gamma_l}}}}{\mathrm{Area}(\gamma_l)}\\
    &\qquad\leq \liminf_{\set{\gamma_k'} \in \mathcal{E}} \limsup_{\set{\gamma_l} \in \mathcal{E}}  \frac{1}{\mathrm{Area}(\gamma_l)}\floor{\frac{\mathrm{Area}(\gamma_l)}{\mathrm{Area}(\gamma_k')}} \paren{-\log\paren{\mu_{\Z^d,p,q}\paren{V_{\gamma_k'}}} + o\paren{\mathrm{Area}(\gamma_l)}}\\
    &\qquad= \liminf_{\set{\gamma_l} \in \mathcal{E}} \frac{-\log\paren{\mu_{\Z^d,p,q}\paren{V_{\gamma_l}}}}{\mathrm{Area}(\gamma_l)} + o\paren{1}\,.
\end{align*}
\end{proof}

\begin{proof}[Proof of Lemma~\ref{lemma:sharpconstant2}]
By extremality of free boundary conditions
$$\mu^{\bf{f}}_{\Lambda_N,p,q}\paren{V_{\gamma_N}^{\Lambda}}\leq \mu^{\mathbf{f}}_{\Z^d,p,q}\paren{V_{\gamma_N}}$$
so
$$\liminf_{l\rightarrow\infty} \frac{-\log\paren{\mu^{\bf{f}}_{\Lambda_N,p,q}\paren{V_{\gamma_N}}}}{\mathrm{Area}(\gamma_N}\geq \tau_{p,q}'\,.$$ 

Since  
\[\mu^{\mathbf{f}}_{\Z^d,p,q} = \lim_{N \to \infty} \mu^{\mathbf{f}}_{\Lambda_N,p,q}\]
as a monotone limit, for any $\epsilon>0$ there exists a $D = D\paren{k}>0$ so that if $\Lambda$ is any sufficiently large box and $\gamma_k'$ is a translate of $\gamma_k$ at least distance $D$ away from the boundary of $\Lambda,$

$$0<\log\paren{\mu^{\mathbf{f}}_{\Z^d,p,q}\paren{V_{\gamma_k}}}-\log\paren{\mu^{\mathbf{f}}_{\Lambda,p,q}\paren{V_{\gamma_k'}}} <\epsilon\,.$$
 
We may tile $r\paren{\gamma_N}$ with $m=m\paren{k}$ translates of $r\paren{\gamma_k}$ at least distance $D$ away from the boundary of $\Lambda_N$ (call them $r_1,\ldots,r_m$) with the exception of $o\paren{\mathrm{Area}(\gamma_N)}$ $i$-plaquettes (call the set of such plaquettes $T.$) Notice that
\[\partial r\paren{\gamma}=\sum_{j=1}^m \partial r_m+\sum_{\sigma \in T} \partial \sigma\]
when the chains are oriented appropriately. It follows that $V_{\gamma_N}$ is implied by $m$ translates of $V_{\gamma_k}$ together with $o\paren{\mathrm{Area}(\gamma_N)}$ additional plaquettes. 

It follows from the FKG inequality that
\begin{align*}
-\log\paren{\mu^{\mathbf{f}}_{\Lambda_N,p,q}\paren{V_{\gamma_N}}}&\leq \sum_{j=1}^m -\log\paren{\mu^{\mathbf{f}}_{\Lambda_N,p,q}\paren{V_{\partial r_j}}}+o\paren{\mathrm{Area}(\gamma_N)}\\
&\leq -m\log\paren{\mu^{\mathbf{f}}_{\Z^d,p,q}\paren{V_{\gamma_k}}}+m\epsilon+o\paren{\mathrm{Area}\paren{\gamma_N}}\,.
\end{align*}

Now fix $\epsilon'>0$ and let $k$ be large enough so that 
\[\abs{\frac{m\mathrm{Area}\paren{\gamma_k}}{\mathrm{Area}\paren{\gamma_k}}\frac{-\log\paren{\mu^{\mathbf{f}}_{\Z^d,p,q}\paren{V_{\gamma_k}}}}{\mathrm{Area}(\gamma_k)} - \tau_{p,q,\mathbf{f}}'} \leq \epsilon'\]
for all sufficiently large $N.$ 
Then
\begin{align*}
    \frac{-\log\paren{\mu^{\mathbf{f}}_{\Lambda_N,p,q}\paren{V_{\gamma_N}}}}{\mathrm{Area}(\gamma_N)}
   & \leq \frac{m\mathrm{Area}\paren{\gamma_k}}{\mathrm{Area}\paren{\gamma_N}}\frac{-\log\paren{\mu^{\mathbf{f}}_{\Z^d,p,q}\paren{V_{\gamma_k}}}}{\mathrm{Area}(\gamma_k)}+\epsilon+o\paren{1}\\
    &\leq  \tau_{p,q,\mathbf{f}}' +\epsilon + \epsilon'+o\paren{1}\,.
\end{align*}
Since $\epsilon$ and $\epsilon'$ were arbitrary, it follows that 
 \[\limsup_{N\rightarrow\infty} \frac{-\log\paren{\mu^{\mathbf{f}}_{\Lambda_N,p,q}\paren{V_{\gamma_N}}}}{\mathrm{Area}(\gamma_N)}\leq  \tau_{p,q,\mathbf{f}}' \,,\]
 and so
 \[\lim_{l\rightarrow\infty} \frac{-\log\paren{\mu^{\mathbf{f}}_{\Lambda_N,p,q}\paren{V_{\gamma_N}}}}{\mathrm{Area}(\gamma_N)}=  \tau_{p,q,\mathbf{f}}'\,.\]
   
\end{proof}

\begin{proof}[Proof of the Area Law]
Let $p<p^*\paren{p_{\mathrm{surf}}(q)}$ and let $\mu^{\xi}_{\Z^d,p,q}$ be an infinite volume PRCM measure. By the result of~\cite{grimmett1995stochastic} we may choose $p<p'<p^*\paren{p_{\mathrm{surf}}(q)}$ be such that the infinite volume random cluster model with parameters $p^*\paren{p'}$ and $q$ is unique. Then by Lemma~\ref{lemma:sharpconstant2} and Proposition~\ref{prop:dualtension}, 
\[\tau_{p',q,\xi}' = \tau_{p,q}'' = \tau_{p^*\paren{p'},q}\,.\]
In particular, since $p'<p^*\paren{p_{\mathrm{surf}}(q)},$ we have
\[\tau_{p,q,\xi}' \geq \tau_{p',q,\xi}' > 0\,.\]
Finally, applying Lemma~\ref{lemma:sharpconstant} shows that taking $\oldconstant{const:prcm1} = \tau_{p,q,\xi}'$ gives the desired result.
\end{proof}

\section{The Perimeter Law Regime}\label{sec:perimeterlaw}
We show that, in the supercritical regime, a perimeter law holds for $(d-2)$-cycles obtained as the boundaries of connected, hyperplanar regions of $\Z^d.$ For a set $X$ that is the union of $i$-dimensional plaquettes, write $\rho_X=\sum_{\sigma \in X} \sigma$ where the sum is taken over the (positively oriented) plaquettes $\sigma$ that compose $X.$ In this section, $\gamma$ will be a $(d-2)$-dimensional cycle of the form $\rho_{\partial X}$ where $X$ is a connected union of $(d-1)$-dimensional plaquettes $\set{\sigma_1,\ldots,\sigma_N}$ contained in a hyperplane of $\Z^d.$  We may assume without loss of generality that $\gamma$ is contained in  $\set{x_d=0}.$

 The proof of the perimeter law is not substantially different than that for independent plaquette percolation~\cite{aizenman1983sharp}, but we include it here for completeness. We provide more detail in the proof of the key geometric argument --- phrasing it in the language of homology --- which is a good warm-up for what follows. Complementarily to the area law section, it is enough to give a perimeter law bound for the PRCM with free boundary conditions. For convenience, we will denote the dual (wired) random-cluster measure by $\mu^{\bullet,\mathbf{w}}_{\Z^d,p^*} = \mu^{\mathbf{w}}_{\Z^d,p^*,q,1}.$

We require the following exponential decay result for the supercritical random-cluster model.

\begin{Theorem}[Duminil-Copin, Raoufi, Tassion~\cite{duminil2019sharp}]\label{thm:expdecay}
Fix $d \geq 2$ and $q \geq 1.$ Let $\theta(p^*) = \mu^{\bullet,\mathbf{w}}_{\Z^d,p^*}\paren{0 \leftrightarrow \infty}$ and let $p_c = p_c\paren{\Z^d,q}.$ Then 
\begin{itemize}
    \item there exists a $c > 0$ so that $\theta\paren{p^*} \geq c\paren{p^*-p_c}$ for any $p^* \geq p_c$ sufficiently close to $p_c;$
    \item for any $p^* < p_c,$ there exists a $b_{p^*}$ so that for every $n \geq 0,$
    \[ \mu^{\bullet,\mathbf{w}}_{\Lambda_n,p^*}\paren{0 \leftrightarrow \partial \Lambda_n} \leq \exp\paren{-b_{p^*} n}\,.\]
\end{itemize}
\end{Theorem}

Here, this theorem will be applied via the following corollary.

\begin{Corollary}\label{cor:C0}
Let $\mathcal{C}_0$ be the component of the origin in the classical random-cluster model on $\Z^d.$ Then, if $p^*< p_c\paren{\Z^d,q},$ 
\[\mathbb{E}_{\mu^{\bullet,\mathbf{w}}_{\Z^d,p^*}}\paren{\abs{\mathcal{C}_0}}<\infty\,.\]
\end{Corollary}
\begin{proof}
 See Theorem 5.86 in~\cite{grimmett2006random}. 
 \end{proof}

Let $p^* < p_c\paren{\Z^d,q}.$  We start by demonstrating that the positive $\vec{e}_d$-axis is disconnected from the hyperplane $W=\set{\vec{e}_d=-1/2}$ with positive probability in $Q$. Let $K = \set{\paren{1/2,\ldots,1/2,1/2+z}:z\in \Z^{\geq 0}}$ and $K_h = K \cap \{x_d \geq 1/2 + h\}.$ Denote by $F_h$ the event that $W$ is connected to $K_h$ in $Q.$ 

\begin{Proposition}\label{prop:axisdisconnect}
    If $p^* < p_c\paren{\Z^d,q}$ then $\mu^{\bullet,\mathbf{w}}_{\Z^d,p^*}\paren{F_0} < 1.$
\end{Proposition}

\begin{proof}
By translation invariance, we have that
    \[\sum_{v \in K}\sum_{w \in W} \mu^{\bullet,\mathbf{w}}_{\Z^d,p^*}(v \xleftrightarrow[]{} w) = \mathbb{E}_{\mu^{\bullet,\mathbf{w}}_{\Z^d,p^*}}\paren{\abs{\mathcal{C}_0}} < \infty\,,\]
    using Corollary~\ref{cor:C0}. Therefore
    \[\mu^{\bullet,\mathbf{w}}_{\Z^d,p^*}\paren{F_h} \leq \sum_{v \in K_h}\sum_{w \in W} \mu^{\bullet,\mathbf{w}}_{\Z^d,p^*}(v \xleftrightarrow[]{} w) \xrightarrow[]{h \to \infty} 0\,,\]
    since the sum is the tail of a convergent series. Let $h$ be large enough so that $\mu^{\bullet,\mathbf{w}}_{\Z^d,p^*}\paren{F_h} < 1/2.$ Notice that if the entire edge boundary of $K \setminus K_h$ is omitted, then $K \setminus K_h$ is disconnected from $W.$ Thus
    \[\mu^{\bullet,\mathbf{w}}_{\Z^d,p^*}\paren{K \setminus K_h \xleftrightarrow[]{} W} \leq 1-\paren{\frac{p}{q}}^{4h+2}\,.\]
    Then, by the FKG inequality,
    \begin{align*}
        \mu^{\bullet,\mathbf{w}}_{\Z^d,p^*}\paren{\neg F_0} &\geq \mu^{\bullet,\mathbf{w}}_{\Z^d,p^*}\paren{\neg F_h \bigcap \neg\paren{K \setminus K_h \xleftrightarrow[]{} W}}\\
        &\geq \mu^{\bullet,\mathbf{w}}_{\Z^d,p^*}\paren{\neg F_h} \mu^{\bullet,\mathbf{w}}_{\Z^d,p^*}\paren{\neg \paren{K \setminus K_h \xleftrightarrow[]{} W}}\\
        &\geq \frac{1}{2}\paren{\frac{p}{q}}^{4h+2} > 0\,.
    \end{align*}
\end{proof}

Let $\sigma_{i_1},\ldots,\sigma_{i_M}$ be the subset of the $(d-1)$-plaquettes $\sigma_i$ which share a $(d-2)$-dimensional face with $\gamma.$ For each $1 \leq i \leq N,$ let $a_i$ be the center of $\sigma_i$ and let $F^i=F_0+a_i$ (that is, $F^i=\set{K+a_i  \xleftrightarrow[]{} W}$).

\begin{figure}
     \centering
     \begin{subfigure}[b]{0.45\textwidth}
         \centering
         \includegraphics[width=\textwidth]{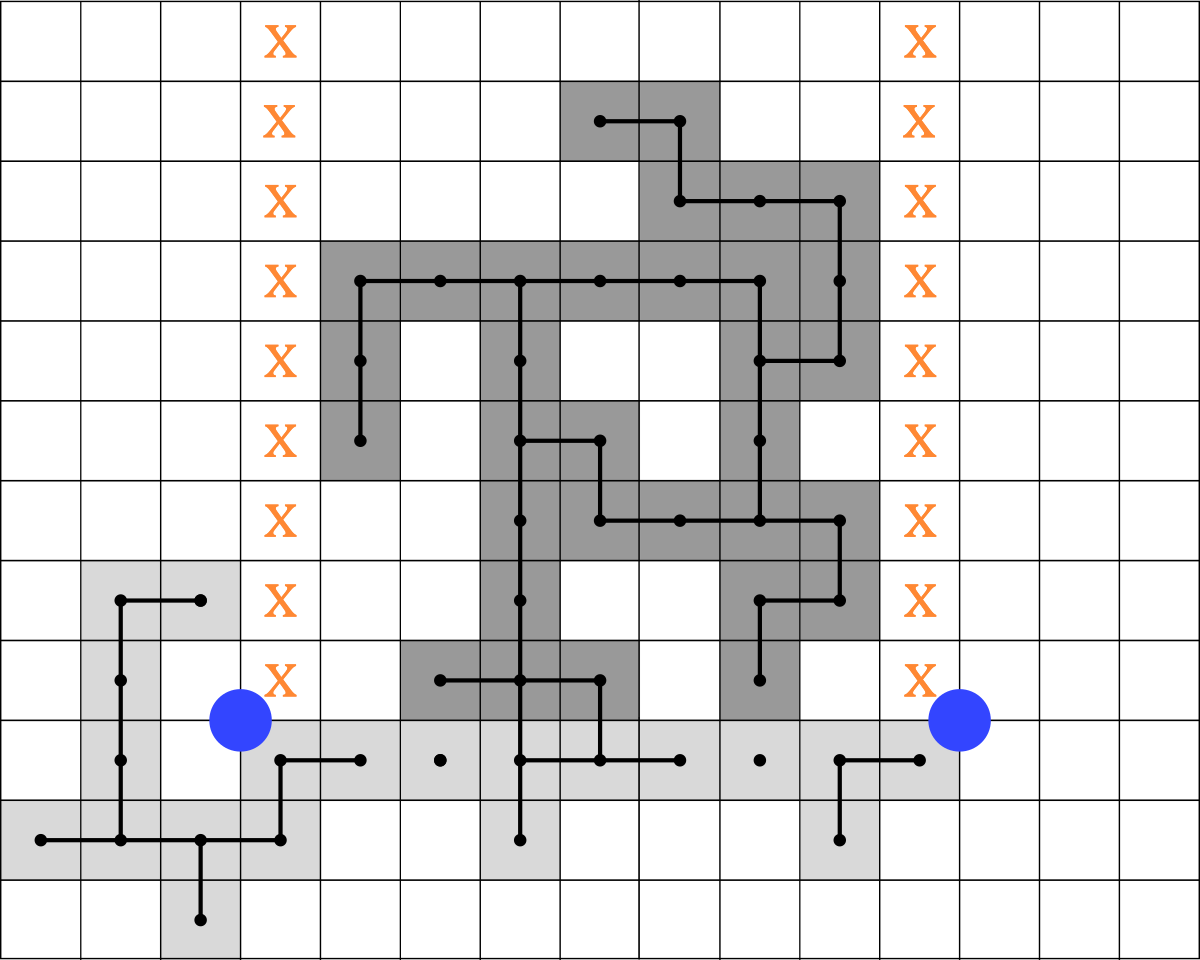}
         \subcaption[]{}
         \label{fig:perimeter1}
     \end{subfigure}
     \hfill
     \begin{subfigure}[b]{0.45\textwidth}
         \centering
         \includegraphics[width=\textwidth]{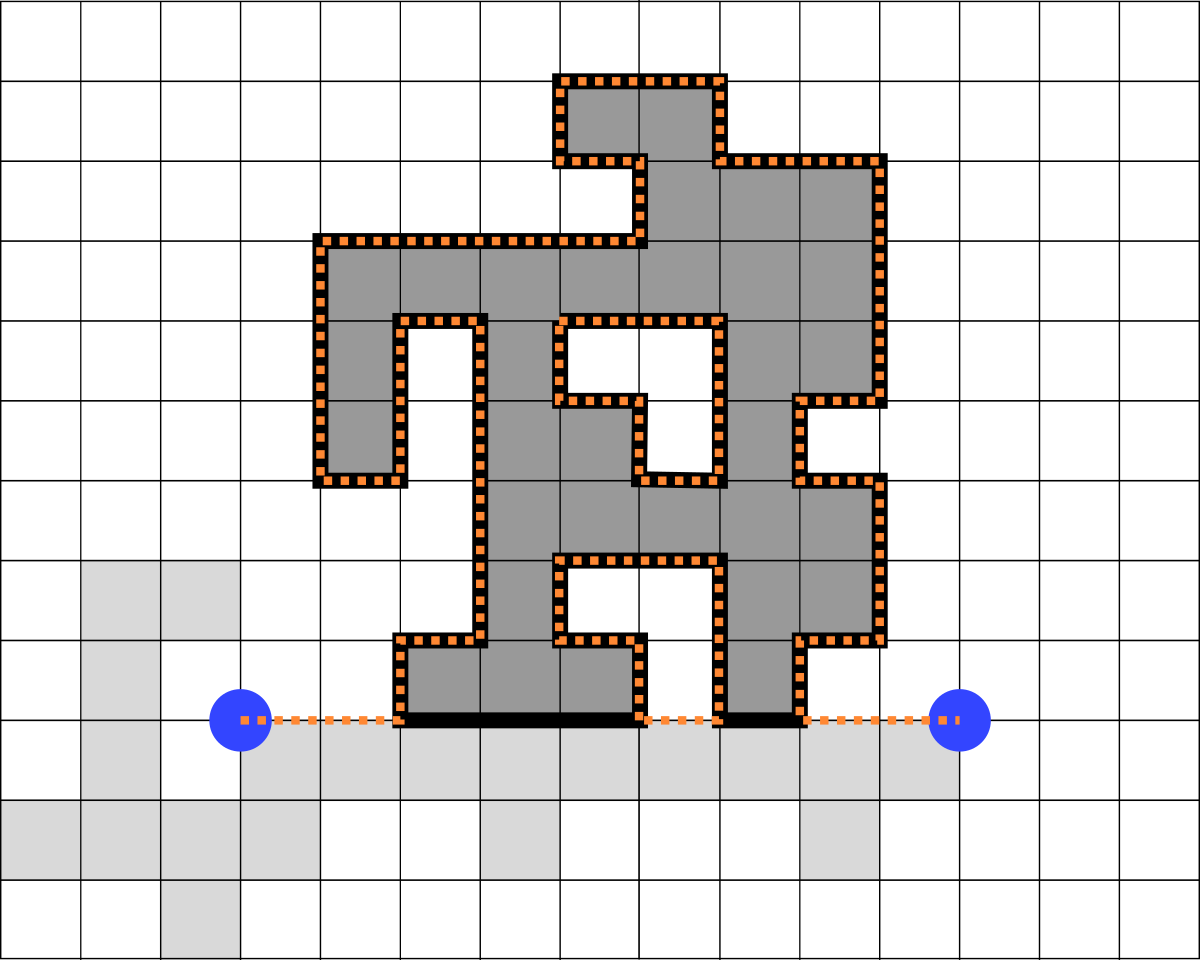}
        \subcaption[]{}

         \label{fig:perimeter2}
     \end{subfigure}
            \caption{The construction in the proof of Lemma~\ref{lemma:perimetertopological} for the case $d=2$ (or a cross-section of it in higher dimensions). In both figures $\gamma$ is shown by the large blue dots, and  the cubes $\tau_1,\ldots,\tau_J$ are colored gray with the subset $\tau_{j_1},\ldots, \tau_{j_K}$ given a darker shade. In (A), $C$ is depicted by the small black dots connected by bonds. Note that the cubes marked with the orange $X$'s are excluded from $c$ by the event $\bigcap_{j \leq M} \neg F^{i_j}.$ (B) shows the support of the chains $\alpha_1$  (the dotted orange paths) and $\alpha_2$ (the thick black paths). This figure was adapted from Figure 3 of~\cite{aizenman1983sharp}.}
        \label{fig:perimeter}
\end{figure}

\begin{Lemma}\label{lemma:perimetertopological}
\[\bigcap_{j \leq M} \neg F^{i_j} \subset V_{\gamma}^{\mathrm{fin}}\paren{\Z}\]
\end{Lemma}
\begin{proof}
Assume that the event 
\begin{equation}
    \label{eq:ptop}
\bigcap_{j \leq M} \neg F^{i_j}
\end{equation}
occurs. Let $\mathcal{C}_i$ denote the connected component of $a_i$ in $Q$ and let $\hat{\mathcal{C}}_i$ be the closure $\mathcal{C}_i$ in $\R^d\setminus P,$ for $i\leq 1 \leq N.$ Set $\mathcal{C}=\cup_{i=1}^N \mathcal{C}_i.$ By Lemma~\ref{prop:PtoQretract}, $\cup_{i=1}^N \hat{\mathcal{C}}_i$ is the union of the $d$-cubes dual to the sites of $\mathcal{C}.$ There are finitely many such cubes because $p<p_c\paren{\Z^d}$; denote them by  $\tau_1,\ldots,\tau_J.$

We may write 
\[\partial\sum_{i=1}^J\tau_j=\alpha_0+\alpha_1\]
where $\alpha_1$ and $\alpha_0$ are supported on $\paren{X\setminus \gamma} \times \left[0,\infty\right)$ and the complement of that set, respectively. Any plaquette in the support of  $\partial \sum \tau_j$ is contained in $P,$ so $\alpha_1 \in C_{d-1}\paren{P;\;\Z}.$ We will show that $\partial \alpha_1= (-1)^{d-1} \gamma.$

Let
\[\set{\tau_{j_k}}_{k=1}^K=\set{\tau:\in \set{\tau_1,\ldots,\tau_J}:\tau\subset X\times \left[0,\infty\right)}\]
We will compute the difference between $\alpha_1$ and $\alpha_2\coloneqq \partial \sum_{j=1}^K \tau_{j_k}.$ For $j\in \set{1,2},$ write
\[\alpha_j=\sum_{\sigma}a_{\sigma}^j\sigma\]
where the sum is taken over the $(d-1)$-plaquettes of $\Z^d.$ Each $(d-1)$-plaquette $\sigma$  is contained in exactly two $d$-cubes, so $a_{\sigma}^j\in\set{-1,0,1}$  for $j=1,2.$ In addition, if $a_{\sigma}^1\neq 0$ and $a_{\sigma}^2\neq 0$ then  $a_{\sigma}^1=a_{\sigma}^2.$

Let $\sigma$ be a $(d-1)$ plaquette of $\Z^d.$ If $a_{\sigma}^1\neq 0$ and $a_{\sigma}^2=0$ then $\sigma$ is dual to an edge connecting a site of $X\times \left[0,\infty\right)\setminus \mathcal{C}$ with a site of $\mathcal{C}$ outside of that cylinder. As we excluded plaquettes contained in $\gamma \times \left[0,\infty\right)$ from the support of $\alpha_1,$ it follows that $\sigma$ is one of the plaquettes $\set{\sigma_1,\ldots,\sigma_N}$ from the definition of $\gamma.$ On the other hand, if $a_{\sigma}^2\neq 0$ and $a_{\sigma}^1=0$ then $\sigma$  corresponds to an edge between two sites of $\mathcal{C},$ one inside the cylinder and one outside of it. By the assumption in (\ref{eq:ptop}), $\sigma\in \set{\sigma}_{i=1}^N.$ Morever, each $\sigma_i$ must fall into one of these two classes. Therefore
\[a_{\sigma_i}^1\neq 0 \iff a_{\sigma_i}^2= 0\]
for $i=1,\ldots,N.$ 

 In addition, by (\ref{eq:boundaryOperator}), $\sigma_i$ either appears with the sign $(-1)^{d-1}$ in $\alpha_1$ or it appears with the sign $(-1)^{d}$ in $\alpha_2.$ Thus
\[\alpha_1=\alpha_2+(-1)^{d-1}\sum_{i=1}^N\sigma_i\,,\]
and, recalling the definitions of $\alpha_2$ and the plaquettes $\sigma_i,$
\[\partial \alpha_1= \partial\circ \partial\paren{\sum_{k=1}^K\tilde{\tau}_k}+(-1)^{d-1} \partial\paren{\sum_{i=1}^N \sigma_i}=(-1)^{d-1} \gamma\,.\]
Therefore $\gamma$ is null-homologous in $P.$ 
\end{proof}

\begin{Proposition}\label{prop:perimeter}
    Let $p^* < p_c\paren{\Z^d,q},$ let $m\in \N$, and let $V_{\gamma}=V_{\gamma}^{\mathrm{inf}}\paren{m}$ or $V_{\gamma}=V_{\gamma}^{\mathrm{fin}}\paren{m}.$ Then there is an $0 < \beta < \infty$ so that 
    \[\mu^{\mathbf{f}}_{\Z^d,p,q,d-1}\paren{V_{\gamma}} \geq \exp\paren{-\beta \mathrm{Per}\paren{\gamma}}\,\]
    for any boundary $\gamma$ of a $(d-1)$-dimensional connected, hyperplanar region of $\Z^d.$
\end{Proposition}

\begin{proof}
By Corollary~\ref{cor:vgammacontain}, it suffices to show the statement for $V_{\gamma}=V_{\gamma}^{\mathrm{fin}}\paren{\mathbb{Z}}.$   Using the FKG inequality, Proposition~\ref{prop:axisdisconnect} and Lemma~\ref{lemma:perimetertopological}, we have
\begin{align*}
    \mu^{\mathbf{f}}_{\Z^d,p,q,d-1}(V_\gamma) &\geq \mu^{\bullet,\mathbf{w}}_{\Z^d,p^*}\paren{\bigcap_{i \leq N} \neg F^i}\\
    &\geq \mu^{\bullet,\mathbf{w}}_{\Z^d,p^*}\paren{\neg\paren{F_0}}^{\mathrm{Per}(\gamma)}\\
    &\geq \exp{(-\beta\mathrm{Per}(\gamma))}\,,
\end{align*} 
for some $\beta>0.$
\end{proof}

\section{A Sharp Constant for the Perimeter Law}
\label{sec:sharpnessperimeter}

In this section, we gain a finer understanding of the supercritical asymptotics of Wilson loop variables corresponding the boundaries of $(d-1)$-dimensional boxes of $\Z^d.$ We may assume without loss of generality that these boxes are of the form $\brac{0,M_1}\times \ldots \times \brac{0,M_{d-1}} \times 0.$ This section is the only place where we require control over how the dimensions of the boxes grow to $\infty.$  Recall from the introduction that a family of $k$-dimensional boxes $\set{r_{l}}$ is suitable if its  $k$ dimensions diverge to $\infty$ and if $m\paren{r_{l}}=\omega\paren{\log\paren{M\paren{r_l}}},$ where $M\paren{r}$ and $m\paren{r}$ denote the largest and smallest dimensions. When $\set{r_{l}}$ is a suitable, we say that $\set{\gamma_l}=\set{\partial r_{l}}$ is a \emph{suitable} family of $(k-1)$-dimensional rectangular boundaries. This hypothesis allows us to interpolate between boundary conditions for the PRCM on a box, and is unnecessary in the case $q=1$ of Bernoulli plaquette percolation.

\begin{Theorem}
\label{theorem:sharp_perimeter}
Let $p>p^{*}\paren{p_c\paren{\Z^d,q}}$ and $m\in \N.$ There is a constant $0 < \oldconstant{const:prcm2} < \infty$ so that if $\set{\gamma_l}$ is a suitable family of $(d-2)$-dimensional rectangular boundaries, $\mu_{\Z^d,p,q}=\mu^{\xi}_{\Z^d,p,q,d-1}$ is any infinite volume random-cluster measure, and $V_{\gamma}=V_{\gamma}^{\mathrm{fin}}\paren{m}$ or  $V_{\gamma}=V_{\gamma}^{\mathrm{inf}}\paren{m}$ then
\[-\lim_{l\rightarrow\infty}\frac{\mu_{\Z^d,p,q}\paren{V_{\gamma}}}{\mathrm{Per}\paren{\gamma_l}}=\oldconstant{const:prcm2}\,.\]
\end{Theorem}

The events are increasing, so by  Proposition~\ref{prop:extremal}, it suffices to show that there exists such a constant \oldconstant{const:prcm2} so that 
\[\limsup_{l\rightarrow\infty}-\frac{\mu_{\Z^d,p,q}^{\mathbf{f}}\paren{V_{\gamma_l}^{\mathrm{fin}}\paren{m}}}{\mathrm{Per}\paren{\gamma_l}}\leq \oldconstant{const:prcm2}\]
and 
\[\limsup_{l\rightarrow\infty}-\frac{\mu_{\Z^d,p,q}^{\mathbf{w}}\paren{V_{\gamma_l}^{\mathrm{inf}}\paren{m}}}{\mathrm{Per}\paren{\gamma_l}}\geq \oldconstant{const:prcm2}\,.\]

For the remainder of this section, fix an infinite volume PRCM $\mu_{\Z^d,p,q}=\mu_{\Z^d,p,q,d-1}^{\xi}$ and denote the dual RCM by $\mu^{\bullet}_{\Z^d,p^*,q}=\mu_{\Z^d,p^*\paren{p},q,1}^{\xi^{\bullet}}.$ The outline of our proof is similar to that of Theorem 3.9 of~\cite{aizenman1983sharp}, though we must modify it to work for higher dimensional plaquette percolation and to handle the dependence between disjoint plaquette events. A non-expert reader may benefit from reading that account first, as the arguments are simpler. 

We split the proof into three parts. In Section~\ref{subsec:interpolate}, we show a technical lemma interpolating between $\mu_{\Z^d,p,q}$ and the plaquette random-cluster measure with wired boundary conditions in a box. Next, we construct a plaquette event that precludes the existence of a dual loop that links with $\gamma$ in Section~\ref{subsec:geometric}, and a related event which is implied by $V_{\gamma}^{\mathrm{fin}}.$ We conclude in Section~\ref{subsec:sharp_perimeter} by ``sandwiching'' the probability of $V_{\gamma}$ between the probabilities of these two events, thereby demonstrating the existence of the constant $c.$

In several places in this section, we will use the following notion of a box crossing.
\begin{Definition}\label{defn:rectangle_crossing}
Let $r$ be a box in $\Z^d.$ The $i$-box crossing event for $r,$ denoted $R^{\square}_i\paren{r},$ is the event that there is a hypersurface of plaquettes contained in the interior of $r$ which separates the two faces of $r$ orthogonal to the $\vec{e}_i$-axis. 
\end{Definition}
Here (and later), a hypersurface of plaquettes contained in a box $r$ will be a collection of plaquettes $S$ so that $\partial \rho_S$ is supported on $\partial r.$

\subsection{An Interpolation Lemma}\label{subsec:interpolate}

Let $r=\brac{0,M_1}\times \ldots \brac{0,M_d}$ be a box in $\Z^d,$ let $r^L$ be the enlarged box $r=\brac{-L,M_1+L}\times \ldots \brac{-L,M_d+L},$ and let $A$ be an increasing event depending only on the edges of $r.$ We will compare the asymptotics of $\mu_{\Z^d,p}\paren{A}$ with those of $\mu_{r^L,p}^{\mathbf{w}}\paren{A}$ in the supercritical regime. Towards this end, we will show that, when the dimensions of the boxes are grown appropriately, there is a high probability that $\partial r^L$ is separated from $r$ by a hypersurface of plaquettes. First, we prove a technical lemma.

\begin{Lemma}\label{lemma:crossing}
    Let $p>p^{*}\paren{p_c\paren{q}}.$ There exists a $b_p>0$ so that 
    \[\mu_{\Z^d,p,q}\paren{R^{\square}_d\paren{r}}\geq 1- e^{-b_p M_d} \prod_{j =1}^{d-1} M_j\,.\]
\end{Lemma}
\begin{proof}
    Let $D^+,D^-$ be the top and bottom faces of $r,$ respectively (with respect to the $d$-th coordinate direction). By definition,
    \[\mu_{\Z^d,p}\paren{\neg R^{\square}_d\paren{r}}=\mu_{\Z^d,p^*}^{\bullet}\paren{D^{+} \xleftrightarrow[]{Q \cap r } D^{-}}\]
    where $\mu_{\Z^d,p^*}^{\bullet}$ is the dual $1$-dimensional random-cluster model and $p^*=p^*\paren{p}.$ 
    
    Let $v$ be a dual vertex contained in $D^+-\frac{1}{2}\vec{e}_d.$ Then  
    \begin{align*}
    \mu_{\Z^d,p^*,q}^{\bullet}\paren{v\xleftrightarrow[r^{\bullet}]{}D^-+\frac{1}{2}\vec{e}_d} \leq & \mu_{\Z^d,p^*,q}^{\bullet}\paren{v\xleftrightarrow[\Lambda_{M_d}\paren{v}]{}\partial \Lambda_{M_d}\paren{v}}\\
    \leq & \mu_{\Lambda_{M_d}\paren{v},p^*,q}^{\bullet,\mathbf{w}}\paren{v\xleftrightarrow[\Lambda_{M_d}\paren{v}]{}\partial \Lambda_{M_d}\paren{v}}\\
    \leq & e^{-b_p M_d}
    \end{align*}
    as a consequence of Theorem~\ref{thm:expdecay}. 
As 
\[\neg R^{\square}_d\paren{r}=\bigcup_{v\in D^+-\frac{1}{2}\vec{e}_d}  v\xleftrightarrow[r]{}D^-+\frac{1}{2}\vec{e}_d\,,\]
the desired statement follows by the union bound. 
\end{proof}

\begin{Corollary}
\label{cor:xi}
Let $\Xi_{r,L}$ be the event that  $\partial r^L$ is separated from $r$ by a hypersurface of plaquettes contained in the annular region  $ r^L \setminus r.$ Then, if $p>p^{*}\paren{p_c\paren{q}},$
\[\mu_{\Z^d,p,q}\paren{\Xi_{r,L}} \geq  \prod_{i=1}^{d} \paren{1-e^{-b_p L}\prod_{j\neq i} M_j}^2 \,.\]
\end{Corollary}
\begin{proof}
We define a box crossing event for each $(d-1)$-dimensional face $D$ of $r.$ For the special case $D=\brac{0,M_1}\times \ldots \brac{0,M_{d-1}}\times \set{0},$ set $r_{D,L}=\brac{-L,M_1+L}\times \ldots \brac{-L,M_{d-1}+L}\times \brac{-L,-M_d}$ and $\hat{R}^{\square}\paren{D,L}$ to be the event $R^{\square}_d\paren{r_{D,L}}.$ More generally, define $\hat{R}^{\square}\paren{D,L}$ by symmetry.  Then $\Xi_{r,L}$ is implied by the occurrence of the $2d$ events $R^{\square}_d\paren{r_{D,L}}$, one for each face of $r.$ Since these events are increasing, the statement follows from the preceding lemma by the FKG inequality.
\end{proof}



\begin{Proposition}\label{prop:interpolate}
Let $A$ be an increasing event that depends only on the plaquettes of $r.$ If $p>p^{*}\paren{p_c\paren{q}}$ then
\[ \mu_{\Z^d,p,q}\paren{\Xi_{r,L}}\mu_{r^L,p}^{\mathbf{w}}\paren{A}\leq \mu_{\Z^d,p,q}\paren{A}\leq \mu_{r^L,p}^{\mathbf{w}}\paren{A}\,.\]
\end{Proposition}

\begin{proof}
A standard application of Holley's inequality (Theorem 2.3 in~\cite{grimmett2006random}) yields that there is a coupling between restriction of the conditional measure $\paren{\mu_{\Z^d,p,q} \middle |  \Xi_{r,L}}$ to $r$ and the restriction of $\mu_{r_L,p,q}^{\mathbf{w}}$ to $r$ such that the former contains the latter almost surely. This implies the first inequality. The second inequality is also an immediate consequence of Holley's inequality.
\end{proof}

\subsection{Geometric Lemmas}\label{subsec:geometric}
In this section, we will prove several geometric results that will be useful for the proof of Theorem~\ref{theorem:sharp_perimeter}. First, we show that box crossings behave nicely under intersections. For convenience, we state the following lemma for $d$-box crossings. The corresponding statement for $i$-box crossings follows by symmetry.
\begin{Lemma}\label{lemma:crossingsubset}
    Let $r_1$ and $r_2$ be boxes of the form $\brac{0,M_1}\times \ldots \times \brac{0,M_d}$ and $\brac{N_1,M_1'}\times \ldots \times \times \brac{N_{d-1},M_{d-1}'}\times \brac{0,M_d},$ with $0\leq N_{i}<M_i'\leq M_i $ for $i=1,\ldots,{d-1}.$ That is, $r_2$ is contained in $r_1$ and has the same height as it. Then $R^{\square}_d\paren{r_1}\implies R^{\square}_d\paren{r_2}.$ 
\end{Lemma}

\begin{proof}
Let $D_1^+$ and $D_1^-$ be the two faces of $r_1$ orthogonal to $\vec{e}_d$ and let $D_2^+$ and $D_2^-$ be faces of $r_2$ orthogonal to $\vec{e}_d$ so that $D_2^+ \subset D_1^+$ and $D_2^- \subset D_1^-.$ Then if $R^{\square}_d\paren{r_2}$ does not occur, there is a dual path between $D_2^+$ and $D_2^-$ in $r_2.$ But then since $r_2 \subset r_1,$ such a path also connects $D_1^+$ from $D_1^-$ in $r_1,$ so $R^{\square}_d\paren{r_2}$ cannot occur. 
\end{proof}

We say that a set of plaquettes $J$ is a minimal witness for a crossing event $R^{\square}_i\paren{r}$ if there is no subset $J'\subset J$ which is also a witness for $R^{\square}_i\paren{r}.$ We now investigate properties of a minimal box crossing.

\begin{Lemma}\label{lemma:crossingtopology}
Let $r$ be a box and let $\hat{B}_i\paren{r}$ be the union of the $(d-1)$-faces of $r$ which are not orthogonal to $\vec{e}_i.$ Assume $R^{\square}_i\paren{r}$ occurs and that $J$ is a minimal collection of plaquettes witnessing that event. Then the map on homology $H_{d-2}\paren{\partial J;\; \Z}\rightarrow H_{d-2}\paren{\hat{B}_i\paren{r};\;\Z}\cong \Z$ induced by inclusion is an isomorphism. In particular, if $\rho_J=\sum_{\sigma\in J}\sigma$ then $\partial{\rho_J}$ generates $H_i\paren{\hat{B}_i\paren{r};\;\Z}\cong \Z$ and is null-homologous in $P.$   
\end{Lemma}

\begin{proof}
Let $D_i^+$ and $D_i^-$ be the two $(d-1)$-faces of $r$ orthogonal to $\vec{e}_i.$ Let $A$ be the union of $d$-cells of the connected component containing $D_i^-$ of $r \setminus J.$ Since $J$ is minimal, it cannot contain any plaquettes of $\partial r$ and it must contain all plaquettes of $\partial A$ that are not supported on $r.$ Thus, we can write $\partial A$ as the union of three disjoint sets of plaquettes 
\[\partial A = D_i^- \cup J \cup E\,,\]
where $E$ is a union of plaquettes contained in $\hat{B}_i\paren{r}.$ Recall that $\rho_Y$ is the sum of the positively oriented plaquettes composing $Y$. Thus 
\begin{align*}
    \partial \rho_{D_i^-} + \partial \rho_{E} + \partial \rho_{J} = \partial \partial \rho_{A} = 0\,.
\end{align*}
So since $0 = \brac{\partial \rho_{E}} \in H_{d-2}\paren{\hat{B}_i\paren{r};\;\Z},$ it follows that
\[\brac{\partial \rho_{D_i^-}} = -\brac{\partial \rho_{J}} \in H_{d-2}\paren{\hat{B}_i\paren{r};\;\Z}\,.\]
Now, $\brac{\partial \rho_{D_i^-}}$ is a generator for $H_{d-2}\paren{\hat{B}_i\paren{r};\;\Z},$ so $\brac{\partial \rho_{J}}$ is as well. Thus, the map $H_{d-2}\paren{\partial J;\; \Z}\rightarrow H_{d-2}\paren{\hat{B}_i\paren{r};\;\Z}\cong \Z$ is an isomorphism.
\end{proof}




\begin{figure}
    \centering
     \includegraphics[width=0.6\textwidth]{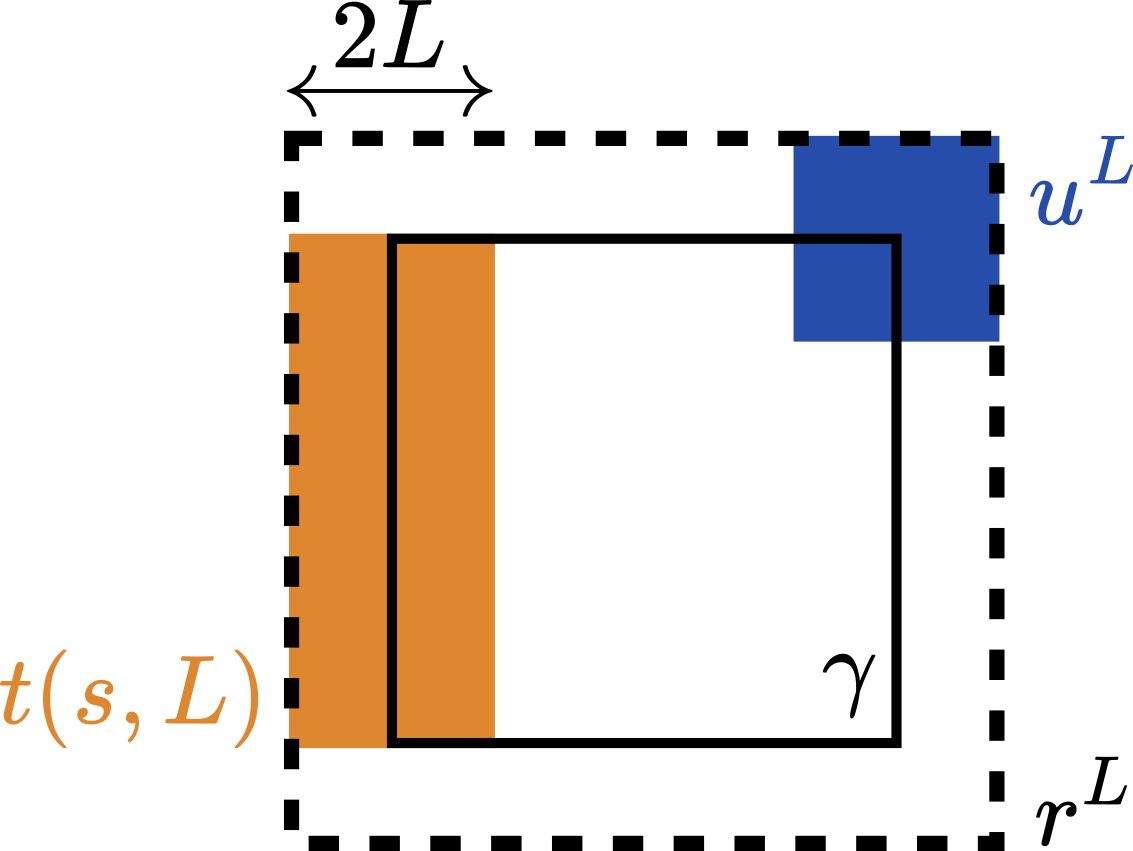}
        \caption{An illustration of some of the notation developed in this section, shown for $d=3$ in the plane containing $\gamma.$} 
    \label{fig:tubes}
\end{figure}

Given a box of the form $t=s'\times \brac{-L,L}^2$ of $\Z^d,$ let $s$ be the $(d-2)$-dimensional box $S'\times\set{0}^2.$ We say that $t=t\paren{s,L}$ is the tube around $s$ of width $L.$ Often $s$ will be a $(d-2)$-face of a rectangular boundary $\gamma.$ Denote by $C_t$ be the event that there a $(d-1)$-chain $\tau\in C_{d-1}\paren{P\cap t;\;G}$ so that $\partial \tau= \rho_{s}+\alpha$ where $\alpha$ is supported on $\partial t.$ See Figure~\ref{fig:CR}. Compare this with the definition of the corresponding event in Section 3(iii) of~\cite{aizenman1983sharp} when $d=3.$

 \begin{figure}
    \centering
     \includegraphics[width=0.2\textwidth]{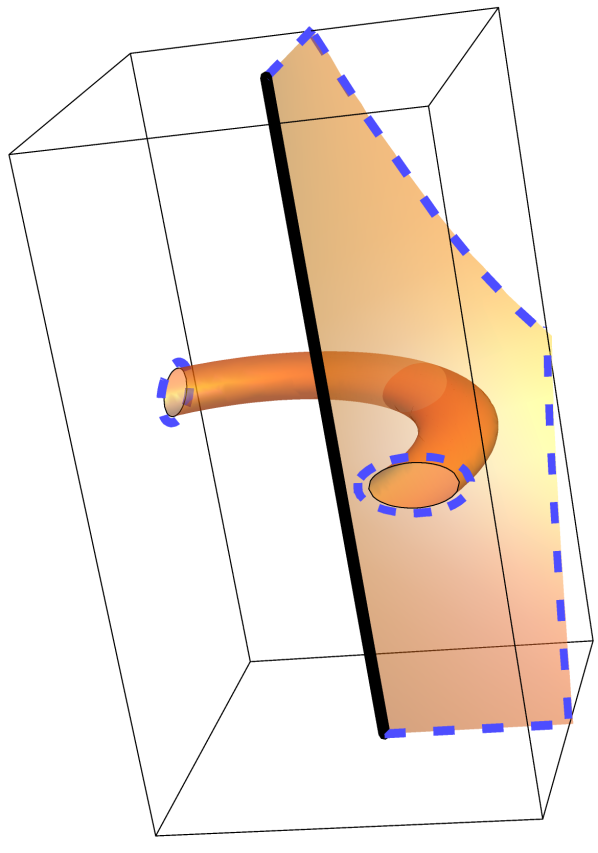}
     \qquad     
     \includegraphics[width=0.2\textwidth]{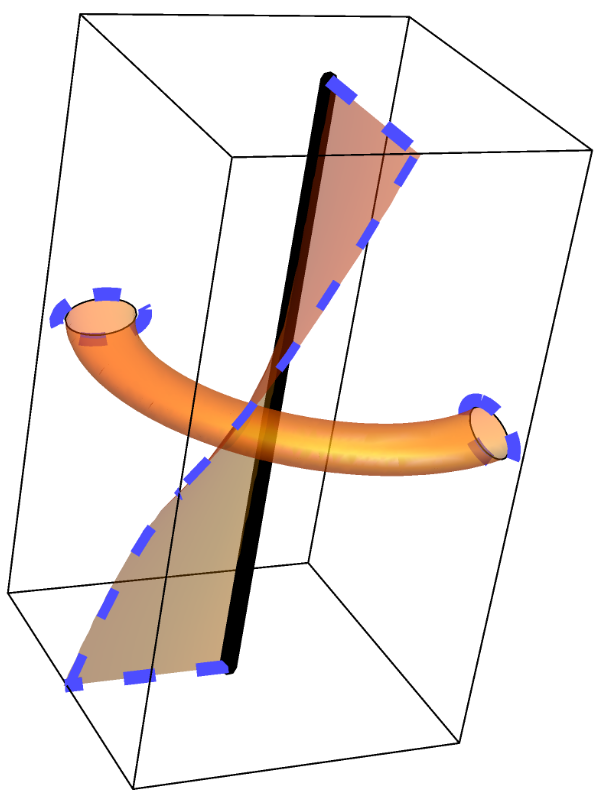}
    \qquad     
    \includegraphics[width=0.2\textwidth]{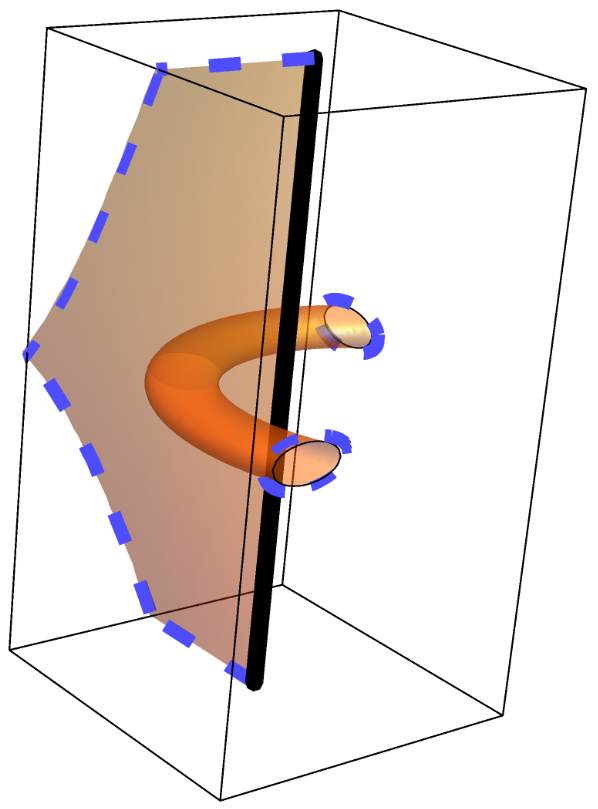}

         \caption{A witness for the event $C_t,$ shown from three different viewpoints. The support of $\tau$ is the orange hypersurface, the support of $\alpha$ is shown by a dotted blue curve, and $s$ is depicted with a thick black line. A handle is included to emphasize the possible complexity of set of plaquettes. See also Figure 13 of~\cite{aizenman1983sharp}.}
    \label{fig:CR}
\end{figure}

\begin{Lemma}\label{lemma:CR}
     Let $t,$ $s,$ and $s'$ be as in the previous paragraph, and suppose that $s''\subset s'$ is a $(d-2)$-dimensional box in $\Z^{d-2}\times \set{0}^2.$ Then if $t''=s''\times \brac{-L,L}^2,$ $C_t\implies C_{t''}.$  
\end{Lemma}
 
\begin{proof}
Assume that the event $C_t$ occurs. Then there exists a chain $\tau=\sum_{\sigma\subset t}a_{\sigma}\sigma$ so that $\partial \tau=\rho_{s}+\alpha$ where $\alpha$ is supported on the boundary of $t.$ If $\tau'=\sum_{\sigma \subset t''}a_{\sigma}\sigma$ then $\partial\tau'=\rho_{s}+\alpha'$ where $\alpha'$ is supported on the boundary of $t''.$
\end{proof}




Consider the four boxes $y_1=s'\times \brac{-L,-L/2}\times \brac{-L,L},$  $y_2=s'\times \brac{L/2,L}\times  \brac{-L,L},$  $y_3=s'\times  \brac{-L,L} \times \brac{-L,-L/2},$ and $y_4=s'\times  \brac{-L,L}\times \brac{L/2,L}$ that surround $s.$ Set 
\[D_t=R^{\square}_{d-1}\paren{y_1}\cap R^{\square}_{d-1}\paren{y_2}\cap R^{\square}_{d}\paren{y_3}\cap R^{\square}_{d}\paren{y_4}\,.\]
$D_t$ implies that $s$ is separated from the faces of $t$ parallel to it by a surface of plaquettes. Let $\tilde{B}$ to be the union of the four faces of $t$ that are parallel to $s.$

\begin{figure}
    \centering
     \begin{subfigure}[b]{2.12in}
         \centering
         \includegraphics[width=\textwidth]{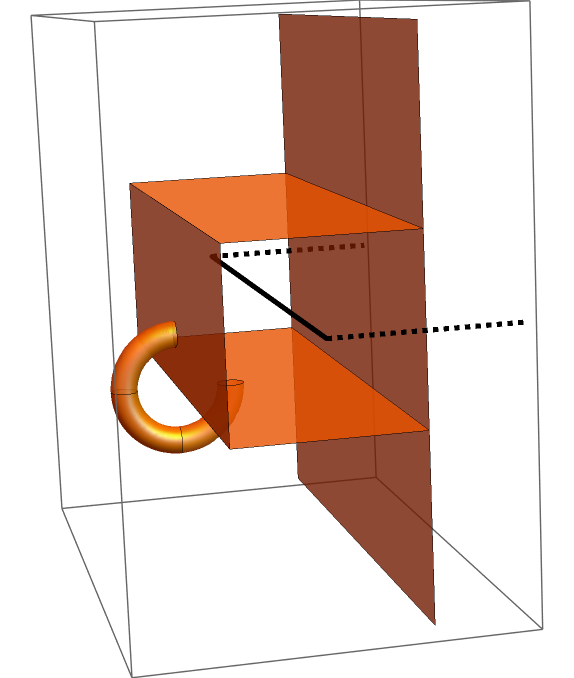}
         \subcaption[]{}
         \label{fig:S1}
     \end{subfigure}
    \qquad\qquad
     \begin{subfigure}[b]{1.38in}
         \centering
         \includegraphics[width=\textwidth]{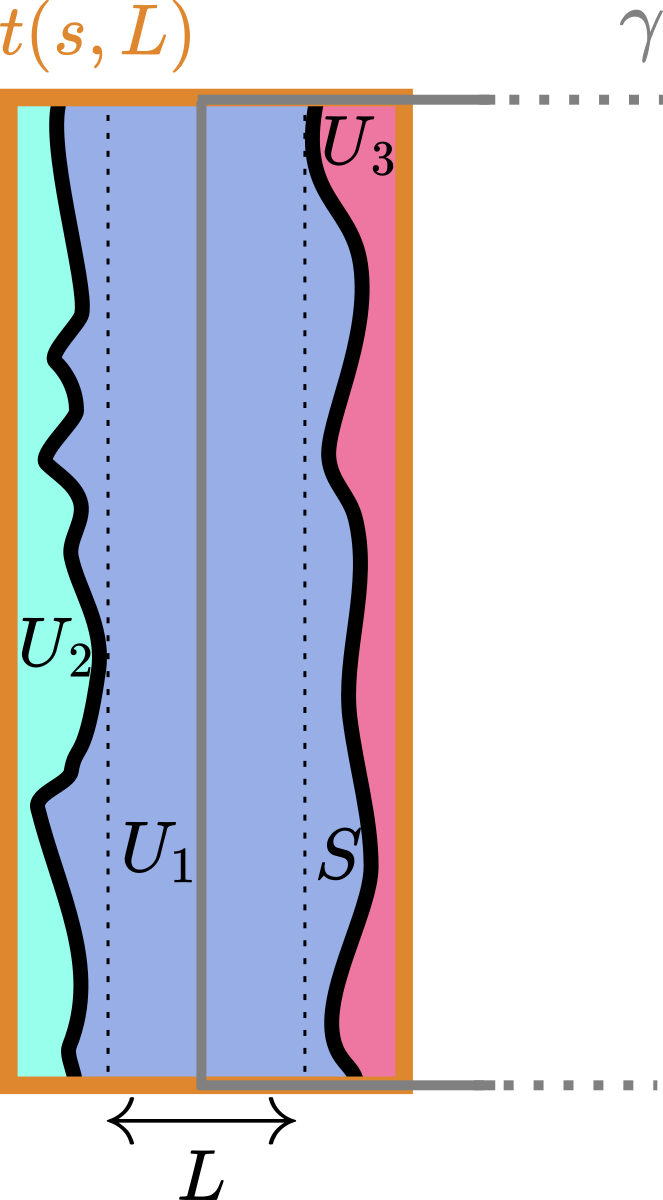}
        \subcaption[]{}
        \end{subfigure}
         \caption{(A) An example of the surface $S$ constructed in Lemma~\ref{lemma:S}, shown in orange. The surface includes a handle to emphasize the possible complexity of $S.$ $s$ is shown by a solid black line, and the neighboring parts of $\gamma$ by a dashed black line. (B) A cross section of the partition of the tube  $t\paren{s,L}$ by the hypersurface $S,$ in the plane containing $\gamma.$ $S$ is shown in black, and $\gamma$ in gray.}
    \label{fig:SPartition}
\end{figure}

\begin{Lemma}\label{lemma:S}
If $D_t$ occurs, then there is a connected hypersurface $S$ of plaquettes of $P\cap t$ that separates $t$ into three components $U_1=U_1(s), U_2=U_2(s),$ and $U_3=U_3(s)$ with the following properties.
\begin{itemize}
    \item $t\paren{s,L/2}$ is contained in $U_1$ and $\tilde{B}$ is contained in $U_2\cup U_3.$
    \item $U_2$ contains the $(d-1)$-face of $t$ that is contained in the boundary of $r^L.$ 
    \item $U_3$ is the component of $y_2 \setminus S_2$ containing the shared face of $y_2$ and $t,$ where $S_2$ is a minimal witness for  $R^{\square}_{d}\paren{y_2}.$
    \item $r\cap t\paren{s,L}$ is contained in $U_1\cup U_3.$ 
\end{itemize}
See Figure~\ref{fig:SPartition}.
\end{Lemma}
\begin{proof}
Let $S_1,S_2, S_3,$ and $S_4$ be minimal witnesses for the crossing events $R^{\square}_{d-1}\paren{y_1}, R^{\square}_{d-1}\paren{y_2},$ $R^{\square}_{d}\paren{y_3},$ and $R^{\square}_{d}\paren{y_4},$ respectively. $S_2$ separates $t\setminus S_2$ into two components $U_3$ and $U_4,$ where the former contains both $\gamma$ and the face of $t$ contained in $\partial r^L$  ($s\times \brac{-L,L} \times \set{-L}$) and the latter contains the opposite face of $t.$ 

The union $\cup_{i=1}^4 S_i$ separates $\gamma$ from $\tilde{B},$ so we may find a minimal hypersurface $S_5$ composed of plaquettes of $\paren{S_1\cup S_3\cup S_4}\cap U_4$ so that $S_2\cup S_5$ does the same. As $S_5$ is minimal, it must divide $U_4$ into two components $U_1$ and $U_2,$
where $U_1$ contains $s.$

Set $S=S_2\cup S_5.$ The first two properties are satisfied by construction, and the third follows from the observation that $r\cap t\paren{s,L} \subset t\paren{s,L/2}\cup y_2.$   

\end{proof}

If $s$ is a $(d-2)$-dimensional box in $\Z^{d},$ there is a rigid motion $\rho$ of $\Z^d$ so that $\rho\paren{s}$ lies in  $\Z^{d-2}\times \set{0}^2.$  Let $t\paren{s,L}=\rho^{-1}\paren{t\paren{\rho{s},L}}.$ Similarly, set $C_{t\paren{s,L}}$ and $D_{t\paren{s,L}}$ to be the events $\rho^{-1}\paren{C_{t\paren{\rho{s},L}}}$ and   $\rho^{-1}\paren{D_{t\paren{\rho{s},L}}},$ respectively.  Also, denote by $\overline{C}_{t\paren{s,L}}$ the event $C_{t\paren{s,L}}\cap D_{t\paren{s,L}}.$ When $s$ is not specified, it will be assumed to be contained in $\Z^{d-2}\times\set{0}^2.$

 We are now ready to state the main topological result of this section. Let $r$ be a $(d-1)$-dimensional box in $\Z^d$ and let $\gamma=\partial r.$ Also, set $T\paren{L}=\cup_s s^L,$ where $s$ ranges over all $(d-2)$-faces of $r\paren{\gamma},$ so $T$ is a solid $(d-1)$-torus surrounding $\gamma.$ Finally, for a box $u,$ denote by $E_{u,L}$ be the event that all plaquettes in $u^L$ are contained in $P.$ 
  \begin{Proposition}\label{prop:implyv}
  
Set $y_6=\brac{-L,M_1+L}\times \ldots \times \brac{-L,M_{d-1}+L}\times\brac{-L,-L/2}$ and  $y_7=\brac{-L,M_1+L}\times \ldots \times \brac{-L,M_{d-1}+L}\times\brac{L/2,-L}.$  Then, for any $m\in \N$ 
\begin{equation}
    \label{eqn:eqn:implyv}
R^{\square}_d\paren{y_6}\bigcap R^{\square}_d\paren{y_7} \bigcap \cap_{s} \overline{C}_{t\paren{s,L}} \bigcap \cap_{u} E_{u,L} \implies V_{\gamma}^{\mathrm{fin}}\paren{m}\,,
\end{equation}
where $s$ ranges over all $(d-2)$-dimensional faces of $\gamma$ and $u$ ranges over all $(d-3)$-faces of $\gamma.$   
\end{Proposition}
While it would suffice to replace $R^{\square}_d\paren{y_6}\cap R^{\square}_d\paren{y_7}$ with a single occurrence of $R^{\square}_d\paren{r^L},$ the proof is simpler for this formulation of the proposition. We begin with a lemma extending the construction in Lemma~\ref{lemma:S}.

\begin{figure}
 \begin{subfigure}[b]{2.06in}
         \centering
         \includegraphics[width=\textwidth]{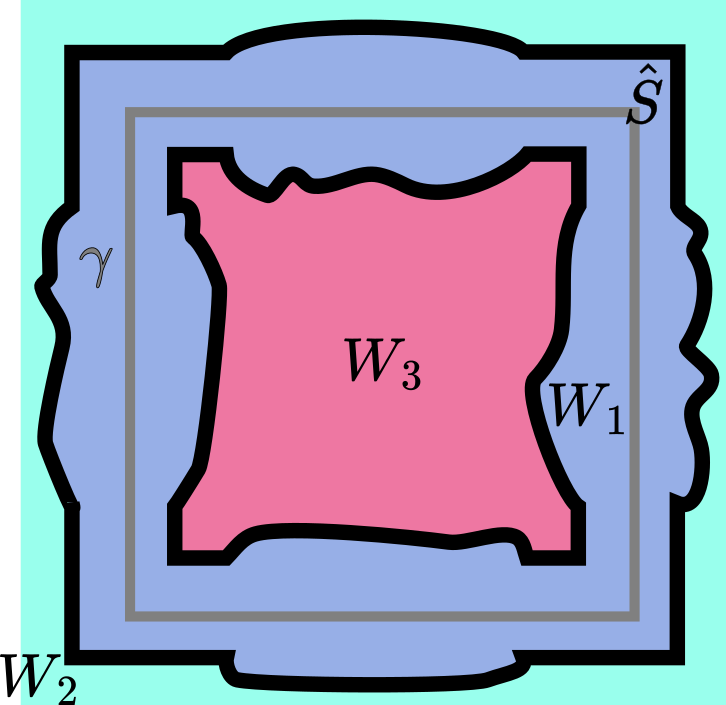}
         \subcaption[]{}
         \label{fig:SHat0}
     \end{subfigure}
    \qquad\qquad
     \begin{subfigure}[b]{2.42in}
         \centering
         \includegraphics[width=\textwidth]{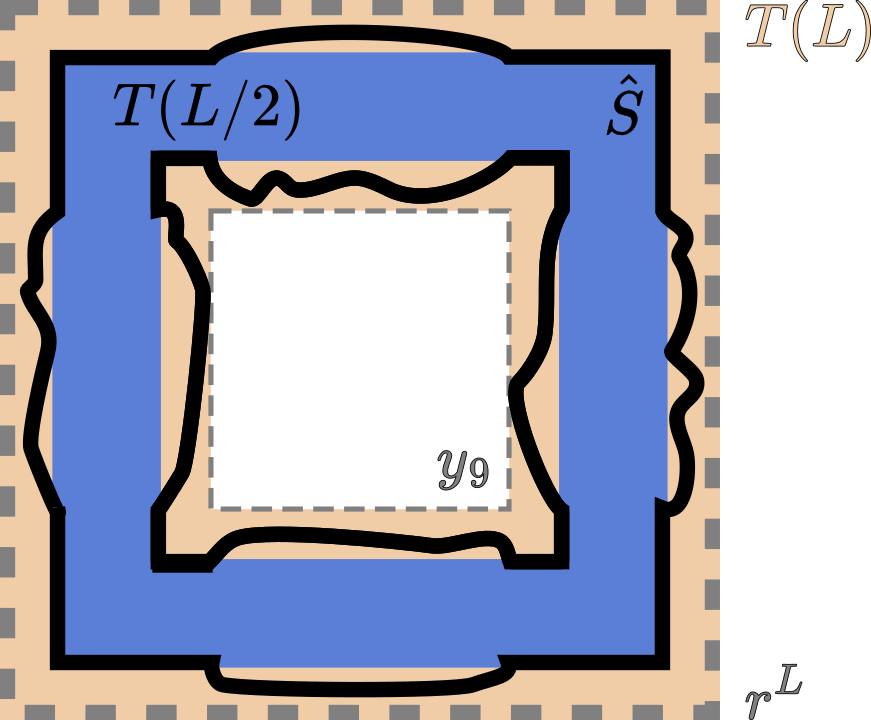}
        \subcaption[]{}
        \end{subfigure}
         \caption{A cross section of the hypersurface $\hat{S}$ constructed in Lemma~\ref{lemma:W123} for the case $d=3$, shown in the plane containing $\gamma.$ It is depicted in two different contexts: (A) with the regions $W_1, W_2,$ and $W_3$ in pastel purple, green, and pink, respectively and (B) with the regions $T\paren{L}$ (light orange), $T\paren{L/2}$ (dark purple), $r^L$ (the outer box, bounded by a thick dashed line), and $y_8$ (the inner box, bounded by a thin dashed line). Observe that $S$ may be taken to coincide with the boundary of $T(L/2)$ in a neighborhood of a corner of $\gamma$; this is possible because of the occurence of the events $E\paren{u,L}.$}
    \label{fig:SHat}
\end{figure}

\begin{Lemma}\label{lemma:W123}
Assume the hypotheses of Proposition~\ref{prop:implyv}. Then there exists a hypersurface $\hat{S}$ of plaquettes of $P\cap r^L$ that separates $r^L$ into three regions $W_1, W_2,$ and $W_3$ satisfying:
\begin{itemize}
    \item $T\paren{L/2}\subset W_1\subset T\paren{L}.$ 
    \item $W_3$ is contained in the shrunken box $y_8\coloneqq \brac{L/2,M_1-L/2}\times \ldots \times \brac{L/2,M_{d-1}-L/2}\times\brac{-L,L}.$
    \item $r$ is contained in $W_1\cup W_3.$ 
\end{itemize}
See Figure~\ref{fig:SHat}.
\end{Lemma}
\begin{proof}
Set $W_1=\cup_{s} U_1\paren{s},$ where $U_1\paren{s}$ was defined in the statement of Lemma~\ref{lemma:S} (and we extend the definition to general $(d-2)$-dimensional boxes by using translations/rotations). By construction, every plaquette in the boundary of $W_1$ is contained in $P$ and  $T\paren{L/2}\subset W_1\subset T\paren{L}.$ 

Let $S'$ and $S''$ be minimal witnesses for the crossings  $R^{\square}_d\paren{y_6}$ and $R^{\square}_d\paren{y_7}.$ They are disjoint, so by Lemma~\ref{lemma:crossingtopology}, they divide $r^L$ into three components, one of which contains the center of $r^L.$ Call this component $W_0.$ Set $y_9= \brac{L,M_1-L} \times \ldots \times \brac{L,M_{d-1}-L} \times \brac{-L,L},$ i.e. the closure of $r^L\setminus T,$ and 
\[W_3=W_0\cap \paren{y_9 \cup \bigcup_{s} U_3\paren{s}}\,.\]
Notice that the second desired containment property is satisfied since it is holds for both $y_9$ and each $U_3\paren{s}.$ It is also not difficult to check that $r \subset W_1 \cup W_3.$
Every plaquette in $\partial W_3$ is contained in $P$ because each face of $\partial y_9$ is contained in the component $U_3\paren{s}$ for some $s.$ That is, $\partial W_3$ is a union of subsets of $S',$ $S'',$ and hypersurfaces of the form $S_5$ constructed in the proof of Lemma~\ref{lemma:S}. 

Set $\hat{S}=\partial W_1\cup \partial W_3$ and $W_2=r^L\setminus \paren{\hat{S} \cup W_1\cup W_3}.$ Then the regions $W_1,W_2,$ and $W_3$ satisfy the required conditions.
\end{proof}

\begin{proof}[Proof of Proposition~\ref{prop:implyv}]
   By Corollary~\ref{cor:linkG} it suffices to show that no dual loop of $\overline{r^{L+1/2}}$ can be linked with $\gamma.$ Any such loop $\gamma^{\bullet}$ must be contained in one of the three components $W_1, W_2,$ and $W_3.$ 

    In the first case, $\gamma^{\bullet}$ is in the interior of one of the tubes $t\paren{s,L},$ as the events $E_{u,L}$ precludes it from entering more than one tube. The occurrence of the event $C_{t\paren{s,L}}$ implies that $\gamma$ is homologous to a cycle $\gamma'$ contained in $T\setminus \mathrm{interior}\paren{t\paren{s,L}}.$ The interior of $t\paren{s,L}$ is contractible in $\R^d\setminus \paren{T\setminus \mathrm{interior}\paren{t\paren{s,L}}}$ so $\gamma^{\bullet}$ cannot be linked with $\gamma.$ See Corollary~\ref{cor:link3}. 
   
   If $\gamma^{\bullet}$ is contained in $W_2$ then $\gamma^{\bullet}$ cannot be linked with $\gamma$ because $r\subset \R^d\setminus W_2$ and $\gamma$ is contractible in $r.$

    Finally, $\gamma$ is contractible in $\R^d\setminus y_8$ so it cannot be linked with any loop contained in $W_3.$ 
\end{proof}

\subsection{Proof of Theorem~\ref{theorem:sharp_perimeter}}
\label{subsec:sharp_perimeter}

Now that we have finished the technical lemmas, the proof proceeds similarly to those of~\cite{aizenman1983sharp}. Their arguments often use the independence of events defined on disjoint edge sets in Bernoulli percolation. In lieu of that, we employ the following lemma for wired boundary boundary conditions.
\begin{Lemma}\label{lemma:wired}

Let $X$ be a subcomplex of $\Z^d,$ and suppose $r_1,r_2\subset X$ are boxes which contain no shared $d$-cubes. If $A_1,A_2$ are increasing events that depend only on the edges of $r_1$ and $r_2$ respectively, then
\[\mu^{\mathbf{w}}_{X,p,q}\paren{A_1\cap A_2}\leq \mu^{\mathbf{w}}_{r_1,p,q}\paren{A_1}\mu^{\mathbf{w}}_{r_2,p,q}\paren{A_2}\,.\]  
\end{Lemma}
{Lemma}
\begin{proof}
Let $P_1\subset r_1$ and $P_2\subset r_2.$ As an application of the Mayer--Vietoris sequence, 
\[H^{d-2}\paren{P_1^{\mathbf{w}}\cup P_2^{\mathbf{w}};\;\Z_q} \cong H^{d-2}\paren{P_1^{\mathbf{w}};\;\Z_q}\oplus H^{d-2}\paren{P_2^{\mathbf{w}};\;\Z_q}\]
so
\[\abs{H^{d-2}\paren{P_1^{\mathbf{w}}\cup P_2^{\mathbf{w}};\;\Z_q}}= \abs{H^{d-2}\paren{P_1^{\mathbf{w}};\;\Z_q}}\abs{H^{d-2}\paren{P_2^{\mathbf{w}};\;\Z_q}}\]
It follows that 
$\mu^{\mathbf{w}}_{r_1\cup r_2,p,q}$ is the independent product measure  $\mu^{\mathbf{w}}_{r_1,p,q}\times \mu^{\mathbf{w}}_{r_2,p,q}.$ 
(Alternatively, one could prove this by counting components of the dual graphs.)  

Let $B$  be the event that all plaquettes of $\partial r_1 \cup\partial r_2$ are contained in $P.$ Then, by the FKG inequality,
    \begin{align*}
        \mu^{\mathbf{w}}_{X,p,q}\paren{A_1\cap A_2} \leq & \mu^{\mathbf{w}}_{X,p,q}\paren{A_1\cap A_2\mid B}\\
        =&  \mu^{\mathbf{w}}_{X_1\cup X_2,p,q}\paren{A_1\cap A_2}\\
        =&  \mu^{\mathbf{w}}_{X_1,p,q}\paren{A_1}\mu^{\mathbf{w}}_{X_2,p,q}\paren{A_2}\,.
    \end{align*}
\end{proof}

Next we show the analogue of Proposition 3.6 of~\cite{aizenman1983sharp}, closely following the argument therein.

\begin{Proposition}\label{prop:C}
For $p\in\brac{0,1},$ 
    \[c \coloneqq -\lim_{n\rightarrow\infty} \frac{\log\paren{ \mu_{\Lambda_n,p,q}^{\mathbf{w}}\paren{C_{\Lambda_n}}}}{\paren{2n}^{d-2}}\]
    exists, and is positive when $p>p^{*}\paren{p_c\paren{q}}.$ 
\end{Proposition}
\begin{proof}
    Let $n$ and $m$ be positive integers with $n>m.$ We can find $k\coloneqq \floor{\frac{n}{m}}^{d-2}$ disjoint cubes of width $m$ which are contained in $\Lambda_n$ and are centered at points of $s.$ Call these cubes $\Lambda^1,\ldots,\Lambda^k$ and let $D$ be the event that all plaquettes in the boundaries of those cubes are occupied. By Lemma~\ref{lemma:CR}, if $C_{\Lambda_n}$ occurs then the events $C_{\Lambda^1},\ldots,C_{\Lambda^k}$ happen as well. Then, 
    \[\mu_{\Lambda_n,p,q}^{\mathbf{w}}\paren{C_{\Lambda_n}}\leq   \mu_{\Lambda_m,p,q}^{\mathbf{w}}\paren{C_{\Lambda_m}}^{k}\,,\]
as a consequence of Lemma~\ref{lemma:wired}.

Taking logs and rearranging yields
\[\frac{\log\paren{ \mu_{\Lambda_n,p,q}^{\mathbf{w}}\paren{C_{\Lambda_n}}}}{n^{d-2}}\leq \paren{1+a\frac{m^{d-2}}{n^{d-2}}}\frac{\log\paren{\mu_{\Lambda_m,p,q}^{\mathbf{w}}\paren{C_{\Lambda_m}}}}{{m^{d-2}}}\]
where 
\[a=k-\paren{\frac{n}{m}}^{d-2}\]
satisfies $\abs{a}<1$ so
\[\limsup_{n\rightarrow\infty} \frac{\log\paren{ \mu_{\Lambda_n,p,q}^{\mathbf{w}}\paren{C_{\Lambda_n}}}}{n^{d-2}} \leq  \frac{\log\paren{ \mu_{\Lambda_m,p,q}^{\mathbf{w}}\paren{C_{\Lambda_M}}}}{m^{d-2}}\]
and we may conclude by taking the limit infimum as $m\rightarrow \infty.$ 
\end{proof}
We will eventually see that $c= \oldconstant{const:prcm2}.$ Note that the definition of the event $C_t$ depends on the choice of abelian group $G$ for homology coefficients, and so $c$ may be contingent on it as well.

\begin{Proposition}
For $p\in\brac{0,1},$ 
    \[\lim_{n\rightarrow\infty} -\frac{\log\paren{ \mu_{\Lambda_n,p,q}^{\mathbf{w}}\paren{\overline{C}_{\Lambda_n}}}}{\paren{2n}^{d-2}}=c\,.\]
\end{Proposition}
\begin{proof}
    The proof is identical to that of Proposition 3.7 in~\cite{aizenman1983sharp}:
    \[\mu_{\Lambda_n,p,q}^{\mathbf{w}}\paren{C_{\Lambda_n}}\mu_{\Lambda_n,p,q}^{\mathbf{w}}\paren{\overline{D}_r} \leq \mu_{\Lambda_n,p,q}^{\mathbf{w}}\paren{\overline{C}_{\Lambda_n}} \leq \mu_{\Lambda_n,p,q}^{\mathbf{w}}\paren{C_{\Lambda_n}}\]
    by the FKG inequality, where we are using the definition of $\overline{C}_{\Lambda_n}.$ The event $\overline{D}_r$ is the intersection of $2d$ box crossing events whose probability goes to $1$ as $N\rightarrow\infty$ by Lemma~\ref{lemma:crossing}. As such, the desired result follows by taking logarithms and dividing by $N.$
    
\end{proof}

\begin{Proposition}\label{prop:barctube}
Let $r_{l}$ be a family of boxes of the form $\brac{0,n_1(l)}\times \ldots \times \brac{0,n_{d-2}(l)}\times \brac{-m\paren{l},m\paren{l}}^2$ all of whose dimensions diverge to $\infty.$ Then
    \[\lim_{l\rightarrow\infty} -\frac{\log\paren{ \mu_{r_{l},p,q}^{\mathbf{w}}\paren{\overline{C}_{r_{l}}}}}{\prod_{i=1}^{d-2}n_i\paren{l}}=c\,.\]
    In addition, if $r_{l}$ is suitable then
\[\lim_{l\rightarrow\infty} -\frac{\log\paren{ \mu_{\Z^d,p,q}\paren{\overline{C}_{r_{l}}}}}{\prod_{i=1}^{d-2}n_i\paren{l}}=c\,.\]
\end{Proposition}
\begin{proof}
The general idea is that we can fit the appropriate number of cubes along $s$ in $r_l$ for large $l$ and we can also fit the appropriate number of copies of $r_l$ along $s$ in a large cube. First, for a fixed value of $N,$ we can find  $\prod_{i=1}^{d-2}\floor{\frac{n_i\paren{l}}{2N}}$ disjoint $2N$-cubes  so that the occurrence of $C_{r_{l}}$ implies that an event of the form  $C_{\Lambda_N}$ in each cube (when $l$ is sufficiently large). The same argument as in the proof of Proposition~\ref{prop:C} yields that   
\[\liminf_{l\rightarrow\infty} -\frac{\log\paren{ \mu_{r_{l},p,q}^{\mathbf{w}}\paren{C_{r_{l}}}}}{\prod_{i=1}^{d-2}n_i\paren{l}}\geq c\,.\]
On the other hand, if we fix $r=\brac{0,n_1}\times\ldots\times\brac{0,n_{d-2}}\times\brac{-m,m}^2$ and choose $N>m,$ then $C_{\Lambda_{N}}$ entails that translates of the event $C_r$ happen in $\prod_{i=1}^{d-2}\floor{\frac{N}{n_i}}$ disjoint boxes (which do not depend on the specific witness for $C_{\Lambda_N}$).  Thus 
\[\limsup_{l\rightarrow\infty} -\frac{\log\paren{ \mu_{r_{l},p,q}^{\mathbf{w}}\paren{C_{r_l}}}}{\prod_{i=1}^{d-2}n_i\paren{l}}\leq c\]
so
\[\lim_{l\rightarrow\infty} -\frac{\log\paren{ \mu_{r_{l},p,q}^{\mathbf{w}}\paren{C_{r_l}}}}{\prod_{i=1}^{d-2}n_i\paren{l}}= c\,.\]
The proof that this limit coincides with the corresponding one for the event $\overline{C}_{r_l}$ is identical to that of the previous proposition.

We apply Proposition~\ref{prop:interpolate} to show the second claim. As $\set{r\paren{l}}$ is a suitable family of boxes, we may choose a thickening parameter $L\paren{l}$ so that $L\paren{l}= \omega\paren{\log\paren{M\paren{r_{l}}}}$ and $L\paren{l}= o\paren{m\paren{r_{l}}}.$ For convenience, set $\tilde{r}_l=r_l^{L\paren{l}}$ and $f\paren{r}$ to be the product of the first $(d-2)$ dimensions of a box $r.$ By Proposition~\ref{prop:interpolate},
\[\frac{\log\paren{\mu_{\Z^d,p,q}\paren{\Xi_{r,L}}}+\log\paren{\mu_{r^L,p,q}^{\mathbf{w}}\paren{\overline{C}_{r_{l}}}}}{f\paren{r_{l}}} \leq \frac{\mu_{\Z^d,p,q}\paren{\overline{C}_{r_{l}}}}{f\paren{r_{l}}}\leq \frac{\mu_{r^L,p,q}^{\mathbf{w}}\paren{\overline{C}_{r_{l}}}}{f\paren{r_{l}}}\,.\]
It follows from Corollary~\ref{cor:xi} that
\[\abs{\log\paren{\mu_{\Z^d,p,q}\paren{\Xi_{r,L}}}} \leq  -\sum_{i=1}^{d} 2\log\paren{1-e^{-b_p L}\paren{M\paren{r_{l}}}^{d-2}}\]
Using the assumption that $L\paren{l}=\omega\paren{\log\paren{M\paren{r_{l}}}},$ we see that the right side is uniformly bounded above for sufficiently large $l.$ Then since $f\paren{r_l} \to \infty$ as $l \to \infty,$
\begin{align*}
\lim_{l \to \infty} \frac{\log\paren{\mu_{\Z^d,p,q}\paren{\Xi_{r_{l},L\paren{l}}}}}{f\paren{r_{l}}} = 0\,.
\end{align*}
As such, it suffices to show that 
\[-\frac{\mu_{r_L,p,q}^{\mathbf{w}}\paren{\overline{C}_{r_{l}}}}{f\paren{r_{l}}}\rightarrow c\,.\]

We have that 
\[\frac{\log\paren{\mu_{\tilde{r}_n,p,q}^{\mathbf{w}}\paren{\overline{C}_{\tilde{r}_{l}}}}}{f\paren{r_{l}}} \leq \frac{\log\paren{\mu_{\tilde{r}_l,p,q}^{\mathbf{w}}\paren{\overline{C}_{r_{l}}}}}{f\paren{r_{l}}}\leq \frac{\log\paren{\mu_{r_{l},p,q}^{\mathbf{w}}\paren{\overline{C}_{r_{l}}}}}{f\paren{r_{l}}}\,.\]
We already showed that term on the right limits to $-c$ as $l\rightarrow \infty.$ To handle the term on the left, note that $f\paren{\tilde{r}_l}-f\paren{r_{l}}\in o\paren{f\paren{r_{l}}}$ because  $L\paren{l}= o\paren{m\paren{r_{l}}}.$ As such, the asymptotics remains unchanged if we replace the denominator with $f\paren{\tilde{r}_l}.$ Thus, we may conclude that the middle term also limits to $-c,$ which suffices by the logic in the previous paragraph.
\end{proof}

We are now ready to prove Theorem~\ref{theorem:sharp_perimeter}.

\begin{proof}[Proof of Theorem~\ref{theorem:sharp_perimeter}]
Let $p>p^{*}\paren{p_c\paren{q}},$ and let $\set{\gamma_l}=\set{\partial \rho_{r_l}}$ be a suitable family of $(d-2)$-dimensional rectangular boundaries. Also fix $m\in N$ and set $V_{\gamma}^{\mathrm{fin}}=V_{\gamma}^{\mathrm{fin}}\paren{m}$ and  $V_{\gamma}^{\mathrm{inf}}=V_{\gamma}^{\mathrm{inf}}\paren{m}.$

As noted after the statement, it suffices to show that 

\[\limsup_{l\rightarrow\infty}-\frac{\mu_{\Z^d,p,q}^{\mathbf{f}}\paren{V_{\gamma_l}^{\mathrm{fin}}}}{\mathrm{Per}\paren{\gamma_l}}\leq c\]
and 
\[\limsup_{l\rightarrow\infty}-\frac{\mu_{\Z^d,p,q}^{\mathbf{w}}\paren{V_{\gamma_l}^{\mathrm{inf}}}}{\mathrm{Per}\paren{\gamma_l}}\geq c\,.\]

We begin by showing the second equation. For a fixed $l,$ let $\lambda_1,\ldots,\lambda_{g\paren{n,l}}$ be a maximal set of disjoint cubes of width $2n$ in $\Z^d$ that are centered at points of $\partial r_l$ and are disjoint from the $(d-3)$-faces of $r_l.$ We have that 
\[\lim_{l\rightarrow\infty }\frac{\mathrm{Per}\paren{\gamma_l}}{g\paren{n,l}}=\paren{2n}^{d-2}\,.\]
The event $V_{\gamma}$ implies that events obtained from $C_{\Lambda_n}$ by rotations or translations occur for each of the cubes $\lambda_i.$ Thus we can apply Lemma~\ref{lemma:wired} to obtain

\[\mu_{\Z^d,p,q}^{\mathbf{w}}\paren{V_{\gamma_l}^{\mathrm{inf}}}\leq \mu^{\mathbf{w}}_{\Lambda_n,p,q}\paren{C_{\Lambda_n}}^{g\paren{n,l}}\,.\]

Therefore
\begin{align*}
\frac{-\log\paren{\mu_{\Z^d,p,q}^{\mathbf{w}}\paren{V_{\gamma_l}^{\mathrm{inf}}}}}{\mathrm{Per}\paren{\gamma_l}}&\geq -\frac{g\paren{n,l}\log\paren{\mu^{\mathbf{w}}_{\Lambda_n,p,q}\paren{C_{\Lambda_n}}}}{\mathrm{Per}\paren{\gamma_l}}\\
&\xrightarrow[]{l \to \infty} \frac{-\log\paren{\mu^{\mathbf{w}}_{\Lambda_n,p,q}\paren{C_{\Lambda_n}}}}{\paren{2n}^{d-2}}\\
&\xrightarrow[]{n \to \infty} c
\end{align*}
by Proposition~\ref{prop:C}.  

On the other hand, choose $L\paren{l}$ so that $L\paren{l}\in o\paren{m\paren{r_l}^{(d-2)/d}}$ and $L\paren{l}\in \omega\paren{\log\paren{M\paren{r_l}}}.$ Set $\tilde{r}_l=\paren{s_l}^{L\paren{l}}.$ Recall from the statement of Proposition~\ref{prop:implyv} that  
\[R^{\square}_d\paren{y_6}\bigcap R^{\square}_d\paren{y_7} \bigcap \cap_{s} \overline{C}_{t\paren{s,L}} \bigcap \cap_{u} E_{u,L} \implies V_{\gamma}^{\mathrm{fin}}\,,\]
where $s$ and $u$ range over the $(d-2)$ and $(d-3)-$ faces of $\gamma,$ respectively. All of these events are increasing, so by the FKG inequality
\begin{align*}
    \log\paren{\mu_{\Z^d,p}^{\mathbf{f}}\paren{V_{\gamma_l}}^{\mathrm{fin}}}
    &\geq \log\paren{\mu_{\Z^d,p}^{\mathbf{f}}\paren{R^{\square}_d\paren{y_6}}}+\log\paren{\mu_{\Z^d,p}^{\mathbf{f}}\paren{R^{\square}_d\paren{y_7}}}\\
    &\quad+\log\paren{ \mu_{\Z^d,p}^{\mathbf{f}}\paren{\cap_{u} E_{u,L(l)}}}+ \sum_{s} \mu_{\Z^d,p}^{\mathbf{f}}\paren{\overline{C}_{t\paren{s,L(l)}}}\,.
\end{align*}
We will show that the first three terms grow asymptotically more slowly than $\mathrm{Per}\paren{\gamma_l},$ and that the third term behaves as $-c\mathrm{Per}\paren{\gamma}.$  

The assumption that $L\in \omega\paren{\log \paren{M\paren{r_l}}}$ yields that 
\[\frac{\log\paren{\mu_{\Z^d,p}^{\mathbf{f}}\paren{R^{\square}_d\paren{y_6}}}}{\mathrm{Per}\paren{\gamma_l}}=\frac{\log\paren{\mu_{\Z^d,p}^{\mathbf{f}}\paren{R^{\square}_d\paren{y_7}}}}{\mathrm{Per}\paren{\gamma_l}}\xrightarrow[]{l \to \infty} 0\]
by Lemma~\ref{lemma:crossing}. 
In addition the events $E_{u,L(l)}$ require the activation of $a d L^d$ plaquettes where $a=4(d-1)(d-2)$ is the number of $(d-3)$-faces of $\gamma.$ Thus 
\[\frac{\log\paren{\mu_{\Z^d,p}^{\mathbf{f}}\paren{\bigcap_{u} E_{u,L(l)}}}}{\mathrm{Per}\paren{\gamma_l}}\leq \frac{\log\paren{\paren{p/q}^{a b L^d}}}{\mathrm{Per}\paren{\gamma_l}}\rightarrow 0\]
because  $L\paren{l}\in o\paren{m\paren{r\paren{l}}^{(d-2)/d}}.$ Finally, by Proposition~\ref{prop:barctube},
\[\lim_{l\rightarrow \infty}-\frac{\log\paren{\mu_{\Z^d,p}^{\mathbf{f}}\paren{\overline{C}_{t\paren{s,L}}}}}{\abs{s}}\rightarrow c\] 
as $\abs{s}\rightarrow\infty,$ where $\abs{s}$ is the total number of $(d-2)$-plaquettes in the face $s$ of $r_l.$ As $\mathrm{Per}\paren{\gamma_l}=\sum_s{\abs{s}},$ we have that 
\[\limsup_{n\rightarrow \infty}-\frac{\log\paren{\mu_{\Z^d,p}^{\mathbf{f}}\paren{V_{\gamma}^{\mathrm{fin}}}}}{\mathrm{Per}\paren{\gamma}}\leq  c\,,\]
so $c = \oldconstant{const:prcm2}$ and has the desired properties.

\end{proof}

\appendix

\section{Topological Tools}
\label{sec:topology}
In this section we will review the relevant topological definitions and tools that we use in the rest of the article. We will start with the general theory and then discuss the specific properties of the cubical complex $\Z^d$. A standard reference for the former is~\cite{hatcher2002algebraic}. When possible, we use the specifics of our model to simplify the presentation. In particular, all spaces we consider are cubical complexes. While these are not covered in~\cite{hatcher2002algebraic}, the definitions given there result in equivalent algebraic structures. See~\cite{kaczynski2004computational,saveliev2016topology} for a more in depth discussion of the homology and cohomology of cubical complexes.

\subsection{Definitions of Homology and Cohomology}
\label{subsec:homology}
First we recall the definition of homology. In this paper, we only consider spaces called cubical complexes, which have a combinatorial structure that makes defining and computing homology simpler than for general topological spaces. Our cubical complexes are built from $i$-dimensional unit cubes, also called $i$-plaquettes, for $0 \leq i \leq d.$ In the case of the $\Z^d$ lattice, the $i$-dimensional cells are the $i$-plaquettes with integer corner points. 

The key operation is then the map that takes a plaquette to its boundary. As an example, consider an oriented $2$-plaquette $\sigma$ in a cubical complex $X.$ The (point set) topological boundary consists of the union of the four edges ($1$-plaquettes) between its corners. It turns out to be useful to handle this union as a linear combination. This motivates the definition of the \emph{chain} group $C_i\paren{X;\;G}$ of finite formal linear combinations of $i$-plaquettes of $X$ with coefficients in an abelian group $G.$ Now to revisit our $2$-plaquette example, the boundary operator takes $\sigma$ to $\partial \sigma,$ which is the sum of the boundary edges of $\sigma$ with orientations consistent with the orientation of $\sigma$ as in Figure~\ref{fig:oneplaquette}. Note that the orientation of each plaquette can be chosen arbitrarily as long as consistency is maintained and the opposite orientation of an $i$-plaquette $\sigma$ is $-\sigma \in C_i\paren{X;\;G}.$ Then adding the requirement that $\partial$ be $G$-linear uniquely defines it on the full space of chains.

It turns out that a similar operator can be defined for plaquettes in every dimension. For concreteness, we give the formal definition here, though a reader unfamiliar with the subject will be comforted to know that it is rarely referred to explicitly in practice. An $i$-plaquette in $\Z^d$ may be written in the following form. Let $1 \leq k_1<k_2<\ldots<k_i \leq d$ and $x_1,\ldots,x_d\in \Z.$  Set $\mathcal{I}_j = \brac{x_j,x_j+1}$ for $j \in \set{k_1,k_2,\ldots,k_i}$ and $\mathcal{I}_j = \set{x_j}$ for $j \in \brac{d} \setminus {k_1,k_2,\ldots,k_i}.$ Then $\sigma = \prod_{1 \leq j \leq d} \mathcal{I}_j$ is an $i$-plaquette in $\R^d,$ and its boundary $\partial \sigma$ is
{\small
\begin{equation}
    \label{eq:boundaryOperator}
    \sum_{l=0}^i \paren{-1}^{l-1} \paren{\prod_{1 \leq j < k_l} \mathcal{I}_j \times \set{x_{k_l}+1} \times \prod_{k_l < m \leq d} \mathcal{I}_m - \prod_{1 \leq j < k_l} \mathcal{I}_j \times \set{x_{k_l}} \times \prod_{k_l < m \leq d} \mathcal{I}_m}\,.
\end{equation}
}

We pause at this point to note that we can now formally define the event $V_\gamma.$ We represent the cycle $\gamma$ as an $(i-1)$-chain consisting of the sum of its $(i-1)$-plaquettes, oriented so that $\partial \gamma = 0.$ Then $V_\gamma$ is  the event that there is some element $\tau$ of $C_{i}\paren{P;\;\Z_q}$ so that $\partial \tau=\gamma.$ Observe that this depends on $q.$ Our interest in boundaries then leads us to the study of homology.

The chains $\alpha \in C_i\paren{X;\;G}$ satisfying $\partial \alpha = 0$ are called \emph{cycles}, and the space of cycles is written $Z_i\paren{X;\;G}.$ The group of chains which are \emph{boundaries} of other chains is denoted $B_{i}\paren{X;\;G}.$ We can then think of $V_\gamma$ as the event that the cycle $\gamma$ is also a boundary. A key fact that underlies much of algebraic topology is that the boundary operator satisfies $\partial \circ \partial = 0,$ so in particular $B_{i}\paren{X;\;G} \subseteq Z_i\paren{X;\;G}.$ Then the \emph{$i$-th homology group} is the quotient of the cycles by the boundaries, or
\[H_i\paren{X;\;G} = Z_i\paren{X;\;G}/B_i\paren{X;\;G}\,.\]

We will also use the dual notion of cohomology. Here the basic objects are not chains, but linear functionals on chains. More precisely, the $i$-th \emph{cochain group} $C^i\paren{X;\;G}$ is defined as the group of $G$-linear functions from $C_i\paren{X;\;G}$ to $G.$ That is, it is the group of assignments of spins in $G$ to the $i$-plaquettes of $X.$ Note that when $X$ is infinite, a chain $\sigma\in C_i\paren{X;\;G}$ is by definition supported on finitely many $i$-plaquettes of $X,$ whereas a cochain may take non-zero values on infinitely many plaquettes. The cochain group comes with an analogous \emph{coboundary operator} $\delta : C^i\paren{X;\;G} \to C^{i+1}\paren{X;\;G},$ which is given by
\[\delta f\paren{\alpha} = f\paren{\partial \alpha}\]
for $f\in C^i\paren{X;\;G}, \alpha \in C_{i+1}\paren{X;\;G}.$ 

We are then interested in similar quantities, now defined with coboundaries instead of boundaries. The $i$-th cocycle group $Z^i\paren{X;\;G}$ is the group of cochains with zero coboundaries and the $i$-th coboundary group $B^i\paren{X;\;G}$ is the image of $\delta$ applied to $C^{i-1}\paren{X;G}.$ Then $i$-th cohomology group is the quotient \[H^i\paren{X;\;G} = Z^i\paren{X;\;G} / B^i\paren{X;\;G}\,.\] 

The notions of homology and cohomology are closely related in a sense that will be made precise shortly. However, it is useful to have both perspectives available. In the current context, the PLGT is naturally a cochain whereas the event $V_{\gamma}$ is best defined in terms of homology. The distinction is also important for algebraic duality theorems, which are often stated in terms of passing from one to the other. These duality theorems are also most naturally stated in terms of variants of homology and cohomology called \emph{reduced homology} and \emph{reduced cohomology}. These are denoted by $\tilde{H}_j\paren{X;\;G}$ and $\tilde{H}^j\paren{X;\;G},$ respectively. The idea is that any nonempty space has at least one connected component, and that it is useful to remove one factor from $H_0\paren{X;\;\Z}$ and $H^0\paren{X;\;\Z}$ so that a contractible space has zero homology in all dimensions. 

This is accomplished by defining a boundary operator on $C_0\paren{X;\;\Z}$ with the map $\epsilon : C_0\paren{X;\;\Z} \to \Z$ that sends each $0$-plaquette (or vertex) to $1.$ The reduced cohomology groups are obtained similarly with the dual map to $\epsilon.$ The resulting groups coincide with the usual $j$-th homology and cohomology groups except in the case $j=0,$ where
\[\tilde{H}_0\paren{X;\;\Z} \oplus \Z \cong H_0\paren{X;\;\Z} \qquad \text{and} \qquad  \tilde{H}^0\paren{X;\;\Z} \oplus \Z \cong H^0\paren{X;\;\Z}\,.\]
for nonempty spaces $X.$ To reduce notational complexity, \emph{we always write $H_j$ and $H^j$ to denote the reduced homology and cohomology groups in the remainder of the article.} Notice that this does not change the random-cluster measure, since it changes the partition function by a constant factor.

\begin{figure}[t]
    \centering
    \includegraphics[width=.5\textwidth]{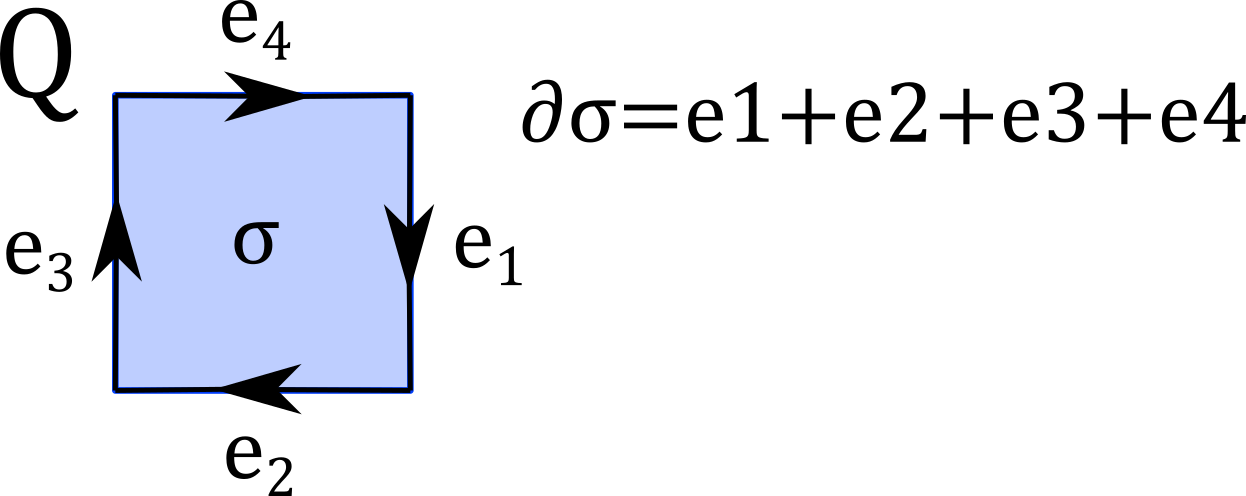}
    \caption{The boundary map for a two dimensional plaquette, reproduced from~\cite{duncan2022topological}.}
    \label{fig:oneplaquette}
\end{figure}

\subsection{Some Properties of Homology and Cohomology}
\label{subsec:homologyprops}

In this section we review a number of basic properties of homology and cohomology groups of percolation subcomplexes. From these statements we will see that there is a straightforward relationship between $H^{d-2}\paren{P;\;G}$ for a general coefficient group $G$ and $H^{d-2}\paren{P;\;\Q}.$

Let $0\leq j \leq d-1.$  The homology of $P$ with coefficients in the integers $\Z$ is a finitely generated abelian group so
\begin{equation}
\label{eqn:Zcoeff}
H_j\paren{P;\;\Z}\cong  \Z^{\mathbf{b}_j\paren{P;\;\Z}} \oplus T_j\paren{P}\,,
\end{equation}
where $\mathbf{b}_j\paren{P;\;\Z}$ is the rank of the $\Z$-module $H_j\paren{P;\;\Z}.$ The second summand, $T_i\paren{P},$ consists of all elements of $H_j\paren{P;\;\Z}$ of finite order and is referred to as the \emph{torsion subgroup} of $H_j\paren{P;\;\Z}.$ When $T_i\paren{P}=0,$ we call $H_j\paren{P;\;\Z}$ torsion-free.

It turns out that $(d-1)$-dimensional percolation complexes are torsion-free in dimensions $(d-1)$ and $(d-2).$ This can be be shown using Alexander duality, which relates the homology of a subset of Euclidean space with that of its complement. It is most naturally stated for subcomplexes of the $d$-dimensional sphere $S^d;$ we can embed our complexes into we can embed our percolation subcomplexes in $S^d$ by adding a point at infinity.  Here is where it is important to note that we are using reduced homology and cohomology.  The following statement is not true otherwise when $i=0$ or $i=d-1.$

\begin{Theorem}[Alexander Duality]\label{thm:adual}
   Let $K$ be a nonempty, locally contractible, proper subspace of the $d$-dimensional sphere $S^d.$ Then there is an isomorphism
   \[\hat{\mathcal{I}}:H_j\paren{S^d \setminus K ;\;\Z} \cong H^{d-j-1}\paren{K;\;\Z}\,.\]
\end{Theorem}

We have the following corollary. 
\begin{Proposition}[Corollary 3.46 of~\cite{hatcher2002algebraic}]\label{prop:torsion}Let $P$ be a percolation subcomplex of a box in $\Z^d.$ Then
$T_{d-1}\paren{P;\;\Z}=T_{d-2}\paren{P;\;\Z}=0.$
\end{Proposition}
It is not hard to prove this statement for the $(d-1)$-dimensional percolation subcomplexes using the tools developed below in Section~\ref{sec:dual}. In particular, the complement of $P$ has the same cohomology as the dual graph, and graphs do not have torsion in their homology or cohomology.

Next, the \emph{Universal Coefficients Theorem for Homology}  (Theorem 3A.3 in \cite{hatcher2002algebraic}) allows us to compute the homology of $P$ with coefficients in abelian group $G$ in terms of the homology with coefficients in $\Z.$ 

In particular, it implies that
\begin{equation}
\label{eq:UCTH}
H_j\paren{P;G}\cong \paren{H_j\paren{P;\;\Z}\otimes G}\oplus \mathrm{Tor}\paren{H_{j-1}\paren{P;\;\Z},G}\,.
\end{equation}
For a definition of the  $\mathrm{Tor}$ functor, please refer to Section 3A of~\cite{hatcher2002algebraic}. Here, we only require two facts. First, $\mathrm{Tor}\paren{0,G}$ vanishes for any abelian group $G.$ Second,  $\mathrm{Tor}\paren{H_{j-1}\paren{X;\;\Z},\Q}=0$ for any topological space $X.$ It follows that $\mathbf{b}_j\paren{X;\;\Z}=\mathbf{b}_j\paren{X;\;\Q},$ the dimension of the vector space $H_{j-1}\paren{X;\;\Q}.$   In particular, we may rewrite (\ref{eqn:Zcoeff}) as
\begin{equation}
\label{eqn:ZcoeffQ}
H_j\paren{P;\;\Z}\cong  \Z^{\mathbf{b}_j\paren{P;\;\Q}} \oplus T_j\paren{P;\;\Z}\,.
\end{equation}

There is also a \emph{Universal Coefficients Theorem for Cohomology} (Theorem 3.2 of~\cite{hatcher2002algebraic}) that relates the homology and cohomology groups of a topological space. We do not need the full statement here, only the following corollary.
\begin{Corollary}\label{cor:UCTC}
If $H_{j-2}\paren{P;\;G}$ vanishes (or, more generally, is a free $G$-module) then
\[H^{j-1}\paren{P;\;G}\cong \mathrm{Hom}\paren{H_{j-1}\paren{P;\;G},G}\,.\] In particular, when $G=\Z_q$ we have that
\begin{equation}
\label{eq:UCTC}
H^{j-1}\paren{P;\;\Z_q}\cong H_{j-1}\paren{P;\;\Z_q}\,.
\end{equation}
\end{Corollary}

We can now compare the $(d-2)$-homology of $P$ across different coefficients.
\begin{Proposition}\label{prop:codim1topology}
Let $P$ be a $(d-1)$-dimensional percolation subcomplex of a box $r\subset \Z^d$.
Then 
\[H^{d-2}\paren{P;\;\Z_q}\cong \Z_q^{\mathbf{b}_{d-2}\paren{P;\;\Q}}\,.\]
\end{Proposition}
\begin{proof}
First, 
\[H_{d-2}\paren{P;\;\Z} \cong \Z^{\mathbf{b}_{d-2}\paren{P;\;\Q}}\]
by (\ref{eqn:ZcoeffQ}) and Proposition~\ref{prop:torsion}. It follows from (\ref{eq:UCTH}) and the fact that $H_{d-3}\paren{P;\;\Z}=0$  that 
\[H_{d-2}\paren{P;\;\Z_q} \cong \Z_q^{\mathbf{b}_{d-2}\paren{P;\;\Q}}\,.\]
Finally, by Corollary~\ref{cor:UCTC}, 
\[H^j\paren{P;\;\Z_q}\cong \mathrm{Hom}\paren{\Z_q^{\mathbf{b}_{d-2}\paren{P;\;\Q}},\Z_q}\cong \Z_q^{\mathbf{b}_{d-2}\paren{P;\;\Q}}\,.\]
\end{proof}

We will use this statement to work with $\Q$ coefficients instead of $\Z_q$ coefficients in codimension one. See Proposition~\ref{prop:codim1}.

\subsection{Duality and Percolation Subcomplexes}
\label{sec:dual}
Let $P$ be a percolation subcomplex of $\Z^d$ or of a box contained therein, and let $Q$ be the dual complex. The topological properties of $Q$ are closely related to those of the complement $\R^d\setminus P.$ Before providing a precise statement, we review the definition of a deformation retraction from a topological space $X$ to a subset $Y\subset X.$ We say that $X$ deformation retracts to $Y$ if there is a continuous function  $I : X \times \brac{0,1}\rightarrow Y$ so that $I\paren{x,0}=x$ for all $x \in X,$ $I\paren{X,1}=Y$ and   $I\paren{y,t} = y$ for all $y\in Y$ and all $t\in\brac{0,1}.$ It is a standard fact that a deformation retraction is a homotopy equivalence and therefore induces isomorphisms on homology and cohomology groups (see chapter 2.1 of~\cite{hatcher2002algebraic}).

\begin{figure}
    \centering
     \includegraphics[width=0.8\textwidth]{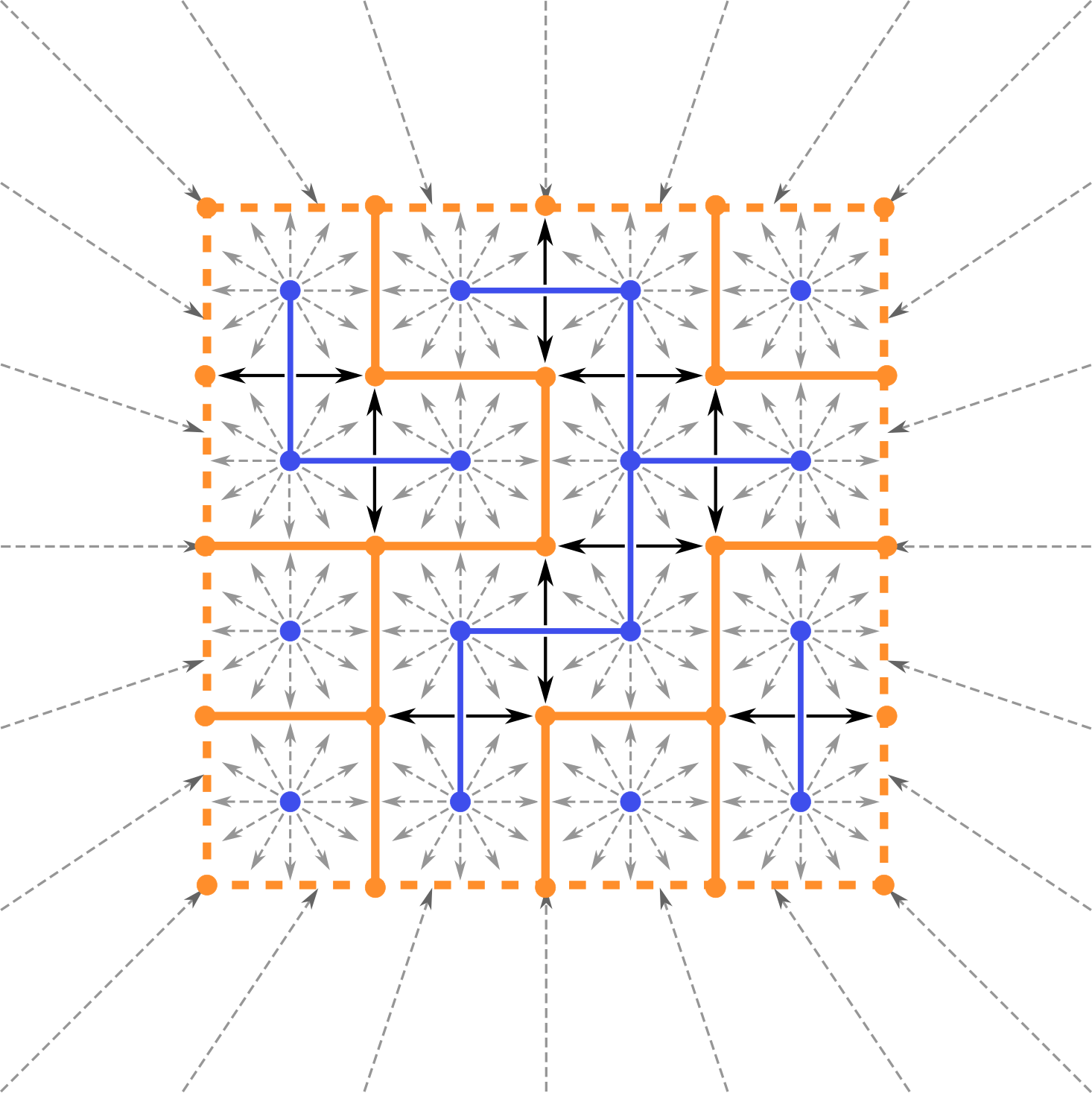}
         \caption{The deformation retract $I'$ defined in the proof of Lemma~\ref{prop:PtoQretract}. $P$ is shown in blue, $Q$ in orange, and $\partial r^{\bullet}$ by a dashed orange line. The deformation retract proceeds first along the gray arrows and then along the black ones.  Adapted from~\cite{duncan2020homological}, which includes details on the construction inside the inner square. Observe that in the case $i=1,$ the addition of the boundary of $r^{\bullet}\setminus r$ to $Q$ has the effect of merging all vertices in the boundary of the square. As we will see below, this is related to the duality between random-cluster models with free and wired boundary conditions.}
    \label{fig:defretract}
\end{figure}

\begin{Lemma}\label{prop:PtoQretract}
    Let $P$ be an $i$-dimensional percolation subcomplex of $X \subset \Z^d$ with dual $Q.$ Also, let $r$ be a box in $\Z^d.$
    \begin{enumerate}
        \item If $X = \Z^d,$ $\R^d \setminus P$ deformation retracts to $Q.$
        \item If $X = \overline{r},$ $\R^d\setminus P$ deformation retracts to $Q\cup \partial r^{\bullet}.$
        \item If $X = r,$ $P$ deformation retracts to a subcomplex $P'$ of $\overline{r^{-1}}$ whose dual complex $Q'$ satisfies $Q'\cup \partial \overline{r^{\bullet}}=\paren{Q\cap \overline{r^{\bullet}}}\cup \partial \overline{r^{\bullet}}.$ 
    \end{enumerate}
\end{Lemma}

\begin{proof} In the first case, the same construction as in the proof of Lemma 8 of~\cite{duncan2020homological} yields a deformation retraction $I:\R^d\setminus P \rightarrow Q.$ In the second case, we combine two deformation retractions: the deformation retraction $I:\R^d\setminus \hat{P} \rightarrow Q$ defined for the subcomplex $\hat{P}$ of $\Z^d$ obtained by adding the $(i-1)$-skeleton of $\Z^d$ to $P,$ and the straight line deformation retraction $J:\R^d\setminus r^{\bullet}\times \brac{0,1}\rightarrow \partial r^{\bullet}.$ Then, we can define a deformation retraction $I': \R^d \setminus P \times \brac{0,1}\rightarrow Q$ by 
\[I'\paren{x,t}=\begin{cases}I\paren{x,t} & x\in r^{\bullet} \\ J\paren{x,t} & x\in \R^d \setminus r^{\bullet}\end{cases}\,.\]
See Figure~\ref{fig:defretract}. In the third case, all $i$-dimensional cells that share an $(i-1)$-face with $\partial r$ are orthogonal to a face of $\partial r$ and can be deformation retracted to $\partial r^{-1}$ via straight lines. The claim then follows from the second case.
\end{proof}

\begin{proof}[Proof of Proposition~\ref{cor:alexander}]
First, let $P$ be a percolation subcomplex of $\overline{r}.$ Compactify $\R^d$ to $S^d$ by adding a point at infinity. Since $i>0$ and $P$ is bounded, applying Alexander duality (Theorem~\ref{thm:adual}) gives
\[H_{i}\paren{P;\;\Z}\cong H_{i}\paren{P \cup \set{\infty};\;\Z} \cong H^{d-i-1}\paren{\R^d\setminus P;\;\Z}\,.\]
The cohomology of $\R^d\setminus P$ is isomorphic to that of $Q\cup \partial r^{\bullet}$ by the previous proposition.

When $P$ is a subcomplex of $r,$ the existence of such an isomorphism follows because of the third item in the previous proposition.
\end{proof}

\section{Wired Boundary Conditions for PLGT}
\label{sec:wired}

Here, we detail the adjustments required to adapt the constructions of Section~\ref{sec:coupling} to wired boundary conditions. 

\subsection{Finite Volume Measures}
\label{subsec:couplingwired}
We begin by reviewing the definitions. The PRCM with wired boundary conditions on a box $r$ is the measure
\[\tilde{\mu}^{\mathbf{w}}_{r,p,G,i}\paren{P}\propto p^{\abs{P}}\paren{1-p}^{\abs{r^{\paren{i}}} - \abs{P}}\abs{H^{i-1}\paren{P^{\mathbf{w}};\;G}}\,,
\]
where $P^{\mathbf{w}}$ is the percolation subcomplex of $\overline{r}$ obtained by adding all $i$-plaquettes in $\partial r$ to $P.$

Let 
\[\psi:C^{i-1}\paren{r;\;\Z_q} \to C^{i-1}\paren{\partial r;\;\Z_q}\]
be the map which restricts a cochain on $r$ to one on $\partial r,$ and let 
\[D_{\eta}\paren{r;\;\Z_q}=\psi^{-1}\paren{\eta}\]
for some $\eta\in Z^{i-1}\paren{\partial r;\Z_q}.$ 
Also, set 
\[D_{\mathbf{w}}\paren{r;\;\Z_q}=\ker \delta \circ \psi\,.\] Then  $\nu^{\mathbf{w}}_{r,\beta,q,i-1}$ and $\nu^{\eta}_{r,\beta,q,i-1}$ are the Gibbs measures on  $D_{\mathbf{w}}\paren{r;\;\Z_q}$ and  $D_{\eta}\paren{r;\;\Z_q}$ induced by the Hamiltonian~(\ref{eq:hamiltonian_potts}), respectively. 

Our next goal is to demonstrate that the expectations of gauge invariant random variables coincide for wired and $\eta$ boundary conditions. Note that the sets $D_{\eta}\paren{r;\;\Z_q}$ are the cosets of $\ker\psi'$ in $D_{\mathbf{w}}\paren{r;\;\Z_q},$ where $\psi'$ is the restriction of $\psi$ to $D_{\mathbf{w}}\paren{r;\;\Z_q}.$  

The gauge action of $C^{k-2}\paren{r;\;\Z_q}$ on $C^{k-1}\paren{r;\;\Z_q}$ is defined by setting  $\phi_h:C^{k-1}\paren{r;\;\Z_q}\rightarrow  C^{k-1}\paren{r;\;\Z_q}$ to be the map
\[\phi_h\paren{f}= f+ \delta h\]
for $f\in C^{k-1}\paren{r;\;\Z_q}$ and $h\in C^{k-2}\paren{r;\;\Z_q}.$
\begin{Lemma}\label{lemma:gaugeaction}
The gauge action descends to a transitive group action on $D_{\mathbf{w}}\paren{r;\;\Z_q}/\ker\psi'.$ 
\end{Lemma}

\begin{proof}
First, we show that the induced action on  $D_{\mathbf{w}}\paren{r;\;\Z_q}/\ker\psi'$ is well-defined. Let $\eta\in  Z^{i-1}\paren{\partial r;\;\Z_q},$ $h\in C^{k-2}\paren{r;\;\Z_q},$ and  $f_1,f_2 \in D_{\eta}\paren{r;\;\Z_q}.$ Then
\[\psi'\paren{\phi_h\paren{f_1}}=\psi'\paren{\phi_h\paren{f_2}}=\eta+\psi'\paren{\delta h}\]
and
\[\delta\paren{\eta+\psi'\paren{\delta h}}=\delta\eta +\psi'\paren{\delta \circ \delta h}=0\]
so $\eta+\psi'\paren{\delta h} \in  Z^{i-1}\paren{\partial r;\;\Z_q}$ and $\phi_h$ maps $D_{\eta}\paren{r;\;\Z_q}$ to $C^{i-1}_{\eta+\psi\paren{\delta h}}\paren{r;\;\Z_q},$ bijectively.  Thus the action of $h$ on  $D_{\mathbf{w}}\paren{r;\;\Z_q}/\ker\psi'$ is well-defined.

Next, we prove transitivity. Let $\eta_1,\eta_2\in Z^{i-1}\paren{\partial r;\;\Z_q}$ and set $\eta=\eta_2-\eta_1.$ As $\partial r$ is homeomorphic to the $(d-1)$-dimensional sphere and $i<d,$ 
\[H^{i-1}\paren{\partial r;\;\Z_q}=Z^{i-1}\paren{\partial r;\;\Z_q}/B^{i-1}\paren{\partial r;\;\Z_q}=0\]
so $\eta=\delta h$ for some $h\in C^{i-2}\paren{\partial r;\;\Z_q}.$ We can extend $h$ to an element $\hat{h}$ of  $C^{i-2}\paren{r;\;\Z_q}$ by, for example, requiring it to vanish on all other $(i-2)$-faces. Then
\[C^{i-1}_{\eta_2}\paren{r;\;\Z_q}=\phi_{\hat{h}}\paren{C^{i-1}_{\eta_1}\paren{r;\;\Z_q}}\,,\]
as desired. 
\end{proof}

\begin{Corollary}
\label{cor:eta}
For $\eta_1,\eta_2\in Z^{i-1}\paren{\partial r;\;\Z_q}$ there exists an $h\in C^{i-2}\paren{r;\;\Z_q}$ so that 
\[\paren{\phi_h}_*\paren{\nu^{\eta_1}_{r,\beta,q,i-1}}=\nu^{\eta_2}_{r,\beta,q,i-1}.\]
In particular, if $\gamma \in Z_{i-1}\paren{r;\;\Z_q}$ then
\[\mathbb{E}_{\nu^{\eta_1}_{r,\beta,q,i-1}}\paren{W_{\gamma}}=\mathbb{E}_{\nu^{\eta_2}_{r,\beta,q,i-1}}\paren{W_{\gamma}}=\mathbb{E}_{\nu^{\mathbf{w}}_{r,\beta,q,i-1}}\paren{W_{\gamma}}\,.\]
\end{Corollary}
\begin{proof}
By transitivity, we may find an $h$ so that 
\[\phi_h\paren{C^{i-1}_{\eta_1}\paren{r;\;\Z_q}}=C^{i-1}_{\eta_2}\paren{r;\;\Z_q}\,.\]
The first claim then follows because $\phi_h$ preserves the Hamiltonian. The gauge-invariance of Wilson loop variables and the law of total conditional expectation implies the second as
\[\mathbb{E}_{\nu^{\mathbf{w}}_{r,\beta,q,i-1}}\paren{W_{\gamma}}=\sum_{\eta \in  Z^{i-1}\paren{\partial r;\;\Z_q}}\mathbb{E}_{\nu^{\eta}_{r,\beta,q,i-1}}\paren{W_{\gamma}}\nu^{\mathbf{w}}_{r,\beta,q,i-1}\paren{D_{\eta}\paren{r;\;\Z_q}}\,.\]
\end{proof}

Next, we couple $\nu^{\eta}_{r,\beta,q,i-1}$ with the wired PRCM measure. We perform a preliminary computation.
\begin{Lemma}\label{lemma:wiredcoupling}
Let $P$ be a percolation subcomplex of $r$ and $\eta\in Z^{i-1}\paren{\partial r;\;\Z_q}.$ Then
    \begin{equation*}
    \abs{Z^{i-1}\paren{P^{\mathbf{w}};\;\Z_q}\cap D_{\eta}\paren{r;\;\Z_q}}=\frac{\abs{Z^{i-1}\paren{P^{\mathbf{w}};\;\Z_q}}}{\abs{Z^{i-1}\paren{\partial r;\;\Z_q}}}\,.
\end{equation*}
\end{Lemma}
\begin{proof}
By construction, $Z^{i-1}\paren{P^{\mathbf{w}};\;\Z_q}\subset  D_{\mathbf{w}}\paren{r;\;\Z_q}.$ The action of gauge transformations  sends  $Z^{i-1}\paren{P^{\mathbf{w}};\;\Z_q}$ to itself and it acts transitively on the sets $D_{\eta'}\paren{r;\;\Z_q},$ so we may find an $f\in Z^{i-1}\paren{P^{\mathbf{w}};\;\Z_q}$ so that $\psi\paren{f}=\eta.$ In other words, the restriction of $\psi$ to $Z^{i-1}\paren{P^{\mathbf{w}};\;\Z_q}$ surjects onto $Z^{i-1}\paren{\partial r;\;\Z_q}.$ Thus,
\[\frac{Z^{i-1}\paren{P^{\mathbf{w}};\;\Z_q}}{Z^{i-1}\paren{P^{\mathbf{w}};\;\Z_q} \cap \ker \psi}\cong  Z^{i-1}\paren{\partial r;\;\Z_q}\,,\]
and the statement follows.
\end{proof}

\begin{Proposition}\label{prop:couplewired}
Let $r\subset \Z^d$ be a box, $q\in \N+1,$ $\beta \in [0,\infty),$ and $p = 1-e^{-\beta}.$ Let $\#=\mathbf{w}$ or $\#=\eta$ for $\eta\in Z^{i-1}\paren{\partial r;\;\Z_q}.$ Define a coupling on $D^{i-1}_{\#}\paren{r;\;\Z_q}\times \set{0,1}^{r^{\paren{i}}}$ by
\[\kappa^\# \paren{f,P} \propto \prod_{\sigma \in r^{\paren{i}}}\brac{\paren{1-p}I_{\set{\sigma \notin P}} + p I_{\set{\sigma \in P,\delta f\paren{\sigma}=0}}}\,.\]
Then $\kappa^{\#}$ has the following marginals.
\begin{itemize}
    \item The first marginal is $\tilde{\mu}^{\#}_{X,p,\Z_q,i}.$ 
    \item The second marginal is $\nu^{\#}_{r,\beta,q,i-1}$
\end{itemize}
\end{Proposition}

\begin{proof}
The computation of the first marginal is no different than before~\cite{hiraoka2016tutte}.

The derivation of the second marginal is nearly identical as that in the proof of Proposition~\ref{prop:couple}. For $\#=\mathbf{w},$ the only differences are that the terms $Z^{i-1}\paren{P;\;\Z_q}, B^{i-1}\paren{P;\;\Z_q},$ and  $H^{i-1}\paren{P;\;\Z_q}$ for $P$ are replaced by the corresponding quantities for $P^{\mathbf{w}}.$ For $\#=\eta,$ on the fourth line one obtains the quantity $\abs{Z^{i-1}\paren{P^{\mathbf{w}};\;\Z_q}\cap D_{\eta}\paren{r;\;\Z_q}}$ in place of $\abs{Z^{i-1}\paren{P;\;\Z_q}}$, but this is proportional to $\abs{Z^{i-1}\paren{P^{\mathbf{w}};\;\Z_q}}$ by Lemma~\ref{lemma:wiredcoupling}.

\end{proof}

\begin{Proposition}\label{prop:comparisonwired}
Let $r$ be a box in $\Z^d,$ $0<i<d-1,$  $q\in\N+1,$ $\eta\in Z^{i-1}\paren{\partial r ;\;\Z_q},$  and $\gamma\in Z_{i-1}\paren{r;\;\Z_q}.$ 
Denote by $V_{\gamma}^{\mathbf{w}}$ the event that $\brac{\gamma}=0$ in $H_{d-2}\paren{P\cup \partial r;\;\Z_q}.$ Then 
\[\tilde{\mu}\paren{V_{\gamma}^{\mathbf{w}}}=\mathbb{E}_{\nu^{\mathbf{w}}}\paren{W_{\gamma}}=\mathbb{E}_{\nu^{\eta}}\paren{W_{\gamma}}\]
where $\nu^{\eta}=\nu^{\#}_{r,\beta,q,i-1,d}$ is the PLGT with boundary conditions $\eta,$ $\tilde{\mu}=\tilde{\mu}^{\mathbf{w}}_{r,1-e^{-\beta},\Z_q,i}$ is the corresponding PRCM, and $V_{\gamma}^{\mathbf{w}}$ is the event that $\brac{\gamma}=0$ in $H_{i-1}\paren{P^{\mathbf{w}};\;\Z_q}.$ 
\end{Proposition}
\begin{proof}
The first equality follows from the argument in the proof of Proposition~\ref{prop:comparisonbody} applied to the complex $P \cup \partial r,$ where $P$ is distributed according to $\tilde{\mu}.$ The second is Corollary~\ref{cor:eta}.
\end{proof}

\subsection{The Infinite Volume Limit}
\label{sec:wiredinfinite}
The proofs in this section are similar to those in Section~\ref{subsec:comparisoninfinite}. Similarly to how the  measure $\nu_{r}^{\mathbf{w}}$ is defined on a subgroup $D_{\mathbf{w}}\paren{r;\;\Z_q}$ of  $Z^{i-1}\paren{r;\;\Z_q},$ the conditional measure of  $\nu_{\Z^d}^{\mathbf{w}}$ given $P$ will be defined on a subgroup $D_{\mathbf{w}}\paren{P;\;\Z_q}$ of  $Z^{i-1}\paren{P;\;\Z_q}.$ We would like to impose an additional constraint to require that cocycles vanish on cycles which are ``homologous to $\infty.$'' Recalling the definition $C^{i-1}\paren{P;\;\Z_q}=\mathrm{Hom}\paren{C_{i-1}\paren{P;\;\Z_q},\Z_q},$ note that we could equivalently define $Z^{i-1}\paren{P;\;\Z_q}$ to be the kernel of the restriction map
\[C^{i-1}\paren{P;\;\Z_q}\to  \mathrm{Hom}\paren{B_{i-1}\paren{P;\;\Z_q},\Z_q}\,.\]
We will enlarge $B_{i-1}\paren{P;\;\Z_q}$ to include finite $(i-1)$-cycles that are the boundaries of infinite $i$-cycles. Toward that end, let $C_j^{\mathrm{lf}}\paren{P;\;\Z_q}$ denote the group of (possibly infinite) formal sums of plaquettes in $P$ with coefficients in $\Z_q$ (here $\mathrm{lf}$ stands for ``locally finite''). Set
\[B_j^{\mathrm{lf}}\paren{P;\;\Z_q}=\ker\paren{\partial:C_{j+1}^{\mathrm{lf}}\paren{P;\;\Z_q}\to C_{j+1}^{\mathrm{lf}}\paren{P;\;\Z_q}}\,,\] 
let
\[\chi:C^{i-1}\paren{P;\;\Z_q}\to \mathrm{Hom}\paren{C_{i-1}\paren{P;\;\Z_q}\cap B_{i-1}^{\mathrm{lf}}\paren{P;\;\Z_q},\Z_q}\]
be the natural map, and define
\[D_{\mathbf{w}}\paren{P;\;\Z_q}=\ker\chi\,.\]

For percolation subcomplex $P$ of $\Z^d$ let $P_N^{\mathbf{w}}$ be the percolation subcomplex of $\overline{\Lambda_N}$ obtained by adding all $i$-plaquettes in $\partial \Lambda_N$ to $P_N.$ Now, we can replace the maps $\phi_{N,n}$ and $\phi_{\infty,n}$ with corresponding maps $\hat{\phi}_{N,n}:Z^{i-1}\paren{P_N^{\mathbf{w}};\;\Z_q}\to Z^{i-1}\paren{P_n;\;\Z_q}$ and $\hat{\phi}_{\infty,n} : D_{\mathbf{w}}\paren{P;\;\Z_q}\to Z^{i-1}\paren{P_n;\;\Z_q}.$ Set $\hat{Y}_{N,n}=\im\hat{\phi}_{N,n}$ and $\hat{Y}_{\infty,n}=\im\hat{\phi}_{\infty,n}.$ Before constructing the infinite volume coupling, we provide analogues for Lemmas~\ref{lemma:phiNM} and~\ref{lemma:cylinder}.

\begin{Lemma}
    \label{lemma:phiNMwired}
    Let $f\in Z^{i-1}\paren{P_n;\;\Z_q}$ and $N>n.$ Then  $f\notin \im\hat{\phi}_{N,n}$ if and only if there exists a cycle $\sigma\in Z_{i-1}\paren{P_n;\;\Z_q}\cap B_{i-1}\paren{P_{N}^{\mathbf{w}};\;\Z_q}$ so that $f\paren{\sigma}\neq 0.$ 

    Similarly  $f\notin \im\hat{\phi}_{\infty,n}$ if and only if there exists a cycle $\sigma\in Z_{i-1}\paren{P_n;\;\Z_q}\cap B_{i-1}^{\mathrm{lf}}\paren{P;\;\Z_q}$ so that $f\paren{\sigma}\neq 0.$ 
\end{Lemma}
\begin{proof}
The proof of the first claim is exactly the same as that of Lemma~\ref{lemma:phiNM}. The ``if'' direction of the second claim is immediate.

Assume that $f\in Z^{i-1}\paren{P_n;\;\Z_q}$ satisfies $f\paren{\sigma}=0$ for all $\sigma\in   Z_{i-1}\paren{P_n;\;\Z_q}\cap B_{i-1}^{\mathrm{lf}}\paren{P;\;\Z_q}.$ We will show that $f$ can be extended to obtain an element of $D_{\mathbf{w}}\paren{P;\;\Z_q}.$ In this case, there exists an $N>n$ so that no element in $Z_{i-1}\paren{P_n;\;\Z_q}$ is homologous to a cycle supported on $\partial \Lambda_N.$ By the first statement, we may find an $f'\in  Z^{i-1}\paren{P_{N}^{\mathbf{w}};\;\Z_q}$ so that $\hat{\phi}_{N,n}\paren{f'}=f.$  As $H^{i-1}\paren{\partial \Lambda_{N};\;\Z_q}=0,$ there exists an $h'\in C^{i-2}\paren{\partial \Lambda_{N};\;\Z_q}$ so that $\delta h'=f'\mid_{\partial \Lambda_{N}}.$ We can extend $h'$ to a cochain $h\in C^{i-2}\paren{\Z^d;\;\Z_q}$ by setting it to equal $0$ on all $(i-2)$-plaquettes outside of $\partial\Lambda_{N}.$ Let $g\in C^{i-1}\paren{P;\;\Z_q}$ be the cochain
\[g\paren{\sigma}=\begin{cases} f'\paren{\sigma}& \sigma \in \Lambda_{N} \\ \delta h\paren{\sigma} & \sigma \notin \Lambda_{N}  \end{cases}\,.\]
Then $g$ is a cocycle and $\phi_{\infty,n}\paren{g}=f.$ 
\end{proof}

\begin{Lemma}
Fix $n.$ 
    \label{lemma:cylinderwired}
    \begin{itemize}
    \item $\hat{Y}_{N,n}=\hat{Y}_{\infty,n}$ for all sufficiently large $N.$
    \item For $N$ sufficiently large as in the previous statement, the pushforward by $\hat{\phi}_{N,n}$ of the uniform measure on $Z^{i-1}\paren{P_N^{\mathbf{w}};\;\Z_q}$ is the uniform measure on $\hat{Y}_{\infty,n}.$ 
    \item For all $N>n,$ $\phi_{N,n}\paren{\hat{Y}_{\infty,N}}=\hat{Y}_{\infty,n}$ and the pushforward of the uniform measure on $\hat{Y}_{\infty,N}$ by $\phi_{N,n}$ is the uniform measure on $\hat{Y}_{\infty,n}.$ 
    \end{itemize}
\end{Lemma}
Note that the third item is for the map $\phi_{N,n}$ rather than $\hat{\phi}_{N,n}.$ 
\begin{proof}
The proof is no different than that for  Lemma~\ref{lemma:cylinder}. 
\end{proof}

We can now construct the infinite volume measures with wired boundary conditions, and prove Theorem~\ref{thm:comparison} for wired boundary conditions. As a consequence, Theorem~\ref{thm:sharpnessRCM} implies Theorem~\ref{thm:sharpness} for both wired boundary conditions and for any infinite volume limit of PLGTS constructed using $\eta$ boundary conditions. 

\begin{Proposition}\label{prop:comparisoninfinitewired}
Let $0<i<d-1,$ $q\in\N+1,$ $\beta \in \paren{0,\infty}$ and $p = 1-e^{-\beta}.$  
The weak limits
\[\mu_{\Z^d,p}^{\mathbf{w}}=\lim_{N\rightarrow \infty} \mu^{\mathbf{w}}_{\Lambda_N,p,q,d-1}\]
and
\[\nu_{\Z^d}^{\mathbf{w}}=\lim_{N\rightarrow \infty} \nu_{\Lambda_N,\beta,q,d-1}^{\mathbf{w}}\]
exist and are translation invariant.
Moreover, if $\gamma$ is a $(i-1)$-cycle in $\Z^d$ then
\[\mathbb{E}_{\nu_{\Z^d}^{\mathbf{w}}}\paren{W_{\gamma}}=\mu_{\Z^d,p}^{\mathbf{w}}\paren{V_{\gamma}^{\mathrm{inf}}}\,.\]
\end{Proposition}

\begin{proof}

 The weak limit $\mu^{\mathbf{w}}_{\Lambda_N,p,q,d-1}$ exists and is translation invariant by the same argument as in the proof of Theorem 4.19 of~\cite{grimmett1999percolation}.  In fact, we may couple the PRCMs with $\overline{P}\paren{1}\supset \overline{P}\paren{2}\supset \ldots$ so that:
 \begin{itemize}
    \item $\overline{P}\paren{N}\cap \Lambda_N=P\paren{N}$ and $\overline{P}\paren{N}\cap \Lambda_n^c$ contains all possible $i$-plaquettes.
     \item  $P\paren{N} \sim \mu^{\mathbf{w}}_{\Lambda_n,p,q,d-1}.$
     \item  $P=\cap_{N} P\paren{N} \sim \mu^{\mathbf{w}}_{\Z^d,p}.$
 \end{itemize}

We will construct a coupling on $\Sigma\times \Omega$ whose first marginal is $\mu^{\mathbf{w}}_{\Z^d,p}$ and whose second marginal is the weak limit of $\nu_{\Lambda_N,\beta,q,d-1}^{\mathbf{w}}$ as $N\to\infty.$ As before, set:
\begin{align*}
    \kappa_{\beta,q}^{\mathbf{w}}\paren{\mathcal{K}\paren{P_n}\times \mathcal{L}\paren{f_n}} = \sum_{H\subseteq Z^{i-1}\paren{P_n;\;\Z_q}}\frac{I_{\set{f_N \in H}}}{\abs{H}}\mu_{\Z^d,p}^{\mathbf{w}}\paren{\set{\hat{Y}_{\infty,n}=H} \cap \mathcal{K}\paren{P_n}}\,.
\end{align*}

The proof that this extends to a translation-invariant measure on  $\Omega \times \Sigma$ proceeds exactly the same as that for free boundary conditions. Similarly, the argument as to why the marginals are as claimed is nearly identical. The only difference is the justification for why we can choose $N$ large enough so that $P\paren{N}\cap \Lambda_n=P_n$: it is because  $\overline{P}\paren{N}\searrow P$ and $\overline{P}\paren{N}\cap\Lambda_n=P\paren{n}$ for $N>n.$ 

We now prove the statement relating Wilson loop variables to $V_{\gamma}^{\mathrm{inf}}.$  Fix $\gamma\in Z_{i-1}\paren{\Z^d;\;\Z_q}.$  Let $V'\paren{n}$ be the event that $\brac{\gamma}=0$ in $H_{i-1}\paren{P\paren{N}\cup \partial \Lambda_N;\;\Z_q}$ and let $V''\paren{N}$ be the event that $\brac{\gamma}=0$ in $H_{i-1}\paren{\overline{P}\paren{n};\;\Z_q}.$  Note that if $\gamma$ is supported on $\Lambda_N,$ then  $V''\paren{N} \iff V'\paren{N}.$ This, together with the fact that  
\[{V_{\gamma}^{\mathrm{inf}}}=\cap_{N\in\N }V''\paren{N}\]
gives the first inequality below:
\[\mu_{\Z^d,p}^{\mathbf{w}}\paren{V_{\gamma}^{\mathrm{inf}}} = \lim_{N \to \infty} \mu_{\Lambda_N}^{\mathbf{w}}\paren{V'\paren{N}} = \lim_{N \to \infty} \mathbb{E}_{\nu_{\Lambda_N}^{\mathbf{w}}}\paren{W_{\gamma}} = \mathbb{E}_{\nu_{\Z^d}^{\mathbf{w}}}\paren{W_{\gamma}}\,,\]
where the second equality is Proposition~\ref{prop:comparisonwired} and the third follows by weak convergence.
\end{proof}

\section*{Index of Notation}

\begin{longtable}{||m{10em} | m{16em} | m{4em}||}
 \hline
 Notation & Description & Section  \\ [0.5ex] 
 \hline\hline
$\beta, \beta_c, \beta_{\mathrm{slab}}$ & Inverse temperature, critical points in the lattice and slab & \ref{sec:background}\\ \hline
$\gamma$ & Loop or $(i-1)$-chain & \ref{sec:background}\\ \hline
$\partial$ & Boundary operator & \ref{subsec:homology}\\ \hline
$\delta$ & Coboundary operator & \ref{subsec:homology}\\ \hline
$\mu^{\xi}_{X,p,q,i}$ & $i$-dimensional PRCM on a complex $X$ with parameters $p,q$ coefficients in $\Q,$ and boundary conditions $\xi$  & \ref{sec:background}\\ \hline
$\mu^{\bullet}_{X,p^*},\mu^{\bullet,\mathbf{w}}_{X,p^*}$ & Dual classical RCM on a complex $X$ with parameter $p^*$ with free or wired boundary conditions & \ref{sec:background}\\ \hline
$\tilde{\mu}_{X,p,G,i}$ & PRCM on $X$ with coefficients in $G$ and parameter $p$  & \ref{sec:background}\\ \hline
$\nu^{\xi}_{X,\beta,q,i-1}$ & $(i-1)$-dimensional PLGT on a complex $X$ with states in $\Z_q,$ inverse temperature $\beta,$ and boundary conditions $\xi$  & \ref{sec:background}\\ \hline
$\rho_Y$ & Chain associated with a subspace $Y$ & \ref{sec:perimeterlaw}\\ \hline
$\tau_{p,q}$ & Surface tension constant & \ref{sec:background}\\ \hline

$\Lambda_n$ & Cube $[-n,n]^d$ & \ref{subsec:comparisoninfinite}\\ \hline
$\Xi_{r,L}$ & Event that r is separated from $\partial r^L$ by a hypersurface of plaquettes & \ref{subsec:interpolate}\\ \hline

$\mathbf{b}_k$ & $k$-th Betti number & \ref{sec:background}\\ \hline

$\Z\paren{q}$ & Multiplicative group of $q$-th roots of unity & \ref{sec:background}\\ \hline

$\mathcal{S}_M$ & Slab of thickness M & \ref{sec:arealaw}\\ \hline





$B_i\paren{X;\;G}, B^i\paren{X;\;G}$& Boundary and coboundary groups & \ref{subsec:homology}\\ \hline
$\mathcal{C}_v$ & Connected component of the vertex $v$ & \ref{sec:perimeterlaw}\\ \hline
$C_i\paren{X,G}, C^i\paren{X,G}$ & Chain and cochain groups & \ref{subsec:homology}\\ \hline
$C_t$ & Event that $s$ is homologous to $\partial t$ in $P$ & \ref{subsec:geometric}\\ \hline

$D_t$ & Event that there are plaquette crossings forming a ``tube'' around $s$ & \ref{subsec:geometric}\\ \hline

$\overline{C}_t$ & $C_t \cap D_t$ & \ref{subsec:geometric} \\ \hline
$E_{u,L}$ & Event that all plaquettes within distance $L$ from a box $u$ are open & \ref{subsec:geometric}\\ \hline
$F_h$ & Event that a ray is separated from an orthogonal plane in the dual bond percolation & \ref{subsec:geometric}\\ \hline
$G$ & Abelian coefficient group  & \ref{sec:background}\\ \hline
$H_i\paren{X,G}, H^i\paren{X,G}$ & Homology and cohomology groups & \ref{subsec:homology}\\ \hline
$I_A$ & Indicator function for the event $A$ & \ref{subsec:couplingfree}\\ \hline
$M$ & Slab thickness & \ref{sec:arealaw}\\ \hline
$M(r), m(r)$ & Maximum and minimum nonzero dimensions of a box $r$  & \ref{sec:background}\\ \hline
$P$ & PRCM plaquette subcomplex  & \ref{sec:background}\\ \hline
$Q$ & Dual RCM bond set  & \ref{sec:background}\\ \hline
$R^{\square}_j\paren{r}$& $(d-1$)-plaquette crossing events for $r$ orthogonal to the $j-$th direction & \ref{sec:sharpnessperimeter}\\ \hline
$V_{\gamma}^{\mathrm{fin}},V_{\gamma}^{\mathrm{inf}}$ & Events for $\gamma$ being a boundary of plaquettes  & \ref{sec:background}\\ \hline
$W_{\gamma}$ & Wilson loop variable & \ref{sec:background}\\ \hline
$X$ & General cell complex & \ref{subsec:homology}\\ \hline
$X^{(k)}$ & $k$-skeleton of $X$ & \ref{subsec:homology}\\ \hline
$\abs{X}$ & Number of top dimensional cells of a complex $X$ & \ref{sec:background}\\ \hline
$Z_i\paren{X,G}, Z^i\paren{X,G}$& Cycle and cocycle groups & \ref{subsec:homology}\\ \hline



$d$ & Dimension of the lattice $\Z^d$ & \ref{sec:intro}\\ \hline
$f$ & PLGT state or $(i-1)$-cochain   & \ref{sec:background}\\ \hline
$l(\gamma_1,\gamma_2)$ & Linking number of $\gamma_1$ and $\gamma_2$ & \ref{sec:dualityvgamma}\\ \hline
$p, p_c, p_{\mathrm{slab}}$ & Percolation parameter, critical points in the lattice and slab  & \ref{sec:background}\\ \hline
$p^*(p) $& Dual parameter to $p$ with respect to $q$  & \ref{sec:background}\\ \hline
$q$ & PRCM/PLGT parameter  & \ref{sec:background}\\ \hline
$r$ & Box in $\Z^d$ without boundary cells  & \ref{sec:background}\\ \hline
$\overline{r}$ & Union of $r$ with its boundary cells & \ref{sec:background}\\ \hline
$s$ & $(d-2)-$face of the support of $\gamma$ & \ref{subsec:geometric}\\ \hline
$t\paren{s,L}$ & Thickening of $s$ by $L$ in the remaining $2$ directions & \ref{subsec:geometric}\\ \hline

\end{longtable}

\newpage
\bibliographystyle{alpha}
\bibliography{bibliography}

\end{document}